\documentclass[11pt]{article}
\usepackage[margin=1in]{geometry}
\usepackage{graphicx,xcolor,cite,mathtools}
\usepackage{amssymb,amsmath,amsthm,mathrsfs,paralist,bm,esint,setspace}
\RequirePackage[colorlinks,citecolor=blue,urlcolor=blue]{hyperref}
\usepackage{cases}

\newcommand{\R}{{\mathbb{R}}}
\newcommand{\E}{\mathrm{E}}
\newcommand{\HH}{\mathcal{H}}
\renewcommand{\P}{\mathrm{P}}
\renewcommand{\d}{\mathrm{d}}
\newcommand{\<}{\langle}
\renewcommand{\>}{\rangle}
\newcommand{\e}{\mathrm{e}}
\newcommand{\Var}{\text{\rm Var}}

\newcommand{\Z}{\mathrm{Z}}

\DeclareMathOperator{\Cov}{\text{\rm Cov}}

\input cyracc.def

\title{Spatial stationarity, ergodicity and CLT for parabolic Anderson model with delta initial condition in dimension $d\geq 1$\thanks{%
	Research supported in part by  NSF grants DMS-1811181 (D.N.) and DMS-1855439 (D.K.), and FNR grant APOGee (R-AGR-3585-10) at University of Luxembourg (F.P.).}}
\author{
		Davar Khoshnevisan\\University of Utah\\\texttt{davar@math.utah.edu}\\
	\and\and
		David Nualart\\University of Kansas\\\texttt{nualart@ku.edu}\\
	\and
		Fei Pu\\University of Luxembourg\\\texttt{fei.pu@uni.lu}
	}
\date{\today}

\begin{document}
\newtheorem{stat}{Statement}[section]
\newtheorem{proposition}[stat]{Proposition}
\newtheorem*{prop}{Proposition}
\newtheorem{corollary}[stat]{Corollary}
\newtheorem{theorem}[stat]{Theorem}
\newtheorem{lemma}[stat]{Lemma}
\theoremstyle{definition}
\newtheorem{definition}[stat]{Definition}
\newtheorem*{cremark}{Remark}
\newtheorem{remark}[stat]{Remark}
\newtheorem*{OP}{Open Problem}
\newtheorem{example}[stat]{Example}
\newtheorem{nota}[stat]{Notation}
\numberwithin{equation}{section}
\maketitle

\begin{abstract}
 Suppose that  $\{u(t\,, x)\}_{t >0, x \in\R^d}$  is the 
solution to a $d$-dimensional parabolic Anderson model with delta initial condition and driven 
by a Gaussian noise that is white in time and has a spatially homogeneous 
covariance given by a nonnegative-definite measure $f$ which satisfies Dalang's condition.
Let $\bm{p}_t(x):=
	(2\pi t)^{-d/2}\exp\{-\|x\|^2/(2t)\}$ denote the standard Gaussian
	heat kernel on $\R^d$. 
We prove that for all $t>0$, the process $U(t):=\{u(t\,, x)/\bm{p}_t(x): x\in \R^d\}$ is stationary  using 
Feynman-Kac's formula,  and is ergodic under the additional condition $\hat{f}\{0\}=0$, where $\hat{f}$ is the Fourier transform of $f$. Moreover, using Malliavin-Stein method, we investigate various central limit theorems for $U(t)$ based on the quantitative analysis of $f$. In particular, when $f$ is given by Riesz kernel, i.e., $f(\d x) = \|x\|^{-\beta}\d x$, we obtain a multiple phase transition for the CLT for $U(t)$ from $\beta\in(0\,,1)$
to $\beta=1$ to $\beta\in(1\,,d\wedge 2)$.

\end{abstract}

\bigskip

 \smallskip

\smallskip
	
\noindent{\it Running head:} Stationarity, ergodicity and CLT for PAM.

{
  \hypersetup{linkcolor=black}
  \tableofcontents
}

\section{Introduction}

Consider the following {\em parabolic Anderson model}:
\begin{equation}\label{PAM}\left[\begin{split}
	&\partial_t u(t\,, x) = \tfrac12\Delta u(t\,, x) + u(t\,, x)\eta(t \,, x)
		&\text{for $(t\,,x)\in(0\,,\infty)\times\R^d$},\\
	&\text{subject to}\qquad u(0)=\delta_0,
\end{split}\right.\end{equation}
where $\eta$  denotes a centered, generalized Gaussian random field with 
\[
	\E[\eta(t\,,x)\eta(s\,,y)] = \delta_0(t-s)f(x-y)
	\qquad[s,t\ge0,\, x,y\in\R^d],
\]
for a non-zero, nonnegative-definite, tempered Borel measure $f$ on $\R^d$. 
As in Walsh\cite{Walsh}, by a ``solution'' to \eqref{PAM}
we mean a solution to the integral equation,
\begin{align}\label{mild}
	u(t\, , x) = \bm{p}_{t}(x) + \int_{(0,t)\times\R^d}
	\bm{p}_{t - s}(x - y)u(s\,, y)\,\eta(\d s\,\d y)
	\qquad\text{a.s.\ for all $t>0$ and $x\in\R^d$,}
\end{align}
where  $\bm{p}_t(x)$ denotes the heat kernel; that is,
\[
	\bm{p}_{t}(x) = (2\pi t)^{-d/2}\e^{-\|x\|^2/(2t)}
	\qquad\text{for $t>0$ and $x\in\R^d$}.
\]
The existence and uniqueness problem for \eqref{PAM} and  of its
variations have been studied extensively by many authors \cite{Dalang1999, CD15, CH19Comparison}. In the present
particular setting, it is easy to see that \eqref{mild} has a 
[unique] predictable solution $u$ if and only if there exists a [unique] predictable
solution $U$ to the following:
\begin{equation}\label{mild:U:pre}
	U(t\,,x) = 1 + \int_{(0,t)\times\R^d}\frac{\bm{p}_{t - s}(x - y) \bm{p}_s(y)}{\bm{p}_t(x)}\,
	U(s\,,y)\,\eta(\d s\,\d y),
\end{equation}
where the pairing $(u\,,U)$ is given by
\begin{equation} \label{U}
	U(t\,,x):= \frac{u(t\,,x)}{\bm{p}_t(x)}
	\qquad\text{for $t>0$ and $x\in\R^d$}.
\end{equation}
It is possible to check directly that
\begin{equation}\label{PPPP}
	\frac{\bm{p}_{t-s}(a)\bm{p}_s(b)}{\bm{p}_t(a+b)} = 
	\bm{p}_{s(t-s)/t}\left( b - \frac st (a+b)\right)
	\quad\text{for all $0<s<t$ and $a,b\in\R^d$}.
\end{equation}
In fact, both sides represent the probability density of $(X_{t-s}\,,X_s)$ where $X$
denotes a Brownian bridge that emenates from zero and is conditioned to reach  $a+b$ at time $t$.

With the preceding in mind, \eqref{mild:U:pre} can be recast as the following
linear integral equation:
\begin{equation}\label{mild:U}
	U(t\,,x) = 1 + \int_{(0,t)\times\R^d}\bm{p}_{s(t-s)/t}\left(
	y - \frac st x \right) U(s\,,y)\,\eta(\d s\,\d y).
\end{equation}
In order to present the basic existence, uniqueness result for \eqref{mild:U},
hence also \eqref{PAM}, let us introduce the following function 
$\Upsilon:(0\,,\infty)\to(0\,,\infty]$:
\begin{equation}\label{Upsilon}
	\Upsilon(\beta) := \frac{1}{(2\pi)^d}\int_{\R^d}\frac{\hat{f}(\d y)}{\beta+\|y\|^2}
	\qquad\text{for all $\beta>0$},
\end{equation}
where $\hat{f}$ denotes the Fourier transform of $f$.

Then we have the following result, which is a variation on a celebrated
theorem of Dalang \cite{Dalang1999} to the linear setting of \eqref{PAM}, started
at initial measure $\delta_0$.

\begin{theorem}\label{th:U}
	Suppose $\Upsilon(\beta)<\infty$ for one,
	hence all, $\beta>0$. Then, the 
	integral equation   \eqref{mild:U}  has a solution $U=\{U(t\,,x)\}_{t>0,x\in\R^d}$
	that is a predictable random field. Moreover, $U$ is the only predictable solution to
	\eqref{mild:U} that satisfies the following for all $\varepsilon\in(0\,,1)$, $t>0$, and $k\ge2$:
	\begin{equation}\label{moment:U}
		\sup_{x\in\R^d}\E\left(|U(t\,,x)|^k\right) \le
		\left(\frac2\varepsilon\right)^k
		\exp\left\{ \frac{tk}{4}\Upsilon^{-1}\left(\frac{1-\varepsilon}{4z_k^2}\right)\right\}:=c_{t,k},
	\end{equation}
	where $z_k$ denotes the optimal constant in the Burkholder--Davis--Gundy inequality
	for continuous $L^k(\Omega)$-martingales.
	Finally, $U(t):=\{U(t\,,x)\}_{x\in\R^d}$ is a stationary random field for every $t>0$,
	and $\lim_{t\to 0}U(t\,,x)=1$ in $L^k(\Omega)$ for every $x\in\R^d$ and for all $k\ge 2$.
\end{theorem}

From now on, we always assume the following. 
\begin{equation}\label{f:finite}
	\Upsilon(1)<\infty
	\quad\text{and}\quad f(\R^d)>0.
\end{equation}
Thanks to \eqref{U} and Theorem \ref{th:U}, 
the finiteness  of $\Upsilon(1)$ implies that \eqref{PAM} has a predictable solution $u$ that uniquely satisfies
that $u(t\,,x) = (1+o(1))\bm{p}_t(x)$ in  { $L^k(\Omega)$ as $t\to0$ for every $x\in\R^d$ and for all $k\ge 2$.}
Furthermore, the strict positivity of the total mass of $f$ is assumed merely to avoid degeneracies in 
\eqref{PAM}.  Before we delve deeper into that topic, however,
let us pause and make a few remarks.


\begin{cremark}
	When $d=1$ and $\eta$ denotes space-time white noise [$f=\delta_0$], the existence 
	and uniqueness of $u$, hence also $U$, are especially
	well known; see for example \cite{CD15}. In that case,
	the stationarity of $U(t)$ was proved first by Amir et al \cite{ACQ11} who used
	the fact that $\{\eta(t\,,x)\}_{t, x}$ has the same
	law as $\{\eta(t\,, x +at)\}_{t,x}$ for all $a\in\R$.
	Our proof of stationarity relies on the Feynman-Kac's formula and 
	works in the present much more general setting. 
\end{cremark}


\begin{cremark}
	It is possible to prove, using ideas from Dalang \cite{Dalang1999}, that the $(d+1)$-parameter random field
	$U$ has a version that is continuous
	in $L^k(\Omega)$ for every $k\ge2$. In turn, this fact and a suitable
	extension of Doob's separability theory (see Doob \cite{Doob}) together show
	that $U$ has a measurable version that solves \eqref{mild:U}. 
	From now on, we always choose this version of $U$ (and denote it also
	by $U$).
\end{cremark}

With \eqref{f:finite} in place and  the above remarks under way, we return to the topic at hand and present the first 
novel contribution of this paper.

\begin{theorem}\label{th:ergodic}
If $\hat{f}\{0\}=0$,  then	$U(t)$ is ergodic for all $t>0$.
\end{theorem}

According to Theorem 1.1 in  Chen et al \cite{CKNP}, the
condition $\hat{f}\{0\}=0$ determines the spatial ergodicity of the solution to
\eqref{PAM}  with flat initial condition. In the case of delta initial condition, 
$\hat{f}\{0\}=0$ also implies the spatial ergodicity of $U$ according 
to Theorem \ref{th:ergodic}. 
For each fixed $N\ge \e$, we introduce the spatial average 
\begin{equation} \label{average}
	\mathcal{S}_{N,t}=\frac {1}{N^d} \int_{[0, N]^d} [ U(t,x) -1]\,\d x.
\end{equation}
Then, condition $\hat{f}\{0\}=0$, Theorem \ref{th:ergodic}, and the ergodic theorem together
imply the following law of large numbers: For every $t>0$,
\[
	\lim_{N\to\infty} \mathcal{S}_{N,t} = 0\qquad\text{a.s.\ and in $L^k(\Omega)$ for all $k\ge2$}.
\]

The main result of this paper is a corresponding central limit theorem (CLT),  which turns out
to hold in the strongest possible sense of convergence in total variation.
Let $\Z$ denote the standard Gaussian random variable, and 
recall that the total variation distance between random variables $X$ and $Y$ on $\R$ is defined as
\[
	d_{\rm TV}  (X\, ,Y)= \sup| \P(X\in B)- \P(Y\in B)|,
\]
where  the supremum is take over all Borel subsets $B$ of $\R$.

Recall that the condition $f(\R^d)<\infty$ implies a CLT for 
the spatial averages of the solution to \eqref{PAM} with flat/constant initial data
\cite[Theorem 1.1]{CKNP_b}. 
The situation is much more involved in the present setting where the initial condition is a delta mass.
In this setting, we first must analyze the asymptotic behavior of 
${\rm Var}(\mathcal{S}_{N,t})$ under different assumptions on the covariance measure $f$. 
In the case of a flat initial condition, the condition  $f(\R^d)<\infty$ implies that
the variance of the spatial average of the solution is of 
the order $N^{-d}$ as $N\to\infty$; see \cite[Proposition 5.2]{CKNP_b}. By contrast,
we will see in Section \ref{analyzevar} that, in the present setting,
the normalization of ${\rm Var}(\mathcal{S}_{N,t})$ depends on the 
detailed structure of the covariance measure $f$, as well as on the spatial dimension $d$.
Moreover, in order to prove the CLT, we appeal to the Malliavin-Stein method 
(see Proposition \ref{propTV} below), from which we will deduce 
how the covariance measure $f$ characterizes the CLT for the spatial average of $U(t)$
in various ways. In the case of a flat initial condition, it has been proved in \cite[Theorem 2.4]{CKNP_d} that the convergence rate for CLT in terms of total variation is $N^{-d/2}$, while for delta initial condition, 
the convergence rate for CLT is determined not only by spatial dimension $d$ but also by the 
behavior of $f$.

We start by introducing 
the following quantity associated with $f$:
\begin{align}\label{R(f)}
	\mathcal{R}(f):= \frac{1}{\pi^d}\int_0^{\infty}\d s 
	\int_{\R^d}\hat{f}(\d z)\prod_{j=1}^{d}\frac{1-\cos(sz_j)}{(sz_j)^2}.
\end{align}
The following theorem states that  the finiteness of $\mathcal{R}(f)$ ensures the CLT for the spatial average of $U(t)$ 
and the convergence rate is $N^{-1/2}$ regardless of the spatial dimension $d$.

\begin{theorem}\label{TVD3}
        If  $\mathcal{R}(f) <\infty$, then for all fixed $t>0$ 
        there exists $C=C(t)>0$ such that
        \[
	 	d_{\rm TV} \left(  \frac{ \mathcal{S}_{N,t}}{ \sqrt{{\rm Var}(\mathcal{S}_{N,t})}} 
		\, , \Z\right)  \le \frac C {\sqrt{N}}
		\qquad\text{for every $N\ge\e$}.        
	\]
\end{theorem}
The asymptotic behavior of ${\rm Var}(\mathcal{S}_{N,t})$ will be discussed in detail
in Theorem \ref{th:AV:d>1} below. It follows from that analysis and from Theorem  \ref{TVD3} that, if
$\mathcal{R}(f) <\infty$, then  
\[
        \sqrt N \mathcal{S}_{N,t} = \frac{1}{N^{d-(1/2)}} \int_{[0, N]^d} [ U(t\,,x) -1]\,\d x 
	\xrightarrow{\text{\rm d}\,}  {\rm N}(0\,,t\mathcal{R}(f))      
	\qquad \text{ as $N\to\infty$,}
\]
where ``$\xrightarrow{\text{\rm d}\,}$'' denotes convergence in distribution.

We will see in Lemma \ref{lem:f:1} below that
$\mathcal{R}(f)<\infty$ only if $d\geq 2$. Thus, the preceding CLT has no content in dimension one.
When $d=1$, we are able to derive a CLT under the additional constraint $f(\R)<\infty$.
According to Theorem 1.1 in  Chen et al \cite{CKNP_b}, the finiteness condition $f(\R)<\infty$ 
implies a CLT for the solution to \eqref{PAM} with flat initial condition. The same holds in the
present setting of delta initial condition, except the rate is different (and so are many
of the underlying arguments).

\begin{theorem}[$d=1$]\label{TVD1}
	If $f(\R)<\infty$ and $d=1$, then for all fixed $t>0$ there exists
	$C=C(t)>0$ such that 
	\[
		d_{\rm TV} \left(  \frac{ \mathcal{S}_{N,t}}{ \sqrt{{\rm Var}(\mathcal{S}_{N,t})}} \, , \Z\right)
		\leq  C \sqrt{\frac{\log N}{N}}\qquad\text{for all $N\ge\e$}.
	\]  
\end{theorem}

In particular, Theorem \ref{TVD1} and Theorem \ref{th:AV:d=1} below 
together imply that, if $d=1$ and $f=a\delta_0$ for some $a>0$, then
 \[
	\sqrt{\frac{N}{\log N}}\,\mathcal{S}_{N,t}=
	\frac{1}{\sqrt{N\log N}} \int_0^N [ U(t,x) -1]\,\d x
	\xrightarrow{\text{\rm d}\,}  {\rm N}(0\,, 2tf(\R))  = {\rm N}(0\,, 2ta)      
         \quad \text{as $N\to\infty$.}
\]
On the other hand, if the measure $f$ is finite, as well as a
Rajchman measure\footnote{We recall that a finite measure $f$ is \emph{Rajchman} if
its Fourier transform 
$\R^d\ni x\mapsto \hat{f}(x) := \int_{\R^d} \e^{ix\cdot y}\,f(\d y)$ 
vanishes at infinity; that is, $\lim_{\|x\|\to\infty}\hat{f}(x)=0$.
Lyons \cite{Lyons} discusses a survey of the rich subject of Rajchman measures. 
Note that, in the present setting, $\hat{f}:\R^d\to\R$
is a non-negative, non-negative definite, 
uniformly bounded, and continuous function.}, then
\[
	\sqrt{\frac{N}{\log N}}\,\mathcal{S}_{N,t}=
	\frac{1}{\sqrt{N\log N}} \int_0^N [ U(t\,,x) -1]\,\d x
	\xrightarrow{\text{\rm d}\,}  {\rm N}(0\,, tf(\R))      
         \quad \text{as $N\to\infty$.}
\]

The above results give a more or less comprehensive idea of the CLT for $U(t)$
when $\mathcal{R}(f)<\infty$, especially when the measure $f$ is in addition finite.
By contrast with this case, there does not seem to be a canonical description of a CLT
when $\mathcal{R}(f)=\infty$. This condition occurs for any ambient dimension $d$, for example when
$f$ is given by a Riesz kernel; see Remark \ref{Rieszinfty} below. 
In the following, we will present the CLT specifically in the case that $f$ is given by Riesz kernel
that satisfies Dalang's condition, $\Upsilon(1)<\infty$;
that is,  when $f(\d x)= \|x\|^{-\beta}\,\d x$ where $0<\beta< 2\wedge d$. 
In contrast with what happens in the case that the initial condition is flat 
(see Huang et al \cite{HNVZ2019}),  the CLT for $U$ undergoes a multiple phase transition from $\beta\in(0\,,1)$
to $\beta=1$ to $\beta\in(1\,,d\wedge 2)$.

\begin{theorem}\label{TVD2}
	If $f(\d x)= \|x\|^{-\beta}\d x$ for some $\beta\in(0\,,d\wedge 2)$, then for all fixed $t>0$
	there exists $C=C(t)>0$ such that for all $N \geq \e$,
	\[
		d_{\rm TV} \left(  \frac{ \mathcal{S}_{N,t}}{ \sqrt{{\rm Var}(\mathcal{S}_{N,t})}} \, , \Z\right) \le
		\begin{cases}
			CN^{-\beta/2}& \text{if $\beta\in(0\,,1)$},\\
			C\sqrt{\log(N)/N} &\text{if $\beta=1$},\\
			CN^{-(2-\beta)/2}& \text{if $\beta\in(1\,,  2)$}. 
		\end{cases}
	\]
\end{theorem}
As a consequence of Theorem \ref{TVD2} and Theorem \ref{Riesz} below, we obtain 
the following CLTs: \\

\noindent \textbf{(A)} If $0<\beta<1$,  then
\[
        N^{\beta/2}\mathcal{S}_{N,t} = 
        \frac{1}{N^{d-(\beta/2)}} \int_{[0, N]^d} [ U(t\,,x) -1]\,\d x
        \xrightarrow{\text{\rm d}\,}  {\rm N}(0\,, t\sigma_{0, \beta, d})      
         \quad \text{as $N\to\infty$;}
\]
\noindent\textbf{(B)} If $\beta=1$,  then
\[
        \sqrt{\frac{N}{\log N}}\,\mathcal{S}_{N,t}=
        \frac{1}{N^{d-(1/2)}\sqrt{\log N}} \int_{[0, N]^d} [ U(t\,,x) -1]\,\d x
        \xrightarrow{\text{\rm d}\,}  {\rm N}(0\,, t\sigma_{1, \beta, d})      
	\quad \text{as $N\to\infty$;\ and}
\]
\noindent\textbf{(C)} If $1<\beta<2\wedge d$,  then
\[
        N^{1-(\beta/2)}\mathcal{S}_{N,t}=
        \frac{1}{N^{d-1+(\beta/2)}} 
        \int_{[0, N]^d} [ U(t\,,x) -1]\,\d x
        \xrightarrow{\text{\rm d}\,}  {\rm N}(0\,, t^{2-\beta}\sigma_{2, \beta, d})      
	\quad \text{as $N\to\infty$;}
\]
where  $\sigma_{0, \beta, d}$,
$\sigma_{1, \beta, d}$, and $\sigma_{2, \beta, d}$ are non-degenerate 
and defined explicitly in Theorem \ref{Riesz}.

The logarithmic correction that appears in Theorems \ref{TVD1} and \ref{TVD2} [$\beta=1$]
is related to the transition functions of the Brownian bridge;
see \eqref{PPPP}. Indeed, the conditional probability density $\bm{p}_{s(t-s)/t}(sx/t)$ 
becomes $t/s$ after a change of variable in $x$. The resulting 
singularity at $s=0$ ultimately give rises to the $\log N$ factor in Theorems \ref{TVD1} and \ref{TVD2}.

\begin{cremark}
	The convergence rates for the total variation distance in Theorems \ref{TVD3}, \ref{TVD1},
	and \ref{TVD2} are natural. Indeed, one can observe that in each case the convergence rate 
	for the total variation distance is of the same order as $\sqrt{\Var(\mathcal{S}_{N,t})}$ 
	as $N\to\infty$; see Theorems \ref{th:AV:d>1},  \ref{th:AV:d=1} and \ref{Riesz}. 
	A similar phenomenon can be observed in the context of spatial CLT for other related SPDEs
	\cite{HNV2018, HNVZ2019, DNZ2018, GNZ20, NZ2020, CKNP_c, CKNP_d}. 
	See \cite{NSZ20} for recent advances on the parabolic Anderson model driven by a
	Gaussian noise that is colored in both its space and time variables.
\end{cremark}

\begin{cremark}
One can follow the method in \cite{CKNP_c} to prove the functional CLT in time corresponding to the CLTs below Theorems  \ref{TVD3}, \ref{TVD1} and \ref{TVD2} respectively.  For instance, one can use the argument in \cite[Proposition 4.1]{CKNP_c} to compute the covariance of the limit Gaussian process and then prove the convergence of finite dimensional distributions and tightness. We leave these for interested readers. 
\end{cremark}

The organization  of this paper is as follows. We establish the well-posedness and spatial stationarity for the solution to \eqref{mild:U} in Theorem \ref{th:U} in Section \ref{wellpose}. The ergodicity property in Theorem \ref{th:ergodic} is proved in Section \ref{ergodicity}. Section \ref{analyzevar} is devoted to analyzing the asymptotic behavior of the variance of spatial average. Moreover, we present the estimates on total variation distance in Theorems \ref{TVD3}, \ref{TVD1} and \ref{TVD2} in Section \ref{Sec:TVD}. And the last section is an Appendix that contains a few technical lemmas that are used throughout the paper.

Let us conclude the Introduction by setting forth some notation that will be used
throughout. 
We write ``$g_1(x)\lesssim g_2(x)$ for all $x\in X$'' when there exists a real number
$L$ such that $g_1(x)\le Lg_2(x)$ for all $x\in X$.
Alternatively, we might write
``$g_2(x)\gtrsim g_1(x)$
for all $x\in X$.'' By ``$g_1(x)\asymp g_2(x)$ for all $x\in X$'' we mean that
$g_1(x)\lesssim g_2(x)$ for all $x\in X$ and $g_2(x)\lesssim g_1(x)$ for all $x\in X$.
Finally,
``$g_1(x)\propto g_2(x)$ for all $x\in X$'' means that there exists a real number $L$
such that $g_1(x)=L g_2(x)$ for all $x\in X$.
For every $Z\in L^k(\Omega)$, we write $\|Z\|_k$ instead of the more cumbersome
$\|Z\|_{L^k(\Omega)}$. 

\section{Preliminaries}

\subsection{The BDG inequality}

Let us  collect a few facts about the optimal constants $\{z_k\}_{k\ge2}$
of the Burkholder--Davis--Gundy [BDG] inequality.

First, recall from the BDG inequality that for every continuous $L^2(\Omega)$-martingale
$\{M_t\}_{t\ge0}$,
\[
	\E\left(|M_t|^k\right) \le z_k^k\E\left( \<M\>_t^{k/2}\right)
	\qquad\text{for all $t\ge0$ and $k\ge2$}.
\]
Davis \cite{Davis1976} has shown that every $z_k$ is the largest positive root of a certain special function.
In particular, that $z_k$ is the largest positive root of the monic Hermite polynomial
${\it He}_k$ when $k$ is an even integer. These remarks and the appendix of Carlen and Kree \cite{CarlenKree1991}
together imply the following:
\begin{equation}\label{z:UB}
	z_2=1,\quad
	z_4=\sqrt{3+\sqrt{6}}\approx 2.334,\quad\text{and}\quad
	\sup_{k\ge2}\frac{z_k}{\sqrt k}=\lim_{k\to\infty}\frac{z_k}{\sqrt k}=2.
\end{equation}
Moreover, the the special case where the martingale $M$ is Brownian motion
shows us that
\begin{equation}\label{z:LB}
	z_k \ge \|{\rm N}(0\,,1)\|_k = \sqrt{2}\left[ \frac{1}{\sqrt\pi}
	\Gamma\left(\frac{k+1}{2}\right)\right]^{1/k}
	\qquad\text{for all $k\ge2$}.
\end{equation}
Therefore, we learn from the Stirling formula that $z_k$ is bounded from above and from below
by non-degenerate multiples of $\sqrt{k}$, uniformly
for all $k\ge2$.

\subsection{The Clark-Ocone formula}
Define $\mathcal{H}_0$ to be the reproducing kernel Hilbert space, spanned by all real-valued functions on $\R^d$,
that corresponds to the inner product 
$\<\phi\,,\psi\>_{\mathcal{H}_0}:=\<\phi\,,\psi*f\>_{L^2(\R^d)}$, and set
$\mathcal{H} := L^2(\R_+\times\mathcal{H}_0)$.
The Gaussian family $\{ W(h)\}_{h \in \mathcal{H}}$ formed by the Wiener integrals
\[
	W(h)= \int_{\R_+\times \R^d}  h(s\,,x)\, \eta(\d s\,\d x)
	\qquad[h\in\mathcal{H}]
\]
defines an  {\it isonormal Gaussian process} on the Hilbert space $\mathcal{H}$.
In this framework we can develop the Malliavin calculus (see, for instance,  \cite{Nualart}).
We denote by $D$ the Malliavin derivative operator, and by $\delta$ the corresponding divergence operator
whose domain in $L^2(\Omega)$ is denoted by $\text{\rm Dom}[\delta]$.

Let $\{\mathcal{F}_s\}_{s \geq 0}$ denote the filtration generated by the infinite dimensional
white noise $t\mapsto\eta(t)$; that is, $\mathcal{F}_t$ is the filtration generated by all 
Wiener integrals of the form $\int_{(0,t)\times\R^d}\phi\,\d\eta$ as $\phi$ ranges over all
test functions of rapid decrease [which are easily seen to be dense in $\mathcal{H}$]. 
A basic idea used in this paper
is the Clark-Ocone formula  (see \cite[Proposition 6.3]{CKNP}),
\begin{equation} \label{Clark-Ocone}
	F= \E [F]  + \int_{\R_+\times\R^d}
	\E\left[D_{s,y} F \mid \mathcal{F}_s\right] \eta(\d s \, \d z),
\end{equation}
valid a.s.\ for every random variable $F$ in the Gaussian
Sobolev space $\mathbb{D}^{1,2}$. Using Jensen's inequality for conditional expectation, this equality leads 
immediately to the following Poincar\'e-type inequality:
\begin{equation} \label{Poincare:Cov}
	|\Cov(F\,, G)| \le \int_0^{\infty}\d s \int_{\R^d}\d y\int_{\R^d}f(\d y')\
	\left\| D_{s,y} F  \right\|_2
	\left\| D_{s,y'+y}G \right\|_2 
\end{equation}
for $F,G\in\mathbb{D}^{1,2}$ provided that $DF$ and $DG$ are real-valued random variables.
\subsection{The Malliavin--Stein method}
	
Recall that the total variation distance between two Borel probability
measures $\mu$ and $\nu$ on $\R$ is defined as
\[
	d_{\rm TV}  (\mu\, ,\nu)= \sup | \mu(B)- \nu(B)|,
\]
where  the supremum is taken over all Borel subsets $B$ of $\R$. 
We might abuse notation and write
$d_{\rm TV}(F\,,G)$, $d_{\rm TV}(F\,,\nu)$, or $d_{\rm TV}(\mu\,,G)$
instead of $d_{\rm TV}(\mu\,,\nu)$ whenever the laws of $F$ and $G$
are respectively $\mu$ and $\nu$. 

A combination of Malliavin calculus and Stein's method for normal approximations leads to the following
bound on the total variation distance (see  \cite[Theorem 8.2.1]{NN}):

\begin{proposition}  \label{propTV}
	Suppose $F\in \mathbb{D}^{1,2}$ satisfies $\E[F^2]=1$ and  $F=\delta(v)$ for 
	some element $v$ in the domain in  $L^2(\Omega)$  of the divergence operator $\delta$.  Then,
	\begin{equation} \label{SM1}
		d_{\rm TV} (F\,,  {\rm N}(0\,,1)) \le  
		2  \sqrt{ {\rm Var} \left ( \langle DF \,, v \rangle_{\mathcal{H}} \right) }.
	\end{equation}
\end{proposition}
	 


\section{Existence, uniqueness and stationarity: proof of Theorem \ref{th:U}}\label{wellpose}

The proof of Theorem \ref{th:U} follows a route that is nowadays standard.
Therefore, we  sketch the bulk argument, enough to make sure that the numerology of
\eqref{moment:U} is explained in sufficient detail.
Also, the proof does require one technical lemma that we state and prove next. The following identity will be used several times later on:
\begin{align}\label{Fourier}
     \left( \bm{p}_r*f\right)(x)
     = \frac{1}{(2\pi)^d}\int_{\R^d}\e^{-r\|y\|^2/2}\e^{ix\cdot y}\hat{f}(\d y), \quad \text{for all $r>0$ and $x\in \R^d$}.
\end{align}
Since  $\bm{p}_r$ is a test function of rapid decrease for every $r>0$, the above identity follows from the very definition of $\hat{f}$.

Recall the function $\Upsilon$ defined in \eqref{Upsilon}.
\begin{lemma}\label{lem:p*f}
	$\int_0^t\exp\{-\beta \{s\wedge(t-s)\}\}(\bm{p}_{2s(t-s)/t}*f)(0)\,
	\d s\le 4\Upsilon(2\beta)$ for every $t,\beta>0$.
\end{lemma}

\begin{proof}
	We  apply the identity \eqref{Fourier} with $r=2s(t-s)/t$
	in order to find that
	\begin{align*}
		\left(\bm{p}_{2s(t-s)/t}*f\right)(0)
		=   \frac{1}{(2\pi)^d}\int_{\R^d} \e^{
		-s(t-s)\|y\|^2/t}\hat{f}(\d y)\,
		\le\frac{1}{(2\pi)^d}\int_{\R^d}\exp\left(
		-\frac{s\wedge(t-s)}{2}\|y\|^2\right)\hat{f}(\d y),
	\end{align*}
	using the elementary fact that $s(t-s)/t\ge \frac12[s\wedge(t-s)]$.
	Integration and symmetry together imply
	\[
		\int_0^t \e^{-\beta\{s\wedge(t-s)\}}\left(\bm{p}_{2s(t-s)/t}*f\right)(0)\,\d s
		\le \frac{2}{(2\pi)^d}\int_0^{t/2}\e^{-\beta s}\,\d s\int_{\R^d}\
		\e^{-s\|y\|^2/2}\hat{f}(\d y).
	\]
	The  bound $\int_0^{t/2}(\,\cdots)\le\int_0^\infty(\,\cdots)$ yields the lemma.
\end{proof}

With Lemma \ref{lem:p*f} under way, we can start the proof of Theorem \ref{th:U}.

\begin{proof}[Proof of Theorem \ref{th:U} (part 1): Existence and uniqueness.]
	Throughout the proof, define
	\begin{equation}\label{beta:eps:k}
		\beta_{\varepsilon,k}:= \frac12\Upsilon^{-1}\left(\frac{1-\varepsilon}{4z_k^2}\right).
	\end{equation}
	We begin by proving existence and uniqueness.
	
	The proof of existence and uniqueness works by Picard iteration, as   is customary,
	and uses ideas from Foondun and Khoshnevisan \cite{FoondunKhoshnevisan2013} in order to establish the 
	moment bound \eqref{moment:U} and uniqueness. 
	
	Define for all $t>0$ and $x\in\R^d$,
	$U_0(t\,,x) :=1$ and
	\begin{equation}\label{Picard:U_n}
		U_{n+1}(t\,,x) = 1 + \int_{(0,t)\times\R^d}\bm{p}_{s(t-s)/t}\left(
		y - \frac st x \right) U_n(s\,,y)\,\eta(\d s\,\d y),
	\end{equation}
	valid for every $n\in\mathbb{Z}_+$. Define, for all $t>0$, $n\in\mathbb{N}$, and $x\in\R^d$,
	\[
		\mathcal{D}_n(t\,,x) := U_n(t\,,x) - U_{n-1}(t\,,x)
		\quad\text{and}\quad
		\mathcal{E}_n(t) := \sup_{a\in\R^d}\|\mathcal{D}_n(t\,,a)\|_k^2.
	\]
	We first observe that, because of the semigroup property of the heat kernel,
	\begin{align*}
		\|\mathcal{D}_1(t\,,x)\|_k^2 & = \left\| \int_{(0,t)\times\R^d}\bm{p}_{s(t-s)/t}\left(
			y - \frac st x \right)\eta(\d s\,\d y) \right\|_k^2\\
		&\le z_k^2 \int_0^t\d s\int_{\R^d}\d y
			\int_{\R^d}f(\d y')\
			\bm{p}_{s(t-s)/t}(y)\bm{p}_{s(t-s)/t}(y'+y)\\
		&=z_k^2 \int_0^t\d s\int_{\R^d}f(\d w)\
			\bm{p}_{2s(t-s)/t}(w).
	\end{align*}
	Therefore, we may appeal Lemma \ref{lem:p*f} to find that for every $\beta>0$,
	\begin{equation}\label{D_1}
		\|\mathcal{D}_1(t\,,x)\|_k^2
		\le z_k^2\e^{\beta t}
		\int_0^t\e^{-\beta \{s\wedge(t-s)\}}\d s\int_{\R^d}f(\d w)\
		\bm{p}_{2s(t-s)/t}(w)\\
		\le 4z_k^2\e^{\beta t}\Upsilon(2\beta).
	\end{equation}
	Next, we might observe that
	\begin{align*}
		&\left\| \mathcal{D}_{n+1} (t\,,x) \right\|_k^2\\
		&= \left\| \int_{(0,t)\times\R^d}
				\bm{p}_{s(t-s)/t}\left(y - \frac st x \right) 
				\mathcal{D}_n(s\,,y)\,\eta(\d s\,\d y)\right\|_k^2\\
		&\le z_k^2\int_0^t\d s\int_{\R^d}\d y
			\int_{\R^d} f(\d w)\
			\bm{p}_{s(t-s)/t}\left(y - \frac st x \right) 
			\bm{p}_{s(t-s)/t}\left(w+y - \frac st x \right) 
			\left\| \mathcal{D}_n(s\,,y)\mathcal{D}_n(s\,,w+y)
			\right\|_{k/2}\\
		&\le z_k^2 \int_0^t[\mathcal{E}_n(s)]^2\,\d s\int_{\R^d}\d y
			\int_{\R^d} f(\d w)\
			\bm{p}_{s(t-s)/t}\left(y - \frac st x \right) 
			\bm{p}_{s(t-s)/t}\left(w+y - \frac st x \right)\\
		&= z_k^2\int_0^t\mathcal{E}_n(s)\,\d s
			\int_{\R^d} f(\d w)\
			\bm{p}_{2s(t-s)/t}(w).
	\end{align*}
	Since the right-hand side does not depend on $x$, we may optimize to find that
	\begin{align*}
		\e^{-\beta t}\mathcal{E}_{n+1}(t) &\le z_k^2\e^{-\beta t}\int_0^t\mathcal{E}_n(s)\,\d s
			\int_{\R^d}f(\d w)\, \bm{p}_{2s(t-s)/t}(w)\\
		&= z_k^2\int_0^t \e^{-\beta\{s\vee(t-s)\}}\mathcal{E}_n(s) \e^{-\beta\{s\wedge(t-s)\}}\,\d s
			\int_{\R^d}f(\d w)\, \bm{p}_{2s(t-s)/t}(w)\\
		&\le z_k^2\int_0^t \e^{-\beta s}\mathcal{E}_n(s) \e^{-\beta\{s\wedge(t-s)\}}\,\d s
			\int_{\R^d}f(\d w)\, \bm{p}_{2s(t-s)/t}(w).
	\end{align*}
	In particular, set
	\[
		\mathcal{F}_n(t\,,\beta) := \sup_{s\in(0,t]} \left[\e^{-\beta s}\mathcal{E}_n(s)\right]
		\quad\text{for all $n\in\mathbb{N}$ and $t,\beta>0$},
	\]
	in order to deduce from Lemma \ref{lem:p*f} that
	$\mathcal{F}_{n+1}(t\,,\beta) \le
	4z_k^2\Upsilon(2\beta)\mathcal{F}_n(t\,,\beta)$. 
	Plug in $\beta=\beta_{\varepsilon,k}$, defined in \eqref{beta:eps:k}, to find inductively that
	\[
		\mathcal{F}_{n+1}(t\,,\beta_{\varepsilon,k}) \le 
		(1-\varepsilon)\mathcal{F}_n(t\,,\beta_{\varepsilon,k})
		\le\cdots\le(1-\varepsilon)^n\mathcal{F}_1(t\,,\beta_{\varepsilon,k}).
	\]
	Now, we can read off from \eqref{D_1} that
	$\mathcal{F}_1(t\,,\beta_{\varepsilon,k}) =
	\sup_{s\in(0,t]}[\exp\{-\beta_{\varepsilon,k}s\}\mathcal{E}_1(s)]
	\le 4z_k^2\Upsilon(2\beta_{\varepsilon,k}) = 1-\varepsilon.$
	This yields $\mathcal{F}_{n+1}(t\,,\beta_{\varepsilon,k})\le(1-\varepsilon)^{n+1}$, and hence
	\[
		\sup_{x\in\R^d}\left\| U_{n+1}(t\,,x) - U_n(t\,,x) \right\|_k^2\le 
		(1-\varepsilon)^{n+1}\e^{\beta_{\varepsilon,k}t}
		\qquad\text{for all $t>0$ and $n\in\mathbb{Z}_+$}.
	\]
	At this point, standard arguments imply that $U(t\,,x) := \lim_{n\to\infty}U_n(t\,,x)$
	exists in $L^k(\Omega)$ for every $k\ge2$ and solves \eqref{PAM}. Moreover,
	\[
		\|U(t\,,x)\|_k \le \|U_0(t\,,x)\|_k+\sum_{n=0}^\infty
		\| U_{n+1}(t\,,x) - U_n(t\,,x)\|_k\le 1 + \e^{\beta_{\varepsilon,k}t/2}
		\sum_{m=1}^\infty(1-\varepsilon)^{m/2}\le
		\frac{\exp\left(\beta_{\varepsilon,k}t/2\right)}{1-\sqrt{1-\varepsilon}}.
	\]
	Since $1-\sqrt{1-\varepsilon}\ge\varepsilon/2$, this proves
	\eqref{moment:U}. 
\end{proof}

\begin{proof}[Proof of Theorem \ref{th:U} (part 2): Stationarity.]
	For every $\varepsilon>0$ we define a new Gaussian noise
	$\eta^\varepsilon$ via its Wiener integrals,
	\[
		\int_{(0,t)\times\R^d} \varphi(y)\,\eta^\varepsilon(\d s\,\d y)
		:= \int_{(0,t)\times\R^d} \left( \varphi*\bm{p}_\varepsilon\right)(y)\,\eta(\d s\,\d y)
		\qquad\text{for all $t>0$ and $\varphi\in\HH_0$}.
	\]
	Because of the semigroup property of the heat kernel, 
	$\eta^\varepsilon$ is a generalized Gaussian random field with 
	\[
		\Cov\left[ \eta^\varepsilon(t\,,x)~,~\eta^\varepsilon(s\,,y) \right]
		=\delta_0(t-s)  f_\varepsilon(x-y),
		\quad\text{where } f_{\varepsilon} := \bm{p}_{2\varepsilon}*f.
	\]
	As is customary in distribution theory, the rapidly-decreasing test function
	$f_\varepsilon$ is identified with 
	a positive-definite tempered measure [also denoted by $f_\varepsilon$] that,
	among many other things, satisfies
	\eqref{f:finite}. In fact, the total mass of the measure $f_\varepsilon$
	is merely the total integral of the function $f_\varepsilon$, which is $f(\R^d)$.
	Let $\Upsilon_\varepsilon$ be defined as in \eqref{Upsilon}, but with $f$ replaced by $f_\varepsilon$,
	in order to see immediately that $\Upsilon_\varepsilon \le \Upsilon$ pointwise.
	Thus, the already-proved portion of Theorem \ref{th:U} applies to show that the stochastic PDE
	\begin{align*}
		&\partial_t u^\varepsilon = \tfrac12\Delta u^\varepsilon + u^\varepsilon\eta^\varepsilon
			\qquad\text{on $(0\,,\infty)\times\R^d$},\\
		&\text{subject to }u^\varepsilon(0)=\delta_0\hskip.55cm\text{on $\R^d$},
	\end{align*}
	has a predictable random-field solution $u^\varepsilon$ that is unique subject to  
	\[
		\sup_{(t,x,\varepsilon)\in(0,T)\times\R^d\times(0,\infty)}
		\|u^\varepsilon(t\,,x)/\bm{p}_t(x)\|_k<\infty\qquad\text{for every $T>0$ and $k\ge2$}.
	\]
	Let us expand on this a little
	as follows: For every $z\in\R^d$, consider the SPDE
	\[\left[\begin{split}
		&\partial_t u^\varepsilon(t\,,x;\,z) = \tfrac12\Delta_x u^\varepsilon(t\,,x\,;z) + 
			u^\varepsilon(t\,,x\,;z)\eta^\varepsilon(t\,,x)
			\qquad\text{on $(0\,,\infty)\times\R^d$},\\
		&\text{subject to }u^\varepsilon(0\,,\bullet\,;z)=\delta_z(\bullet)\hskip.55cm\text{on $\R^d$}.
	\end{split}\right.\]
	Then we can apply the same argument that was used in the already-proved portion of Theorem
	\ref{th:U} in order to establish the existence of a random-field solution $u^\varepsilon(\bullet\,,\bullet\,;z)$
	to the preceding, one for every $z\in\R^d$, that is unique among all that satisfy
	\[
		L_{T,k} := \sup_{(t,x,z,\varepsilon)\in(0,T)\times\R^d\times\R^d\times(0,\infty)}
		\left\| \frac{u^\varepsilon(t\,,x\,;z)}{\bm{p}_t(x-z)}\right\|_k
		<\infty\qquad\text{for every $T>0$ and $k\ge2$}.
	\]
	We remark that $u^\varepsilon(t\,,x\,;0)=u^\varepsilon(t\,,x)$ for all $t,\varepsilon>0$
	and $x\in\R^d$. 
	
	Let $U^\varepsilon(t\,,x):=u^\varepsilon(t\,,x)/\bm{p}_t(x)$
	and $U^\varepsilon(t\,,x\,;z):=u^\varepsilon(t\,,x\,;z)/\bm{p}_t(x-z)$ for all
	$t>0$, $\varepsilon\in(0\,,1)$, and $x,z\in\R^d$; confer with \eqref{U}.
	The method of Dalang \cite{Dalang1999} can be used to show also that
	$(t\,,x\,,z)\mapsto U^\varepsilon(t\,,x\,;z)$ -- hence also
	$(t\,,x\,,z)\mapsto u^\varepsilon(t\,,x\,;z)$ --
	is continuous in $L^k(\Omega)$ for
	every $k\ge2$ and $\varepsilon>0$. We skip the details and mention only that, 
	in particular, $u^\varepsilon$ and $U^\varepsilon$ both have Lebesgue-measurable versions for
	every $\varepsilon>0$, which we always use.
	
	Choose and fix an arbitrary non-random function $v_0\in L^\infty(\R^d)$ to see 
	from linearity that
	\begin{equation}\label{v_0}
		v^\varepsilon(t\,,x) := \int_{\R^d} u^\varepsilon(t\,,x\,;z)v_0(z)\,\d z
		\qquad[t>0,\, x\in\R^d]
	\end{equation}
	is the unique predictable solution to the SPDE
	\[\left[\begin{split}
		&\partial_t v^\varepsilon(t\,,x)= \tfrac12\Delta v^\varepsilon(t\,,x) + 
			v^\varepsilon(t\,,x)\eta^\varepsilon(t\,,x)
			\qquad\text{for $(t\,,x)\in(0\,,\infty)\times\R^d$},\\
		&\text{subject to }v^\varepsilon(0)=v_0\hskip.54cm\text{on $\R^d$},
	\end{split}\right.\]
	that satisfies
	\[
		L_{T,k,\varepsilon}' := \sup_{(t,x)\in(0,T)\times\R^d}\| v^\varepsilon(t\,,x)\|_k<\infty
		\qquad\text{for every $T,\varepsilon>0$ and $k\ge2$}.
	\]
	Recall that $v^\varepsilon$ has the following mild formulation:
	\begin{align*}
		v^\varepsilon(t\,,x) &= (\bm{p}_t*v_0)(x) + \int_{(0,t)\times\R^d}
			\bm{p}_{t-s}(y-x) v^\varepsilon(s\,,y)\,\eta^\varepsilon(\d s\,\d y)\\
		&=(\bm{p}_t*v_0)(x)  + \int_{(0,t)\times\R^d}\left(\int_{\R^d}
			\bm{p}_{t-s}(y-x) v^\varepsilon(s\,,y)\bm{p}_\varepsilon(y-z)\,\d y\right) \eta(\d s\,\d z),
	\end{align*}
	thanks to a stochastic Fubini argument, which we skip.
	The spatial correlation function $f_\varepsilon$
	of $\eta^\varepsilon$ clearly is in $\mathscr{S}(\R^d)$ and hence is bounded; in fact, 
	\[
		f_\varepsilon(x) = \left(\bm{p}_{2\varepsilon}*f\right)(x)\le
		\frac{1}{(2\pi)^d}\int_{\R^d}\e^{-\varepsilon\|y\|^2}\hat{f}(\d y)<\infty
		\qquad\text{for all $\varepsilon\in(0\,,1)$ and $x\in\R^d$; see \eqref{Fourier}.}
	\]
	Let $B$ denote a standard Brownian motion that is independent of $\eta$,
	and let $\E_B$ and $\E_\eta$ denote, respectively, the conditional
	expectation operators given $B$ and $\eta$. According to 
	general theory (see Hu and Nualart \cite[Proposition 5.2]{HuNualart09}),
	$v^\varepsilon$ has a Feynman--Kac representation
	\[
		v^\varepsilon(t\,,x) = \E_B\left[ v_0(B_t+x)\exp\left( \int_{(0,t)\times\R^d}
		\bm{p}_\varepsilon\left( y - x - B_{t-s}\right) \eta(\d s\,\d y) - \tfrac12 tf_\varepsilon(0)\right)
		\right].
	\]
	Define 
	\[
		B^{t,w}_s := B_s - \frac{s}{t}(B_t-w)\qquad\text{for all $s\in[0\,,t]$ and $t>0$ and $w\in\R^d$}.
	\]
	We can see that $B^{t,w}$ is a Brownian bridge on $[0\,,t]$, conditioned to go from
	the space-time point $(0\,,0)$ to the space-time point $(t\,,w)$. And in fact,
	\[
		\bm{p}_{s(t-s)/t}\left( y-\frac st x\right)
		\ \text{of \eqref{mild:U} is the probability density of $B^{t,x}_s$ at $y$.}
	\]
	Because
	$\{B^{t,w}_s\}_{s\in[0,t]}$ is independent of $B_t$, we may disintegrate and write
	\[
		v^\varepsilon(t\,,x) 
		=\int_{\R^d} \bm{p}_t(z-x)
		\E_B\left[  \exp\left( \int_{(0,t)\times\R^d}
		\bm{p}_\varepsilon\left( y - x - B_{t-s}^{t,z}\right) 
		\eta(\d s\,\d y) - \tfrac12 tf_\varepsilon(0)\right)
		\right]v_0(z)\,\d z.
	\]
	We compare the above to \eqref{v_0} in order to deduce from the fact that
	$v_0\in L^\infty(\R^d)$ is arbitrary
	that the following is a version of $u^\varepsilon(t\,,x\,;z)$:
	\[
		u^\varepsilon(t\,,x\,;z) = \bm{p}_t(z-x)
		\E_B \left[  \exp\left( \int_{(0,t)\times\R^d}
		\bm{p}_\varepsilon\left( y - x - B_{t-s}^{t,z}\right) 
		\eta(\d s\,\d y) - \tfrac12 tf_\varepsilon(0)\right)
		\right].
	\]
	We adopt this version of $u^\varepsilon(t\,,x\,;z)$ [rather than the old ones].
	Set $z=0$ to see that we have adopted the following versions of $u^\varepsilon(t\,,x)$
	and $U^\varepsilon(t\,,x)$:
	\begin{align*}
		u^\varepsilon(t\,,x) &= \bm{p}_t(x)
			\E_B\left[  \exp\left( \int_{(0,t)\times\R^d}
			\bm{p}_\varepsilon\left( y - x - B_{t-s}^{t,0}\right) 
			\eta(\d s\,\d y) - \tfrac12 tf_\varepsilon(0)\right)
			\right],\quad\text{and hence}\\
		U^\varepsilon(t\,,x) &= \E_B\left[  \exp\left( \int_{(0,t)\times\R^d}
			\bm{p}_\varepsilon\left( y - x - B_{t-s}^{t,0}\right) 
			\eta(\d s\,\d y) - \tfrac12 tf_\varepsilon(0)\right)
			\right].
	\end{align*}
	According to the Malliavin-calculus 
	method of Hu and Nualart \cite{HuNualart09} (see also  \cite[Theorem 1.9(2)]{CH19Comparison}), 
	$\lim_{\varepsilon\to0} U^\varepsilon(t\,,x)
	=U(t\,,x)$ in $L^2(\Omega)$ for every $t>0$ and $x\in\R^d$. Therefore, our goal of
	proving the stationarity of $U(t)$ would follow once we demonstrate the stationarity of
	$U^\varepsilon(t)$ for every $t,\varepsilon>0$.
	But that is not hard to do. Indeed, by the It\^o-Walsh isometry for stochastic integrals,
	\begin{align*}
		&\E_{\eta}\left[\int_{(0,t)\times\R^d}
			\bm{p}_\varepsilon\left( y-a - B_{t-s}^{t,0}\right) 
			\eta(\d s\,\d y)\times
			\int_{(0,t)\times\R^d}
			\bm{p}_\varepsilon\left(y-b - B_{t-s}^{t,0}\right) 
			\eta(\d s\,\d y)\right]\\
		&=\int_0^t\d s\int_{\R^d}\d y\int_{\R^d}f(\d z)\
			\bm{p}_\varepsilon\left( y-a - B_{t-s}^{t,0}\right) 
			\bm{p}_\varepsilon\left( z+y-b - B_{t-s}^{t,0}\right)\\
		&= t\int_{\R^d}f(\d z)\
			\bm{p}_{2\varepsilon}\left( z - b +a  \right)
			\qquad\text{[semigroup property]}\\
		&=t \left(\bm{p}_{2\varepsilon}*f\right)\left( b-a\right),
	\end{align*}
	which proves the asserted stationarity of $U^\varepsilon(t)$
	for every $t,\varepsilon>0$. 
\end{proof}

\begin{remark}\label{positivity}
As a consequence of Feynman-Kac formula, we can
see immediately that $U(t\,, x)\geq 0$ a.s.\ for all $t>0$ and $x\in \R^d$.
\end{remark}

\begin{proof}[Proof of Theorem \ref{th:U} (part 3): Behavior near $t=0$.]
	
	We now complete the proof by showing that $\lim_{t\to0}U(t\,,x)= 1$ in $L^k(\Omega)$
	for every $x\in\R^d$. By stationarity, it suffices to consider only the
	case that $x=0$.
	Now in accord with \eqref{mild:U} and \eqref{moment:U}, there exists a real number
	$K$ such that, uniformly for all $t\in(0\,,1)$,
	\begin{align*}
		\E\left(|U(t\,,0) - 1|^k\right) &\le \E\int_0^t\d s\int_{\R^d}\d y\int_{\R^d} f(\d w)\
			\bm{p}_{s(t-s)/t}(y)\bm{p}_{s(t-s)/t}(w+y)
			\|U(s\,,y)U(s\,,w+y)\|_{k/2}\\
		&\le K\int_0^t\d s\int_{\R^d}\d y\int_{\R^d} f(\d w)\
			\bm{p}_{s(t-s)/t}(y)\bm{p}_{s(t-s)/t}(w+y)\\
		&=K\int_0^t\left( \bm{p}_{2s(t-s)/t}*f\right)(0)\,\d s \le K\e^{\beta t}
			\int_0^t\e^{-\beta\{s\wedge(t-s)\}}\left( \bm{p}_{2s(t-s)/t}*f\right)(0)\,\d s\\
		&\le 4K\e^{\beta t}\Upsilon(2\beta)
			\qquad\text{for all $\beta>0$}.
	\end{align*}
	Set $\beta=1/t$ to find that
	$\E(|U(t\,,0) - 1|^k) \le 4K\e\Upsilon(2/t)\to 0$ as $t\to 0$,
	owing to the dominated convergence theorem, \eqref{Upsilon}, 
	and the theorem's condition that $\Upsilon(\beta)<\infty$ for one, hence all,
	$\beta>0$. This concludes the proof.
\end{proof}

\section{Ergodicity: proof of Theorem \ref{th:ergodic}}\label{ergodicity}

The following bound on the Malliavin derivative of $U(t\,,x)$ is a key technical result of the paper. Among other things,
it also plays a central role in our proof of Theorem \ref{th:ergodic}.

\begin{proposition}\label{pr:Du}
	Choose and fix $k\ge2$, $t>0$, and $x\in\R^d$. Then, $U(t\,,x)\in\cap_{k\ge2}\mathbb{D}^{1,k}$,
	and for almost every
	$(s\,,y)\in(0\,,t)\times\R^d$,
	\begin{align}
		\|D_{s,y}U(t\,,x) \|_k &\le \frac{64}{7}
		\exp\left\{ \frac t2\left[\beta_{7/8,k}+\frac{1}{2}\Upsilon^{-1}\left(\frac{1}{32z_k^2}\right)
		\right]\right\}\bm{p}_{s(t-s)/t}\left( y- \frac st x\right)\nonumber \\
		&:=C_{t, k} \, \bm{p}_{s(t-s)/t}\left( y- \frac st x\right) \label{C_{t, k}},
	\end{align}
	where $\beta_{\varepsilon,k}$ was defined in \eqref{beta:eps:k}.
\end{proposition}

Because of \eqref{U},
$U(t\,,x)\in\cap_{k\ge2}\mathbb{D}^{1,k}$ iff $u(t\,,x)\in\cap_{k\ge2}\mathbb{D}^{1,k}$.
Thus, portions of the above are already included in the work of Chen et al \cite[Proposition 5.1]{CHN}. The point
here is mainly the explicit bound for the moments of the Malliavin derivative of $U(t\,,x)$.

\begin{remark}
	Properties of the Malliavin derivative, and \eqref{PPPP}, together imply that the inequality of
	Proposition \ref{pr:Du} is equivalent to the following:
	\begin{align*}
		\|D_{s,y}u(t\,,x) \|_k &\le \frac{64}{7}
			\exp\left\{ \frac t2\left[\beta_{7/8,k}+\frac{1}{2}\Upsilon^{-1}\left(\frac{1}{32z_k^2}\right)
			\right]\right\}\bm{p}_t(x)\bm{p}_{s(t-s)/t}\left( y- \frac st x\right)\\
		&=\frac{64}{7}
			\exp\left\{ \frac t2\left[\beta_{7/8,k}+\frac{1}{2}\Upsilon^{-1}\left(\frac{1}{32z_k^2}\right)
			\right]\right\}\bm{p}_{t-s}(x-y)\bm{p}_s(y).
	\end{align*}
\end{remark}

The proof of Proposition \ref{pr:Du} requires some notation and two intervening lemmas.

Define $u_0(t\,,x) = \bm{p}_t(x)$, and iteratively let
\begin{equation}\label{mild:n}
	u_{n+1}(t\, , x) = \bm{p}_{t}(x) + \int_{(0,t)\times\R^d}
	\bm{p}_{t - r}(x - z)u_n(r\,, z)\,\eta(\d r\,\d z)
\end{equation}
for every $n\in\mathbb{Z}_+$. It is easy to see that, for every $n\ge 2$,
$u_n(t\,,x) = \bm{p}_t(x)U_n(t\,,x)$, where $U_n$ was defined in the proof
of Theorem \ref{th:U}, and denotes the $n$th stage in the Picard iteration approximation
of $U$. It follows from the proof of
Theorem \ref{th:U} that $u_n(t\,,x)$ converges to 
$u(t\,,x)=\bm{p}_t(x)U(t\,,x)$ in
$L^k(\Omega)$ as $n\to\infty$ for every $k\ge2$. It also follows from basic properties of the Malliavin derivative that a.s.,
\begin{equation} \label{BB2}
	D_{s,y}u_{n+1}(t\,,x) = \bm{p}_{t-s}(x-y)u_n(s\,,y)
	+ \int_{(s,t)\times\R^d}
	\bm{p}_{t - r}(x - z) D_{s,y}u_n(r\,, z)\,\eta(\d r\,\d z),
\end{equation}
for almost every $(s\,,y)\in(0\,,t)\times\R^d$, and all $n\in\mathbb{Z}_+$
for which the right-hand side is well defined. The following shows inductively that indeed
the right-hand side is well defined for every $n$, and provides a bound on its $L^k(\Omega)$-norms.
	
\begin{lemma}\label{lem:Du_n}
	Choose and fix $n\in\mathbb{N}$, $k\ge2$, $t>0$, and $x\in\R^d$,
	and let $\beta:=\beta_{7/8,k}$, as defined in \eqref{beta:eps:k}.
	Then,
	\begin{equation}\label{eq:induct}
		\|D_{s,y}u_n(t\,,x)\|_k\le\alpha_n \e^{\beta (t-s)/2}\bm{p}_{t-s}(x-y)\bm{p}_s(y),
	\end{equation}
	for almost every $(s\,,y)\in(0\,,t)\times\R^d$, where
	\[
		\alpha_1 := \adjustlimits\sup_{m\in\mathbb{Z}_+}\sup_{x\in\R^d}
		\sup_{s\in(0,t]}
		\|U_m(s\,,x)\|_k<\infty \quad\text{and}\quad
		\alpha_n := \left( \sqrt{8}\left[1-2^{-n}\right]  + 2^{-n}\right)\alpha_1
		\le4\alpha_1,
	\]
	for the random fields $\{U_n\}_{n=0}^\infty$ defined in the proof of Theorem \ref{th:U}.
\end{lemma}

The fact that $\alpha_1$ is finite is a consequence of the proof of 
\eqref{moment:U}. In fact, the proof of Theorem \ref{th:U} [with $\varepsilon=7/8$]
shows that
\begin{equation}\label{alpha_1}
	\alpha_1 \le \frac{16}{7}
	\exp\left\{ \frac{t}{4}\Upsilon^{-1}\left(\frac{1}{32z_k^2}\right)\right\}.
\end{equation}

\begin{proof}[Proof of Lemma \ref{lem:Du_n}]
	We proceed to prove \eqref{eq:induct} by using induction on $n$.
	
	Because $D_{s,y}u_0(t\,,x)=0$, it follows from \eqref{BB2} that
	$\|D_{s,y}u_1(t\,,x)\|_k\le\alpha_1\bm{p}_{t-s}(x-y)\bm{p}_s(y)$.
	In particular, \eqref{eq:induct} holds for $n=1$.
	Next, we suppose \eqref{eq:induct} is true for some integer $n\ge1$ and proceed
	to prove that it is true when $n$ is replaced by $n+1$.
	With this aim in mind, observe using the BDG inequality that
	\[
		\mathcal{E}_{n+1} = \mathcal{E}_{n+1}(s\,,y\,,t\,,x\,,k) :=
		\| D_{s,y}u_{n+1}(t\,,x)\|_k^2
	\]
	satisfies
	\begin{align*}
		\mathcal{E}_{n+1}
			&\le 2\alpha_1^2\left[\bm{p}_{t-s}(x-y)\bm{p}_s(y)\right]^2
			 + 2z_k^2 \int_s^t\d r\int_{\R^d}\d z\int_{\R^d}f(\d z')\\
		& \qquad \qquad    \times \bm{p}_{t-r}(x - z) \bm{p}_{t-r}(x - z - z')
			\left\| D_{s,y}u_n(r\,, z)\right\|_k \left\| D_{s,y}u_n(r\,, z'+z)\right\|_k\\
		&\le 2\alpha_1^2\left[\bm{p}_{t-s}(x-y)\bm{p}_s(y)\right]^2
			+ 2z_k^2\alpha_n^2\left[\bm{p}_s(y)\right]^2 
			\int_0^{t-s}\e^{\beta r}
			\,\d r\int_{\R^d}\d z\int_{\R^d}f(\d z')\\
		& \qquad \qquad \times \bm{p}_{t-s-r}(x - z) \bm{p}_r(z-y)\bm{p}_{t-s-r}(x - z - z')
			\bm{p}_r(z'+z-y),
	\end{align*}
	thanks to the induction hypothesis and a change of variables [$r\leftrightarrow r-s$].
	Apply \eqref{PPPP} in order to find that
	\begin{align*}
		\mathcal{E}_{n+1}
			&\le 2\alpha_1^2\left[\bm{p}_{t-s}(x-y)\bm{p}_s(y)\right]^2
			+ 2z_k^2\alpha_n^2\left[\bm{p}_{t-s}(x-y)\bm{p}_s(y)\right]^2
			\int_0^{t-s}\e^{\beta\{r\vee(t-s-r)\}}\,\d r\int_{\R^d}\d z\int_{\R^d}f(\d z')\\
		&\qquad \qquad \times  \bm{p}_{r(t-s-r)/(t-s)}\left( z-y - \frac{r}{t-s}(x-y) \right)
			\bm{p}_{r(t-s-r)/(t-s)}\left( z'+z-y - \frac{r}{t-s}(x-y) \right)\\
		&= \left[\bm{p}_{t-s}(x-y)\bm{p}_s(y)\right]^2\left\{ 2\alpha_1^2
			+ 2z_k^2\alpha_n^2\int_0^{t-s}\e^{\beta\{r\vee(t-s-r)\}}
			\left(\bm{p}_{2r(t-s-r)/(t-s)}*f\right)(0)\,\d r\right\},
	\end{align*}
	where we have appealed to the semigroup property of the heat kernel 
	for the last line.  Take square roots and apply 
	the simple inequality $(|a|+|b|)^{1/2}\le|a|^{1/2}+|b|^{1/2}$ ---
	valid for all $a,b\in\R$ --- to see that
	\[
		\frac{\| D_{s,y}u_{n+1}(t\,,x)\|_k}{\bm{p}_{t-s}(x-y)\bm{p}_s(y)}\le
		\sqrt{2}\alpha_1
		+ \alpha_n\left\{2z_k^2\int_0^{t-s}\e^{\beta\{r\vee(t-s-r)\}}
		\left(\bm{p}_{2r(t-s-r)/(t-s)}*f\right)(0)\,\d r\right\}^{1/2}.
	\]
	Since $r\vee(t-s-r) = t-s - \{r\wedge(t-s-r)\},$ this proves that
	\begin{align*}
		\frac{\| D_{s,y}u_{n+1}(t\,,x)\|_k}{\bm{p}_{t-s}(x-y)\bm{p}_s(y)} &\le
			\sqrt{2}\alpha_1
			+ \alpha_n\e^{\beta(t-s)/2}
			\left\{2z_k^2\int_0^{t-s}\e^{-\beta\{r\wedge(t-s-r)\}}
			\left(\bm{p}_{2r(t-s-r)/(t-s)}*f\right)(0)\,\d r\right\}^{1/2}\\
		&\le\sqrt{2}\alpha_1
			+ \alpha_n\e^{\beta(t-s)/2}\sqrt{8z_k^2\Upsilon(2\beta)}
			\hskip1in\text{[see Lemma \ref{lem:p*f}]}\\
		&\le \sqrt{2}\alpha_1
			+ \tfrac12\alpha_n\e^{\beta(t-s)/2} \le
			\alpha_{n+1}\e^{\beta(t-s)/2},
	\end{align*}
	thanks to the definition \eqref{beta:eps:k} of $\beta=\beta_{7/8,k}$
	and the readily-checkable fact that $\alpha_{n+1}=\sqrt{2}\alpha_1+\frac12\alpha_n$.
	This proves \eqref{eq:induct} with $n$ replaced by $n+1$ and concludes the inductive
	stage of the argument.
\end{proof}

Our next technical lemma implies, inductively, that $u_n(t\,,x)\in\mathbb{D}^{1,2}$ for
every $n\in\mathbb{N}$.

\begin{lemma}\label{lem:E(||D||)}
	There exist real numbers $A,B>0$ such that
	\[
		\adjustlimits
		\sup_{n\in\mathbb{N}}\sup_{x\in\R^d}
		\E\left( \left\| Du_n(t\,,x) \right\|_{\mathcal{H}}^2\right)
		\le A t^{-d} \e^{Bt}
		\qquad\text{for all $t>0$}.
	\]
\end{lemma}

\begin{proof}
	We compute directly, using Lemma \ref{lem:Du_n}, as follows:
	\begin{align*}
		\E\left(\|D u_n(t\,,x)\|_{\mathcal{H}}^2\right) &=
			\int_0^t\d s\int_{\R^d}\d y\int_{\R^d}f(\d y')\
			\E\left[ D_{s,y}u_n(t\,,x) D_{s,y+y'}u_n(t\,,x)\right]\\
		&\le\int_0^t\d s\int_{\R^d}\d y\int_{\R^d}f(\d y')
			\left\| D_{s,y}u_n(t\,,x) \right\|_2 \left\| D_{s,y+y'}u_n(t\,,x)\right\|_2\\
		&\le c\int_0^t\d s\int_{\R^d}\d y\int_{\R^d}f(\d y')\
			\bm{p}_{t-s}(x-y)\bm{p}_s(y)\bm{p}_{t-s}(x-y-y')\bm{p}_s(y+y'),
	\end{align*}
	where $c:=16\alpha_1^2 \exp(\beta_{7/8,2}t)$, using the constants
	of Lemma \ref{lem:Du_n}. Note that $\alpha_1$ depends on $t$, 
	and in fact Theorem \ref{th:U} ensures
	that $c\le c_1\exp(c_2 t)$ where $c_1$ and $c_2$ do not depend on $t$.
	Apply \eqref{PPPP} to see that
	\begin{align*}
		\E\left(\|D u_n(t\,,x)\|_{\mathcal{H}}^2\right) 	
			&\le c_1\e^{c_2t}[\bm{p}_t(x)]^2\int_0^t\d s\int_{\R^d}\d y\int_{\R^d}f(\d y')\
			\bm{p}_{s(t-s)/t}\left( y - \frac st x\right)
			\bm{p}_{s(t-s)/t}\left( y+y' - \frac st x\right)\\
		&=c_1\e^{c_2t}[\bm{p}_t(x)]^2\int_0^t\left( \bm{p}_{2s(t-s)/t}*f\right)(0)\,
			\d s\\
		&\le \frac{c_1\e^{(1+c_2)t}}{(2\pi t)^d}\int_0^t \e^{-\{s\wedge(t-s)\}}
			\left( \bm{p}_{2s(t-s)/t}*f\right)(0)\,\d s.
	\end{align*}
	Since $\Upsilon(2)<\infty$, an appeal to Lemma \ref{lem:p*f} completes the proof.
\end{proof}

We are in position to verify Proposition \ref{pr:Du}.

\begin{proof}[Proof of Proposition \ref{pr:Du}]
	The proof is similar to that of \cite[Theorem 6.4]{CKNP}. 
	Choose and fix $k\ge2$, $t>0$, and $x\in\R^d$.
	Thanks to \eqref{U} and \eqref{PPPP}, the proposition's assertion is equivalent to the following
	inequality, valid for a.e.\ $(s\,,y)\in(0\,,t)\times\R^d$:
	\[
		\|D_{s,y}u(t\,,x) \|_k \le \frac{64}{7}
		\exp\left\{ \frac t2\left[\beta_{7/8,k}+\frac{1}{2}\Upsilon^{-1}\left(\frac{1}{32z_k^2}\right)
		\right]\right\}
		\bm{p}_{t-s}(x-y)\bm{p}_s(y),
	\]
	We will prove the above reformulation of the proposition.
	
	Thanks to Lemma \ref{lem:E(||D||)} and closeablility properties of the Malliavin derivative
	operator (see Nualart \cite{Nualart}), it follows that, after possibly moving to subsequence, that
	$Du_n(t\,,x)$  converges to $Du(t\,,x)$ in the weak topology of
	$L^2(\Omega\, ; \mathcal{H})$.
	Then, we use a smooth approximation $\{\psi_\varepsilon\}_{\varepsilon>0}$
	to the identity in $\R_+\times \R^d$,  and apply Fatou's lemma and
	duality for $L^k$-spaces in order
	to find that, for almost every $(s\,,y) \in (0\,,t) \times \R^d$ and for all $k\ge 2$,
	\begin{align*}
		\|D_{s,y}u(t\,,x) \|_k & \le  \limsup_{\varepsilon \to 0}
			\left \| \int_{\R_+\times \R^d} D_{s'\!,y'} u(t\,,x)
			\psi_\varepsilon(s-s', y-y') \,\d s'\d y' \right\|_k\\
		 & \le  \limsup_{\varepsilon\to 0}
			\sup_{\|G \|_{k/(k - 1)}\le 1}
			\left| \int_{\R_+\times \R^d}  \E\left[ G D_{s'\!,y'} u(t\,,x) \right]
			\psi_\varepsilon(s-s', y-y')\, \d s'\d y'  \right|.
	\end{align*}
	Choose and fix a random variable $G\in L^{2}(\Omega)$ such that
	$ \| G \|_{k/(k-1)}\le 1$. We can find an unbounded subsequence $n(1) < n(2)<\cdots$
	of positive integers such that
	\begin{align*}
		& \left| \int_{\R_+\times \R^d}  \E \left[ G D_{s'\!,y'} u(t\,,x) \right]
			\psi_\varepsilon(s-s', y-y') \,\d s'\d y'  \right|  \\
		&= \lim _{\ell\rightarrow \infty} \left|
			\int_{\R_+\times \R^d}  \E\left[ G D_{s'\!,y'} u_{n(\ell)}(t\,,x) \right]
			\psi_\varepsilon(s-s', y-y')\, \d s'\d y'   \right| \\
		&\le\limsup_{\ell\to\infty}
			\int_{\R_+\times \R}  \left\| D_{s'\!,y'} u_{n(\ell)}(t\,,x) \right\|_k
			\psi_\varepsilon(s-s', y-y') \,\d s'\d y' \\
        	&\le \sup_{n\in\mathbb{N}}\alpha_n\int_{(0,t)\times \R^d}   \e^{\beta(t-s')/2}
			\bm{p}_{t - s'}(x - y')\bm{p}_{s'}(y')
			\psi_\varepsilon(s-s', y-y')\, \d s'\d y'\\
		&\le 4\alpha_1\e^{\beta t/2} \int_{(0,t)\times \R^d}   
			\bm{p}_{t - s'}(x - y')\bm{p}_{s'}(y')
			\psi_\varepsilon(s-s', y-y')\, \d s'\d y';
	\end{align*}
	see Lemma \ref{lem:Du_n}. Let $\varepsilon\to 0$ and appeal to \eqref{alpha_1} in order to finish.
\end{proof}

The second, and final, step of the proof of Theorem \ref{th:ergodic} 
is a Poincar\'e-type inequality for certain
nonlinear functionals of $U$. In order to describe that inequality, let us first choose
and fix points $\zeta^1,\ldots,\zeta^k\in\R^d$ and bounded Lipschitz-continuous functions
$g_1,\ldots,g_k\in C^1_b(\R)$ such that
\begin{equation}\label{eq:WLOG}
	g_j(0)=0\quad\text{and}\quad
	\text{\rm Lip}(g_j)=1
	\qquad \text{for every $j=1,\ldots,k$}.
\end{equation}
Then define for every $t,N>0$ and $x\in\R^d$,
\begin{equation}\label{eq:G}
	\mathcal{G}(t\,,x) := \prod_{j=1}^k g_j \left( U(t\,,x+\zeta^j)\right).
\end{equation}

\begin{lemma}\label{lem:V_N(t)}
	Choose and fix an integer $k\ge2$, points
	$x,\zeta^1,\ldots,\zeta^k\in\R^d$, and functions
	$g_1,\ldots,g_k\in C^1_b(\R)$ that satisfy \eqref{eq:WLOG}. Then, there exists
	a real number $A=A(t\,,k\,,g_1\,,\ldots,g_k)$ given by \eqref{eq:A} below such 
	that
	\[
		\left| \Cov\left( \mathcal{G}(t\,,0)\,,\mathcal{G}(t\,,x)\right)\right|
		\le A^2\sum_{j_0=1}^k\sum_{j_1=1}^k\int_0^t
		\left( \bm{p}_{2s(t-s)/t} * f\right)\left(\frac st (x + \zeta^{j_0}-\zeta^{j_1})\right)\d s.
	\]
\end{lemma}


\begin{proof}
	By the chain rule of Malliavin calculus (see Nualart \cite{Nualart}),
	\[
		D_{s,z} \mathcal{G}(t\,,x) = \bm{1}_{(0,t)}(s)\sum_{j_0=1}^k
		\left(\prod_{\substack{j=1\\j \not= j_0}}^{k}
		g_j \left( U(t\,,x+\zeta^j) \right) \right)
		g_{j_0}' \left(U(t\,,x+\zeta^{j_0})\right)
		D_{s,z} U(t\,, x+ \zeta^{j_0}),
	\]
	for almost every $(s\,,z)\in(0\,,t)\times\R^d$. Therefore, Proposition \ref{pr:Du} ensures that
	\begin{align*}
		\left\| D_{s,z}\mathcal{G}(t\,,x)\right\|_k &\le \bm{1}_{(0,t)}(s)
			\adjustlimits
			\max_{1\le j\le k}\sup_{a\in\R}|g_j(a)|^{k-1}
			\sum_{j_0=1}^k \left\| D_{s,z} U(t\,, x+ \zeta^{j_0})\right\|_k\\
		&\le A\bm{1}_{(0,t)}(s)\sum_{j_0=1}^k \bm{p}_{s(t-s)/t}\left( z- \frac st (x+\zeta^{j_0})\right),
	\end{align*}
	with 
	\begin{equation}\label{eq:A}
		A  := \frac{64}{7}
		\exp\left\{ \frac t2\left[\beta_{7/8,k}+\frac{1}{2}\Upsilon^{-1}\left(\frac{1}{32z_k^2}\right)
		\right]\right\}\adjustlimits\max_{1\le j\le k}\sup_{a\in\R}|g_j(a)|^{k-1}.	
	\end{equation}
	It follows from the Poincar\'e inequality \eqref{Poincare:Cov} that
	$| \Cov( \mathcal{G}(t\,,x) \,, \mathcal{G}(t\,,0))|$ is bounded from above by
	\[
		A^2\sum_{j_0=1}^k\sum_{j_1=1}^k\int_0^t\d s\int_{\R^d}\d y\int_{\R^d}f(\d y')\
		\bm{p}_{s(t-s)/t}\left( y- \frac st (x+\zeta^{j_0})\right)
		\bm{p}_{s(t-s)/t}\left( y'+y- \frac st \zeta^{j_1}\right).
	\]
	Apply the semigroup property of the heat kernel together with Fubini's theorem to finish.
\end{proof}

\begin{proof}[Proof of Theorem \ref{th:ergodic}]

    Define
	\[
		V_N(t) := {\rm Var} \left (\frac{1}{N^d}\int_{[0,N]^d}  \mathcal{G}(t,x) \d x \right)
		\quad\text{and}\quad
		\mathcal{G}(x) :=
		\prod_{j=1}^k g_j ( U(t\,,x+\zeta^j))\qquad
		\text{for all $x\in\R^d$},
	\]
     where   $\mathcal{G}(t,x)$ has been defined in \eqref{eq:G} 
     and  the bounded functions $g_1,\ldots,g_k$ therein satisfy \eqref{eq:WLOG}.
    Since $U(t)$ is stationary [Theorem \ref{th:U}], 
    \cite[Lemma 7.2]{CKNP} implies the desired ergodicity provided that we prove that, for all $t>0$,
   \begin{align}\label{V_N}
   \lim_{N\to\infty}V_N(t)=0.
   \end{align}
    For every real number $N>0$, define the functions
    \begin{align}\label{I_N}
		I_N(x) := N^{-d} \bm{1}_{[0,N]^d}(x)  \quad \text{and}
		\quad \tilde {I}_N(x)= I_N(-x) \qquad\text{for $x \in \R^d$}.
   \end{align}
   
	By Lemma \ref{lem:V_N(t)}, 
	\begin{align}
		V_N(t)& = \frac{1}{N^{2d}}
			\int_{[0, N]^d}\d x\int_{[0,N]^d}\d y\  
			{\rm Cov} \left[  \mathcal{G}(t\,,x) \,, \mathcal{G}(t\,,y) \right]\nonumber\\
		&\le \frac{A^2}{N^{2d}}\sum_{j_0=1}^k\sum_{j_1=1}^k\int_0^t \d s
			\int_{[0, N]^d}\d x\int_{[0,N]^d}\d y\
			\left( \bm{p}_{2s(t-s)/t} * f\right)
			\left(\frac st (x-y + \zeta^{j_0}-\zeta^{j_1})\right)\nonumber\\
		& = A^2\sum_{j_0=1}^k\sum_{j_1=1}^k\int_0^t \d s
			\int_{\R^d} \d x \left(I_N*\tilde{I}_N\right)(x)
			\left( f*\bm{p}_{2s(t-s)/t} \right)
			\left(\frac st (x +    \zeta^{j_0}-\zeta^{j_1})\right).\nonumber
	\end{align}
	Therefore,  \eqref{inequality} implies that 
	\begin{align} 
		V_N(t)& \leq\frac{k^2A^2 }{\pi^d}\int_0^t\d s\int_{\R^d }\hat{f}(\d y) \
		\e^{-\frac{s(t-s)}{t}\|y\|^2}\prod_{j=1}^d\frac{1-\cos(Nsy_j/t)}{ (Nsy_j/t)^2}.\nonumber
	\end{align}
	The quantity $\prod_{j=1}^d\{1-\cos(Nsy_j/t)\}/(Nsy_j/t)^2$ 
	is bounded above by $2^{-d}$ and converges to zero as $N\to \infty$ for each $s>0$ and $y\not=0$.
	Since $\hat{f}\{0\}=0$, the dominated convergence theorem implies
	that $\lim_{N\rightarrow \infty} V_N(t)=0$,  taking into account that
	\[
		\int_0^t \d s\int_{\R^d } \hat{f}(\d y)\  \e^{-\frac{s(t-s)}{t}\|y\|^2} <\infty,
	\]
	which follows from Dalang's condition $\Upsilon(1)<\infty$.  
	This proves \eqref{V_N}, whence follows ergodicity.
\end{proof}

\section{Asymptotic variance}\label{analyzevar}

Recall the spatial average $\mathcal{S}_{N,t}$ and the quantity $\mathcal{R}(f)$ are defined in \eqref{average} and \eqref{R(f)}, respectively. 

\begin{theorem}[$d\ge1$]\label{th:AV:d>1}
	For all $t>0$, 
	\begin{equation}\label{eq:AV:d>1}
		\lim_{N\to\infty} N\Var(\mathcal{S}_{N,t}) =t\, \mathcal{R}(f).
	\end{equation}
	The quantity on the
	right-hand side is  strictly positive and \eqref{eq:AV:d>1} holds whenever $\mathcal{R}(f)$ is finite or infinite. 
\end{theorem}

According to the criteria in Proposition \ref{pr:R} and Lemma \ref{lem:f:1} below, the value of $\mathcal{R}(f)$ could be finite or infinite. For example, if  $f(\d x)=\bm{p}_1(x)\d x$ and $d \geq 2$, then $\mathcal{R}(f)<\infty$; if $f$ is given by the Riesz kernel, i.e., $f(\d x)=\|x\|^{-\beta}\d x, 0<\beta<2\wedge d$, then  $\mathcal{R}(f)=\infty$.
Moreover, it is easy to deduce from Lemma \ref{lem:f:1} below that, in the case that $d=1$, $\mathcal{R}(f)$ is always infinite,
which might suggest that the above $1/N$ rate of decay of $\Var(\mathcal{S}_{N,t})$ is
not the right one in one dimension. Indeed, this is the case. And the following result identifies
the correct rate canonically as $N^{-1}\log N$ in dimension one.

\begin{theorem}[$d=1$]\label{th:AV:d=1}
	Assume $f(\R)<\infty$. Then for all $t>0$,
	\begin{align}\label{lim(V):d=1}
		tf(\R)\le
		\liminf_{N\to\infty} \frac{N}{\log N}\,\Var(\mathcal{S}_{N,t}) \le
		\limsup_{N\to\infty} \frac{N}{\log N}\,\Var(\mathcal{S}_{N,t}) \le
		2tf(\R).
	\end{align}
	Both bounds are sharp in the following sense:
	\begin{compactenum}
		\item If $f=a\delta_0$ for some $a>0$, then 
			$\Var(\mathcal{S}_{N,t}) \sim 2tf(\R) N^{-1}\log N$ as $N\to\infty$;
		\item If $\lim_{x\to\infty}\hat{f}(x)=0$, then
			$\Var(\mathcal{S}_{N,t}) \sim tf(\R) N^{-1}\log N$ as $N\to\infty$.
	\end{compactenum}
\end{theorem}

\begin{remark}
	The condition in Item 2 of Theorem \ref{th:AV:d=1} is a well-known one.
	Indeed, finite Borel measures whose Fourier transforms vanish at infinity are called
	\emph{Rajchman measures}. See Lyons \cite{Lyons} for the background and rich
	history of the work on Rajchman measures in classical harmonic analysis.
\end{remark}

We now turn to the Riesz kernel case.  Define
\begin{equation}\label{varphi}
	\varphi(y) := \frac{1-\cos y}{y^2}\quad\text{for all $y\in\R\setminus\{0\}$},
\end{equation}
and $\varphi(0):=1/2$ to preserve continuity.

\begin{theorem}[Riesz kernel]\label{Riesz}
          Assume $f(\d x)=\|x\|^{-\beta}\d x$  
          and $\hat{f}(\d x)=\kappa_{\beta, d}\,\|x\|^{\beta-d}\d x$, where $0<\beta< 2\wedge d$ and $\kappa_{\beta, d}$
           is a positive constant depending on $\beta$ and $d$.
         \begin{compactenum}
            \item  If $0<\beta<1$, then
      \begin{equation} \label{2var:beta<1}
\lim_{N\rightarrow \infty} N^{\beta} V_N(t)= \frac t {1-\beta} \int_{[-1,1]^d}   \| z\| ^{-\beta}  \prod_{i=1}^d (1- |z_i|)   \d z :=t \, \sigma_{0, \beta, d}.
\end{equation}
         \item If $1=\beta < 2\wedge d$, then
        \begin{align}\label{var:beta=1}
        \lim_{N\to\infty}\frac{N}{\log N}V_N(t)=  \frac{2t\, \kappa_{1, d}}{\pi^d } \int_{\R^d} \|z\|^{1-d}
        \prod_{j=1}^d\varphi(z_j)  \d z:= t\, \sigma_{1, \beta, d}.
       \end{align}
                \item If $1<\beta < 2\wedge d$, then
         \begin{align}\label{var:beta>1}
         \lim_{N\to\infty}N^{2-\beta}V_N(t)= \frac{t^{2-\beta}\,\kappa_{\beta, d}}{\pi^d} \int_{\R^d}\|z\|^{2-\beta-d}
         \prod_{j=1}^d\varphi(z_j)\d z \int_{0}^{\infty} r^{\beta-2} \e^{-r}\d r:= t^{2-\beta}\sigma_{2,\beta, d}.
         \end{align}
\end{compactenum}
\end{theorem}
We now begin  to work toward proving the above theorems. First, we denote 
\begin{equation}\label{V_N(t)}
 	V_N(t):=\Var(\mathcal{S}_{N,t}) = \int_{\R^d} \left( I_N*\tilde{I}_N\right)(x)
	\chi_t(x)\,\d x,
\end{equation}
where $I_N$ and $\tilde{I}_N$, defined in \eqref{I_N}, are given by 
$I_N(x) := N^{-d} \bm{1}_{[0,N]^d}(x)  $ and
$ \tilde {I}_N(x)= I_N(-x) $ for $x \in \R^d$, and for every $N,t>0$ and $x\in\R^d$,
\begin{equation}\label{I:chi}
	\chi_t(x) := \Cov\left[ U(t\,,0) ~,~ U(t\,,x)\right].
\end{equation}

Now we begin to establish a series of supporting lemmas.

\begin{lemma}[$d\ge1$]\label{lem:chi}
 	Let $\chi$ be defined by \eqref{I:chi}. Then,
	for every $t>0$ and $x\in\R^d$,
 	\[
		\chi_t(x) = \int_0^t\left( \bm{p}_{2s(t-s)/t}*f\right)(sx/t)\,\d s
		+ \int_0^t\d s\int_{\R^d} f(\d y)\
		\bm{p}_{2s(t-s)/t}\left( y - \frac st x\right)\chi_s(y).
	\]
\end{lemma}
 
\begin{proof}
 	Apply \eqref{mild:U} and elementary properties of the Walsh integral to see that
	\begin{align*}
		&\E\left[ U(t\,,0)U(t\,,x)\right]\\
		&= 1 + \int_0^t\d s\int_{\R^d}\d y'\int_{\R^d}f(\d y)\
			\bm{p}_{s(t-s)/t}(y')\bm{p}_{s(t-s)/t}\left( y+y'-\frac st x\right)
			\E\left[ U(s\,,y')U(s\,,y+y')\right]\\
		&= 1 + \int_0^t\d s\int_{\R^d}\d y'\int_{\R^d}f(\d y)\
			\bm{p}_{s(t-s)/t}(y')\bm{p}_{s(t-s)/t}\left( y+y'-\frac st x\right)
			\E\left[ U(s\,,0)U(s\,,y)\right],
	\end{align*}
	owing to the stationarity (Theorem \ref{th:U}). This and the semigroup property
	of the heat kernel together imply the lemma since 
	$\E[U(t\,,0)U(t\,,x)]=\chi_t(x)+1.$
\end{proof}

Our second supporting lemma describes the behavior of $\chi_t$ as $t\to0$.

\begin{lemma}[$d\ge1$]\label{lem:chi->0}
	$\lim_{t\downarrow0}\chi_t(x)=0$ uniformly for all $x\in\R^d$.
\end{lemma}

\begin{proof}
	It is easy to deduce from Lemma \ref{lem:chi} and positivity of the solution
	 (see Remark \ref{positivity}) that
	\begin{equation}\label{chi>0}
		\chi_t(x) \ge \int_0^t\left( \bm{p}_{2s(t-s)/t}*f\right)(sx/t)\,\d s\ge0
		\qquad\text{for all $t>0$ and $x\in\R^d$}.
	\end{equation}
	Now,
	the Cauchy--Schwarz inequality and stationarity together ensure that
	$\chi_t(x)\le\chi_t(0)$. Therefore,  it suffices to prove that
	$\chi_t(0)\to0$ as $t\downarrow0$. Theorem \ref{th:U} ensures that 
	$C:=\sup_{t\in(0,1)}\chi_t(0)=\sup_{t\in(0,1)}\sup_{x\in\R^d}\chi_t(x)<\infty$.
	Therefore, we deduce from Lemma \ref{lem:chi}  that
	\[
		\chi_t(0) \le  (1+C)\int_0^t\left( \bm{p}_{2s(t-s)/t}*f\right)(0)\,\d s.
	\]
	Since $\chi_t(0) \le  (1+C)4 \e^{\beta t}\Upsilon(2\beta)$ for every $\beta,t>0$
	[Lemma \ref{lem:p*f}], it follows that
	\[
		\limsup_{t\to0} \chi_t(0)\le(1+C)4 \lim_{\beta\to\infty}\Upsilon(2\beta)=0.
	\]
	This concludes the proof.
\end{proof}

In light of \eqref{V_N(t)} and Lemma \ref{lem:chi}, we write 
\begin{align}\label{V1+2}
          V_N(t)=\Var(\mathcal{S}_{N,t}) = V_N^{(1)}(t) +  V_N^{(2)}(t), 
\end{align}
where
\begin{align}
           V_N^{(1)}(t)&= \int_0^t\d s\int_{\R^d}\d x\,\left( I_N*\tilde{I}_N\right)(x)
			\left( \bm{p}_{2s(t-s)/t}*f\right)(sx/t)\label{V^{(1)}_N(t)},\\
			V_N^{(2)}(t)&= 
           \int_0^t\d s\int_{\R^d}\d x\,\left( I_N*\tilde{I}_N\right)(x)
			\int_{\R^d} f(\d y)\
		\bm{p}_{2s(t-s)/t}\left( y - \frac st x\right)\chi_s(y). \label{V^{(2)}_N(t)}
\end{align}
As we will see, the main contribution for the asymptotic behavior of $V_N(t)$ is $V_N^{(1)}(t)$,
 thanks to Lemma \ref{lem:chi->0}.

\subsection{Analysis in dimension $d\ge2$}

The primary goal of this section is to prove Theorem \ref{th:AV:d>1}.  Therefore,
in this section, we will not assume that $d\ge2$ unless we say so explicitly.
Recall the function $\varphi$ defined in \eqref{varphi}.

\begin{lemma}[$d\ge1$]\label{lem:AV:1}
	For every $N,t>0$,
 	\begin{align*}
		V_N^{(1)}(t)&=\int_{\R^d}\left( I_N*\tilde{I}_N\right)(x)\,\d x
			\int_0^t\d s\left( \bm{p}_{2s(t-s)/t}*f\right)(sx/t)\\
		&= \frac{t}{N\pi^d} \int_0^N\d s\int_{\R^d}\hat{f}(\d z)\,\e^{-t\|z\|^2(1-s/N)s/N} 
		\prod_{j=1}^d\varphi(z_js).
	\end{align*}
\end{lemma}
 
\begin{proof}
 	By \eqref{equal}, 
	 \begin{align*}
		&\int_{\R^d}\left( I_N*\tilde{I}_N\right)(x)\,\d x
			\int_0^t\d s\left( \bm{p}_{2(t-s)/t}*f\right)(sx/t)\\
		&\hskip1in= \frac{1}{\pi^d} \int_0^t\d s\int_{\R^d}\hat{f}(\d z)\,\e^{-s(t-s)\|z\|^2/t} 
		\prod_{j=1}^d\varphi(Nz_js/t)\\
		&\hskip1in=  \frac{t}{N\pi^d} \int_0^N\d s\int_{\R^d}\hat{f}(\d z)\,\e^{-t\|z\|^2(1-s/N)s/N} 
		\prod_{j=1}^d\varphi(z_js),
	\end{align*}
	where in the second equality we use change of variable ($s\to st/N$).
\end{proof}

Before we prove Theorem \ref{th:AV:d>1}, we give some estimates on the quantity $\mathcal{R}(f)$.
\begin{proposition}[$d\ge 1$]\label{pr:R}
	 Recall $\mathcal{R}(f)$ from \eqref{R(f)}.
	Then,
	\[
				2^{1-2d} \int_0^\infty f\left([-r\,,r]^d\right)\frac{\d r}{r^d} \leq
				\mathcal{R}(f) \leq \int_0^\infty f\left([-r\,,r]^d\right)\frac{\d r}{r^d}.
	\]
\end{proposition}
\begin{proof}
         We observe that $\prod_{j=1}^d\varphi(z_jr)=2^{-d} r^{-d}[(I_1*\tilde{I}_1)(\bullet /r)]\widehat{\phantom{a}}(z)$
         for all $z\in\R^d$ and $r>0$.
         Hence we can write
	\begin{align*}
		\mathcal{R}(f) 
			&=\frac{1}{(2\pi)^d}\int_0^\infty \frac{\d r}{r^{d}}\int_{\R^d}\hat{f}(\d z)\
			\left(\widehat{\left(I_1*\tilde{I}_1\right)(\bullet /r)}\right)(z).
	\end{align*}
	
	Denote $\phi_r= \left(I_1*\tilde{I}_1\right)(\bullet /r)$ for every fixed $r>0$.
	Choose a non-negative smooth function $\psi$ with compact support such that $\int_{\R^d}\psi(x) \d x =1$. 
	For $0<\varepsilon <1$, define $\psi_{\varepsilon}(x)=\varepsilon^{-d}\psi(x/\varepsilon)$ for all $x\in \R^d$. 
	It is clear that $\psi_{\varepsilon}*\phi_r$ has compact support uniformly for all $0<\varepsilon<1$ 
	and $\sup_{0<\varepsilon<1}\sup_{x\in\R^d}\left(\psi_{\varepsilon}*\phi_r\right)(x)<\infty$.
	Moreover, we have $\sup_{0<\varepsilon<1}\sup_{x\in \R^d}|\hat{\psi}_{\varepsilon}(x)|\leq 1$ 
	and $\lim_{\varepsilon\to0}\hat{\psi}_{\varepsilon}(x)=1$ for all $x\in \R^d$. 
	Using these facts and that $f$ is locally integrable as a tempered distribution 
	and $\int_{\R^d}\hat{\phi}_r(x)\hat{f}(\d x)<\infty$ by Dalang's condition,  we obtain that 
	for every fixed $r>0$,
	\begin{align*}
	\int_{\R^d}\phi_r(x)f(\d x)&=\lim_{\varepsilon\to 0} \int_{\R^d}\left(\psi_{\varepsilon}*\phi_r\right)(x)f(\d x)
	= \lim_{\varepsilon\to 0} \frac{1}{(2\pi)^d}\int_{\R^d}
	\hat{\psi}_{\varepsilon}(x)\hat{\phi}_r(x)\hat{f}(\d x) \\
	&= \frac{1}{(2\pi)^d}\int_{\R^d}\hat{\phi}_r(x)\hat{f}(\d x) ,
	\end{align*}
	where the first and third equalities hold by dominated convergence theorem
	 and the second by  the definition of the Fourier transform and the property
	 $ \widehat{ \psi_\varepsilon *\phi_r}= \hat{\psi}_{\varepsilon}\hat{\phi}_r$.
	
	Therefore, 
		\begin{align*}
		\mathcal{R}(f) 
			&=\int_0^\infty \frac{\d r}{r^d}\int_{\R^d}
			\left(I_1*\tilde{I}_1\right)( z/r)
			f(\d z).
	\end{align*}
	Now appealing to the inequality  
	$2^{-d}\bm{1}_{[-1/2, 1/2]^d}\leq I_1*\tilde{I}_1\leq \bm{1}_{[-1, 1]^d}$ (see \cite[(3.17)]{CKNP}), we obtain
	\begin{align*}
		2^{1-2d} \int_0^\infty f\left([-r\,,r]^d\right)\frac{\d r}{r^d} \leq \mathcal{R}(f) 
			&\leq \int_0^\infty f\left([-r\,,r]^d\right)\frac{\d r}{r^d},
	\end{align*}
	which completes the proof.  
\end{proof}

Now we can prove Theorem \ref{th:AV:d>1}.

\begin{proof}[Proof of Theorem \ref{th:AV:d>1}]
         By Proposition \ref{pr:R}, it is clear that $\mathcal{R}(f)$ is strictly positive since we assume $f(\R^d)>0$ throughout 
         the paper. Let us proceed with the proof of \eqref{eq:AV:d>1}.

	Assume $\mathcal{R}(f)<\infty$ first.
	By Lemma \ref{lem:AV:1} and the dominated convergence theorem,
	\[
		\lim_{N\to\infty}NV_N^{(1)}(t)= \lim_{N\to\infty}N\int_{\R^d}\left( I_N*\tilde{I}_N\right)(x)\d x
		\int_0^t\d s\left( \bm{p}_{2(t-s)/t}*f\right)(sx/t) = t\mathcal{R}(f).
	\]
	In light of \eqref{V^{(2)}_N(t)}, it remains to prove that
	\begin{equation}\label{goal:AV:1}
		\lim_{N\to\infty}N\int_{\R^d}\left( I_N*\tilde{I}_N\right)(x)\,\d x
		\int_0^t\d s\int_{\R^d} f(\d y)\
		\bm{p}_{2s(t-s)/t}\left( y - \frac st x\right)\chi_s(y)=0.
	\end{equation}
	
	By the Cauchy--Schwarz inequality and stationarity, $\chi_t(x)\le\chi_t(0)$
	for all $t>0$ and $x\in\R^d$. Therefore,
 	\begin{align*}
		&\int_{\R^d}\left( I_N*\tilde{I}_N\right)(x)\,\d x
			\int_0^t\d s\int_{\R^d} f(\d y)\
			\bm{p}_{2s(t-s)/t}\left( y - \frac st x\right)\chi_s(y)\\
		&\hskip1in\le\int_0^t\chi_s(0)\,\d s \int_{\R^d}
			\left( I_N*\tilde{I}_N\right)(x)\,\d x\int_{\R^d} f(\d y)\
			\bm{p}_{2s(t-s)/t}\left( y - \frac st x\right),
	\end{align*}
	for every $N,t>0$.
	Repeat the computation of Lemma \ref{lem:AV:1} to find that,
	for every $N,t>0$,
 	\begin{align} \nonumber
		&\int_{\R^d}\left( I_N*\tilde{I}_N\right)(x)\,\d x
			\int_0^t\d s\int_{\R^d} f(\d y)\
			\bm{p}_{2s(t-s)/t}\left( y - \frac st x\right)\chi_s(y)\\  \label{BB4}
		&\hskip1in\leq\frac{t}{N\pi^d} \int_0^N\d s\, \chi_{st/N}(0)\int_{\R^d}\hat{f}(\d z)\,\e^{-t\|z\|^2(1-s/N)s/N} 
		\prod_{j=1}^d\varphi(z_js).
	\end{align}
	Since,
	 $\sup_{0<r\le t}\chi_{r}(0)<\infty$ and $\lim_{N\to\infty}\chi_{st/N}(0)=0$ for all $s>0$, [see 
	Lemma \ref{lem:chi->0}],
	the equality \eqref{BB4} and the dominated convergence theorem together
	imply \eqref{goal:AV:1}. This completes the proof of the theorem when $\mathcal{R}(f)<\infty$.

         We
	now assume that $\mathcal{R}(f)=\infty$ and aim to prove \eqref{eq:AV:d>1}. 	
	Thanks to \eqref{chi>0} and Lemma \ref{lem:chi},
	$\chi_t(x) \ge \int_0^t( \bm{p}_{2s(t-s)/t}*f)(sx/t)\,\d s$.
	Therefore, \eqref{V1+2} and  Lemma \ref{lem:AV:1}
	together imply that, for every $N,t>0$,
	\begin{align*}
		N\Var(\mathcal{S}_{N,t}) &\ge
                  \frac{t}{\pi^d} \int_0^N\d s\int_{\R^d}\hat{f}(\d z)\,\e^{-t\|z\|^2(1-s/N)s/N} 
		\prod_{j=1}^d\varphi(z_js).
			\end{align*}
	Now we apply Fatou's lemma to conclude $\liminf_{N\to\infty}N\Var(\mathcal{S}_{N,t})  \geq t\mathcal{R}(f)=\infty$. 
	 This implies \eqref{goal:AV:1}.
\end{proof}

In the following, we give some criteria for the finiteness of $\mathcal{R}(f)$.

\begin{lemma}\label{lem:f:1}
         If $d=1$, $\mathcal{R}(f)=\infty$.
         If $d \geq 2$, 
         $\mathcal{R}(f)<\infty$
	is  equivalent to one of the following:
	\begin{compactenum}
	\item  $\int_0^\infty r^{-d}f([-r\,,r]^d)\,\d r<\infty$;
	\item  $\int_{\R^d} \|x\|^{1-d}\,f(\d x)<\infty$;\
	\item  $\int_{\R^d} \| z \|^{-1} \hat{f}(\d z) <\infty$.
	\end{compactenum}
\end{lemma}

\begin{proof}
         Let $d=1$. According to \eqref{f:finite}, there exists $R>0$ such that $f([-R, R])>0$. 
         Hence  by Proposition \ref{pr:R},
         \begin{align*}
         \mathcal{R}(f)\geq \frac{1}{2}\int_0^\infty r^{-1}f([-r\,,r])\,\d r \geq  f([-R\,,R]) \int_R^\infty r^{-1}\,\d r=\infty.
         \end{align*}

         Assume $d \geq 2$. By Proposition \ref{pr:R}, 
         we only need to prove that  Items 1, 2 and 3 are equivalent. 
 	Let $B_r=\{x\in\R^d:\, \|x\|\le r\}$ to see that
	$B_r\subseteq [-r\,,r]^d\subseteq B_{r\sqrt d}$, whence
	\[
		\int_0^\infty \frac{f(B_r)}{r^d}\,\d r\le
		\int_0^\infty f\left( [-r\,,r]^d\right)\frac{\d r}{r^d} \le(\sqrt d)^{d-1}\,
		\int_0^\infty \frac{f(B_r)}{r^d}\,\d r.
	\]
	This proves the equivalence of 1 and 2 since Fubini's theorem ensures that
	\[
		\int_0^\infty \frac{f(B_r)}{r^d}\,\d r = \frac 1{d-1}  \int_{\R^d}\frac{f(\d x)}{\|x\|^{d-1}}.
	\]
	
	Next, we prove the equivalence of 1 and 3. We observe that for alll $r>0$ and $z\in \R^d$,
	$$
	\prod_{j=1}^d\frac{\sin^2(rz_j)}{(rz_j)^2} =
	2^{-d}r^{-d} \left[\left(\bm{1}_{[-1, 1]^d}*\bm{1}_{[-1, 1]^d}\right)(\bullet/r)\right]\widehat{\phantom{a}}(z).
	$$
	Using the same approximation argument as in the proof of Proposition \ref{pr:R}, we have 
	\begin{align*}
	\int_0^{\infty}\d r\int_{\R^d}\hat{f}(\d z)\prod_{j=1}^{d}\frac{\sin^2(rz_j)}{(rz_j)^2}
	&=2^{-d}\int_0^{\infty}\frac{\d r}{r^d}\int_{\R^d}\hat{f}(\d z)
	\left[\left(\bm{1}_{[-1, 1]^d}*\bm{1}_{[-1, 1]^d}\right)(\bullet/r)\right]\widehat{\phantom{a}}(z)\\
	&=\pi^d\int_0^{\infty}\frac{\d r}{r^d}\int_{\R^d}f(\d z)
	\left(\bm{1}_{[-1, 1]^d}*\bm{1}_{[-1, 1]^d}\right)(z/r).
	\end{align*}
	Now we apply the inequality 
	$\bm{1}_{[-1, 1]^d} \leq \bm{1}_{[-1, 1]^d}*\bm{1}_{[-1, 1]^d} \leq 2^d\bm{1}_{[-2, 2]^d}$ and use Lemma \ref{lem:equiv} 
	to conclude the the equivalence of 1 and 3.
\end{proof}

\begin{lemma} \label{lem:equiv}
The following relation holds
\[
	\int_0^{\infty}\prod_{j=1}^{d}\frac{\sin^2(rz_j)}{(rz_j)^2} \,\d r \asymp \|z\|^{-1}.
\]
\end{lemma}
\begin{proof}
On one hand, we can write
\[
\prod_{j=1}^{d}\frac{\sin^2(rz_j)}{(rz_j)^2}  \le \prod_{j=1}^{d} (1\wedge (r|z_j|)^{-2}) \le  
 1\wedge (r  \max_{1\le j \le d}|z_j|)^{-2}  \le 1\wedge (  d^{-1/2}r \|z\|)^{-2},
\]
which implies
\[
\int_0^{\infty}\prod_{j=1}^{d}\frac{\sin^2(rz_j)}{(rz_j)^2}\,\d r  \le  \| z \|^{-1} \int_0^\infty  \d r\,  (1\wedge (   d^{-1} r^{-2})).
\]
On the other hand,
\[
\int_0^{\infty}\prod_{j=1}^{d}\frac{\sin^2(rz_j)}{(rz_j)^2}\,\d r  \ge \int_ { r\|z\| \le 1} \prod_{j=1}^{d}\frac{\sin^2(rz_j)}{(rz_j)^2} \,\d r
\ge   \| z \|^{-1} \inf_{0<|x| \le 1} \left( \frac { \sin x } {x} \right)^{2d}. \qedhere
\]
\end{proof}

\begin{remark}\label{Rieszinfty}
From item 2 of Lemma \ref{lem:f:1}, we deduce that $\mathcal{R}(f)=\infty$ if $f$ is given by 
a Riesz kernel that satisfies Dalang's condition, $\Upsilon(1)<\infty$;
i.e.,  $f(\d x)=\|x\|^{-\beta}\d x$ for some $0<\beta<d\wedge 2$.
\end{remark}

\subsection{Analysis in dimension $d=1$}

	Set $d=1$ and  repeat the computations  in the proof of Lemmas \ref{lem:chi} and \ref{lem:AV:1}  to see that
	\begin{align}\notag
		\Var(\mathcal{S}_{N,t})
		&=\frac{t}{\pi N}\int_0^{t}\frac{\d r}{r}\,
			\int_{-\infty}^\infty\d z\
			\varphi(z)
			\e^{ -  \frac{t(t-r)}{r}\frac{z^2}{N^2}}
			\hat{f}\left(\frac{tz}{Nr}\right)\\
		& \quad + \int_0^t\d s\int_{\R}f(\d y)\, \chi_s(y)
		\int_{\R^d}(I_N*\tilde{I}_N)(x)\bm{p}_{2s(t-s)/t}\left(y-\frac st x\right)\d x\nonumber\\
	        &:= V^{(1)}_N(t) +V^{(2)}_N(t),
	\end{align}
	where $\varphi$ and $\chi_s(y)$ are defined in \eqref{varphi} and \eqref{I:chi} respectively.

        \begin{lemma}\label{v2}
         For all $t>0$, $V^{(2)}_N(t)=o(\log(N)/N)$
         as $N\to\infty$.
         \end{lemma}
         \begin{proof}
         Choose and fix $\varepsilon >0$. Since $\sup_{y\in \R^d}\chi_s(y)=\chi_s(0)$, 
      we apply  Lemma  \ref{identity} and the change of variables $z\mapsto tz/(Ns)$ to see that
         \begin{align*}
         \frac{N}{\log N}V^{(2)}_N(t) 
         &\leq \frac{t}{\pi \log N}
        \int_{\R}\d z\
		\varphi(z)
         \int_0^t\d s\ \frac{\chi_s(0)}{s}\, \hat{f}\left(\frac{tz}{Ns}\right) \exp\left\{ - \frac{t(t-s)}{N^2 s}z^2\right\}
		\\
	 &\leq \frac{tf(\R)}{\pi \log N}
        \int_{\R}\d z\
		\varphi(z)
         \int_0^t\d s\ \frac{\chi_s(0)}{s} \exp\left\{ - \frac{t(t-s)}{N^2 s}z^2\right\}
		\\
	 &:= T_{2, 1} + T_{2, 2},
         \end{align*}
         where
        \begin{align*}
        T_{2, 1}&= \frac{tf(\R)}{\pi\log N}
        \int_{\R}\d z\
		\varphi(z)
          \int_0^t\frac{\d s}{s}\, \mathbf{1}_{\{s\leq tN^{-\varepsilon}\}}\chi_s(0)
		 \exp\left\{ - \frac{t(t-s)}{N^2 s}z^2\right\},
		 \\
         T_{2, 2}&= \frac{tf(\R)}{\pi\log N}
        \int_{\R}\d z\
		\varphi(z)
           \int_0^t\frac{\d s}{s}\, \mathbf{1}_{\{s> tN^{-\varepsilon}\}}\chi_s(0)
		 \exp\left\{ - \frac{t(t-s)}{N^2 s}z^2\right\}.
       \end{align*}
       By Lemma A.1 in Chen et al \cite{CKNP_c}, for all  $\varepsilon >0$,
       \begin{align*}
       T_{2, 1}& \leq \frac{7t\log_+(1/t) f(\R)}{\pi}\sup_{0\leq s\leq tN^{-\varepsilon}}\chi_s(0)
       \int_{\R}\d z\
		\varphi(z)\log_+(1/|z|). 
       \end{align*}
       Hence by Lemma \ref{lem:chi->0}, for all  $\varepsilon >0$,
       \begin{align}\label{T_{2,1}}
       \limsup_{N\to\infty}T_{2, 1}=0.
       \end{align}
       Similarly, using Theorem \ref{th:U} and the fact that  $\int_{\R} \varphi(z) \d z =\pi$, we deduce
       \begin{align}\label{T_{2,2}}
        T_{2, 2}& \leq tf(\R)\sup_{0\leq s\leq t}\chi_s(0) \frac{\log t - \log (tN^{-\varepsilon})}{\log N}.
        	 \end{align}
      Therefore, we conclude from \eqref{T_{2,1}} and  \eqref{T_{2,2}} that, for all  $\varepsilon >0$,
     \begin{align*}
     \limsup_{N\to\infty}\frac{N}{\log N}V^{(2)}_N(t) \le  tf(\R)\sup_{0\leq s\leq t}\chi_s(0)\varepsilon,
     \end{align*}
     which proves this lemma by letting $\varepsilon \to 0$.
              \end{proof}

\begin{proof}[Proof of Theorem \ref{th:AV:d=1}]
           In the case that $f=\delta_0$,  Item 1 of Theorem \ref{th:AV:d=1} was proved in Chen et al
	\cite{CKNP_c}. The same proof works for the more 
	general $f$ of the form $a\delta_0$. Therefore, we prove only
	 \eqref{lim(V):d=1} and Item 2.  
	 
	 We recall that, from Lemma \ref{identity}, 
	 \[
	  V^{(1)}_N(t)=  \frac t{N\pi}  \int_{\R} \d z\
	  \varphi(z) \int_0^t \frac {\d s} s \exp \left\{ -\frac { t(t-s) }{N^2s} z^2 \right\} \hat{f}  \left( \frac {tz}{Ns} \right).
	  \]
           Since $\hat{f}$ is maximized at $0$, 
          \begin{align*}
         V^{(1)}_N(t)&\leq \frac{t f(\R)}{\pi N}\int_{\R}\d z\
		\varphi(z)\int_0^t\, \frac{ \d s }{s}
		\exp\left\{ - \frac{t(t-s)}{N^2 s}z^2\right\}.
         \end{align*}
         Hence Lemma A.1 of \cite{CKNP_c} and Lemma \ref{v2} imply the third inequality in \eqref{lim(V):d=1}.

         On the other hand, using change of variables $s=tr N^{-2}$,
         \begin{align} \notag
         V^{(1)}_N(t)&= \frac{t}{\pi N}\int_0^{N^2}\frac{\d r}{r}\,\int_{\R}\d z\
         \varphi(z)
		\exp\left\{ -  tz^2\left[\frac{1-(r/N^2)}{r}\right]\right\}\hat{f}\left( \frac{zN}{r}\right)\\
	&=\frac{t}{\pi N} ( T_{1,1} + T_{1,2}),  \label{A1}
        \end{align}
        where
       \begin{align*}
	T_{1,1} &:= \int_0^1\frac{\d r}{r}\,		\int_{\R} \d z\
		\varphi(z)
		\exp\left\{ -  tz^2\left[\frac{1-(r/N^2)}{r}\right]\right\}\hat{f}\left( \frac{zN}{r}\right),\\
	T_{1,2} &:= \int_1^{N^2}\frac{\d r}{r}\,		\int_{\R} \d z\
		\varphi(z)
		\exp\left\{ -  tz^2\left[\frac{1-(r/N^2)}{r}\right]\right\}\hat{f}\left( \frac{zN}{r}\right).
        \end{align*}
        It is easy to see that, for all $N\ge1$ and for all $a>0$
       \begin{align*}
	\int_0^1\exp\left\{ - a\left[\frac{1-(r/N^2)}{r}\right]\right\}\frac{\d r}{r}
		&\le\int_0^1\exp\left\{ - a\left[\frac{1-r}{r}\right]\right\}\frac{\d r}{r}
	&=\e^a\int_a^\infty\frac{\e^{-s}\,\d s}{s}\le \log_+(\e/a),
        \end{align*}
        where $\log_+(x) = \log(\e +x)$ for $x>0$.
        Since $\sup\hat{f}=\hat{f}(0)=f(\R)$, it follows that
        \[
	T_{1,1} \le f(\R)
	\int_{\R} \varphi(z)\log_+\left(\frac{\e}{tz^2}\right)\d z
	<\infty.
        \]
        Therefore, 
       \begin{equation} \label{A2}
	\limsup_{N\to\infty} \frac{T_{1,1}}{\log N}=0.
        \end{equation}
         So all of the asymptotic behavior of $V^{(1)}_N(t)$ is captured via the asymptotic
         behavior of $T_{1,2}$. Now
         \begin{align*}
	T_{1,2} &= \int_{-\infty}^\infty\d z\
		\varphi(z)\int_{1/N^2}^1\frac{\d s}{s}
		\exp\left\{ -  tz^2\left[\frac{1-s}{sN^2}\right]\right\}\hat{f}\left( \frac{z}{sN}\right)\\
	&= T_{1,2,1} +T_{1,2,2},
         \end{align*}
        where
        \begin{align*}
	T_{1,2,1} &:= \int_{-\log N}^{\log N}\d z\
         \varphi(z)\int_{1/N^2}^1\frac{\d s}{s}
		\exp\left\{ -  tz^2\left[\frac{1-s}{sN^2}\right]\right\}\hat{f}\left( \frac{z}{sN}\right),\\
	T_{1,2,2} &:= \int_{|z|>\log N}\d z\
		\varphi(z)\int_{1/N^2}^1\frac{\d s}{s}
		\exp\left\{ -  \frac{tz^2}{N^2}\left[\frac{1-s}{s}\right]\right\}\hat{f}\left( \frac{z}{sN}\right).
       \end{align*}
      Now,
      \begin{equation} \label{A3}
	0\le T_{1,2,2}\le f(\R)\log(N^2)\int_{|z|>\log N}\d z\
	\left(\frac{1-\cos z}{z^2}\right) = o(\log N).
      \end{equation}
       So all of the asymptotic behavior of $V^{(1)}_N(t)$ is captured via the asymptotic
       behavior of $T_{1,2,1}$. To study that term, we rescale one more time 
       [but slightly differently from before] in order to see that
       \begin{align*}
	T_{1,2,1} & = \int_{-\log N}^{\log N}\d z\
		\varphi(z)\int_{1/N}^N\frac{\d r}{r}
		\exp\left\{ -  \frac{tz^2}{N}\left[\frac1r - \frac1N\right]\right\}
		\hat{f}\left( \frac{z}{r}\right)\\
	&\ge \exp\left\{ -  \frac{t|\log N|^2}{N} \right\}\int_{-\log N}^{\log N}\d z\
		\varphi(z)\int_1^N\frac{\d r}{r}
		\hat{f}\left( \frac{z}{r}\right).
       \end{align*}
    Hence,
       \begin{align*}
	T_{1,2,1} & \ge(1+o(1))\int_{-\log N}^{\log N}\d z\
		\left(\frac{1-\cos z}{z^2}\right)\int_{(\log N)^2}^N\frac{\d r}{r}
		\hat{f}\left( \frac{z}{r}\right)\\
	&= (f(\R)+o(1))\int_{-\log N}^{\log N}\d z\
		\left(\frac{1-\cos z}{z^2}\right)\int_{(\log N)^2}^N\frac{\d r}{r}\\
	&=(f(\R)+o(1))\int_{-\infty}^\infty
		\left(\frac{1-\cos z}{z^2}\right)\d z\ \log N\\
	&=(\pi f(\R)+o(1)) \log N.
        \end{align*}
       This proves that  
       \begin{align}\label{lower}
	 \pi f(\R)\le \liminf_{N\to\infty} \frac{1}{\log N} T_{1,2,1}.
        \end{align}
       Therefore, Lemma \ref{v2} and the relations \eqref{A1}, \eqref{A2}, \eqref{A3} and  \eqref{lower},  prove the first inequality in  \eqref{lim(V):d=1}.

        It remains to prove Item 2. We assume
	that $\hat{f}$ vanishes at infinity. Combining Lemma \ref{v2} and the above arguments,  the problem is reduced 
	to the following:
	\begin{equation}\label{goal:Rajchman}
		\limsup_{N\to\infty} \frac{T_{1,2,1}}{\log N}\le\pi f(\R).
	\end{equation}
	With this in mind,
	let us  recall from the definition of $T_{1,2,1}$ that
	\[
		T_{1,2,1} \le  \int_{-\log N}^{\log N}\varphi(z)\,\d z
		\int_{1/N}^N\frac{\d r}{r}\
		\hat{f}(z/r).
	\]
	Because
	\[
		\int_{-\log N}^{\log N}\varphi(z)\,\d z\
		\int_{(\log N)^2}^N\frac{\d r}{r}\
		\hat{f}(z/r) \le\pi f(\R)\int_{(\log N)^2}^N\frac{\d r}{r}
		\sim\pi f(\R)\log N,
	\]
	as $N\to\infty$, this and symmetry reduce our goal \eqref{goal:Rajchman} to proving that,
	when $\hat f$ vanishes at infinity,
	\[
		\int_0^{\log N}\varphi(z)\,\d z\
		\int_{1/N}^{(\log N)^2}\frac{\d r}{r}\ 
		\hat{f}(z/r)= o(\log N)
		\qquad\text{as $N\to\infty$}.
	\]
	Since $\int_{1/\sqrt{\log N}}^{(\log N)^2} r^{-1}\,\d r=o(\log N)$, we can further reduce our goal to 
	proving the following: When $\hat f$ vanishes at infinity,
	\[
		\int_0^{\log N}\varphi(z)\,\d z\
		\int_{1/N}^{1/\sqrt{\log N}} \frac{\d r}{r}\ 
		\hat{f}(z/r)= o(\log N)
		\qquad\text{as $N\to\infty$}.
	\]
	But this is so since: (1)
	\begin{equation}\label{garb}
		\int_{1/(\log N)^{1/4}}^{\log N}\varphi(z)\,\d z\
		\int_{1/N}^{1/\sqrt{\log N}} \frac{\d r}{r}\ 
		\hat{f}(z/r) \le \pi\sup_{w\ge(\log N)^{1/4}}\hat{f}(w)\log N=o(\log N);
	\end{equation}
	and (2) because $\varphi\le1$,
	\[
		\int_0^{1/(\log N)^{1/4}}\varphi(z)\,\d z\
		\int_{1/N}^{1/\sqrt{\log N}} \frac{\d r}{r}\ 
		\hat{f}(z/r) \le \frac{f(\R)}{(\log N)^{1/4}}\int_{1/N}^{1/\sqrt{\log N}} \frac{\d r}{r}
		=o(\log N).
	\]
	This proves item 2.
\end{proof}

\subsection{Analysis of Riesz kernel case}

We now aim to prove Theorem \ref{Riesz}.           
Assume $f(\d x)=\|x\|^{-\beta}\d x$  
 and $\hat{f}(\d x)=\kappa_{\beta, d}\,\|x\|^{\beta-d}\d x$, 
 where $0<\beta< 2\wedge d$ and $\kappa_{\beta, d}$
           is a positive constant depending on $\beta$ and $d$.
In this case, we first provide another supporting lemma on the behavior of $\chi_t(x)$ as $x\to\infty$. 

\begin{lemma}\label{chi:infty}
           Recall \eqref{I:chi}.
          For all $t>0$, $\lim_{x\to\infty}\chi_t(x)=0$.
\end{lemma}
\begin{proof}
          By the Poincar\'{e} inequality \eqref{Poincare:Cov}, 
          \begin{align*}
          |\chi_t(x)|&= |\Cov(U(t\,, 0)\,, U(t\,, x))|\\
          &\leq  \int_0^t\d s\int_{\R^d}f(\d y)\int_{\R^d}\d y' \,
          \|D_{s, y'}U(t\,, 0)\|_2          \|D_{s, y+y'}U(t\,, x)\|_2 \\
          &\leq    C_{t, 2}^2        \int_0^t\d s \int_{\R^d}f(\d y)\int_{\R^d}\d y' \,
          \bm{p}_{s(t-s)/s}\left(y'\right)\bm{p}_{s(t-s)/s}\left(y'+y-\frac{s}{t}x\right)\\
          &=    C_{t, 2}^2        \int_0^t\d s\int_{\R^{2}}f(\d y) \,
            \bm{p}_{2s(t-s)/s}\left(y-\frac{s}{t}x\right)= \int_0^t \d s \left(\bm{p}_{2s(t-s)/s}*f\right)\left(\frac{s}{t}x\right),
          \end{align*}
          where in the second inequality we use Proposition \ref{pr:Du} and in the first equality we use semigroup property.
          Now we apply \eqref{Fourier} to see that 
	\begin{align*}
		|\chi_t(x)|&\leq C_{t, 2}^2\int_0^t \d s \int_{\R^d}\hat{f}(\d z)\
		\exp\left\{ -\frac{s(t-s)\|z\|^2}{t} + i\left(\frac{s}{t}\right)z\cdot x\right\}\\
          &= \kappa_{\beta,d}C_{t, 2}^2\int_0^t \d s \int_{\R^d}\d z\ \|z\|^{\beta-d} 
		\exp\left\{ -\frac{s(t-s)\|z\|^2}{t} + i\left(\frac{s}{t}\right)z\cdot x\right\}.
          \end{align*}
          Since $\int_0^t \d s \int_{\R^d}\d z\, \|z\|^{\beta-d} \exp\{-s(t-s)\|z\|^2/t\}<\infty$, 
          the dominated convergence theorem and the Riemann-Lebesgue lemma together
          imply that $\lim_{x\to\infty}\chi_t(x)=0$.
\end{proof}

\begin{proof}[Proof of Theorem \ref{Riesz} part 1: $0<\beta<1$]
	Let $\psi(x):= \prod_{i=1}^d (1- |x_i|)$ for all $x\in \R^d$. We observe that 
       $(I_N*\tilde{I})(x) = N^{-d}\psi(x/N)\bm{1}_{[-N, N]^d}(x)$
        for all $x\in \R^d$. Recall \eqref{V^{(1)}_N(t)} and \eqref{V^{(2)}_N(t)}. 
        Since $f(\d x)=\|x\|^{-\beta}\d x$, 
        we can write
\begin{align*}
V^{(1)}_N(t)& = \frac 1{N^d}  \int_{[-N,N]^d} \d x\,  \psi(x/N) \int_0^t \d s  \int_{\R^d} \d y\,   \| y\|^{-\beta} \pmb{p}_{2s(t-s)/t} \left(y-\frac st x\right),\\
V^{(2)}_N(t) &= \frac 1{N^d}  \int_{[-N,N]^d} \d x\, \psi(x/N)  \int_0^t \d s  \int_{\R^d} \d y \,  \| y\|^{-\beta} \pmb{p}_{2s(t-s)/t} \left(y-\frac st x\right)
\chi_{s}(y).
\end{align*}
The term  $V^{(1)}_N(t)$ can be expressed as
	\begin{align*}
		V^{(1)}_N(t) &= \frac 1{N^d}  \int_{[-N,N]^d} \d x\
			\psi(x/N) \int_0^t \d s \,   \E\left(\left\| \sqrt{ \frac{2s(t-s)}{t}} Z - \frac st x  
			\right\| ^{-\beta}  \right)\\
		&=   N^{-\beta} \int_{[-1,1]^d} \d z\ \psi(z) \int_0^t \d s  \  
			\E\left(\left\|\frac 1N \sqrt{ \frac{2s(t-s)}{t}} Z - \frac st z  \right\| ^{-\beta}  \right),
\end{align*}
where we have made the change of variable $x=Nz$ and  $Z$ denotes  a $d$-dimensional standard normal random variable.  
An easy exercise shows that $\lim_{N\to\infty}\left(\bm{p}_{1/N}*\|\cdot\|^{-\beta}\right)(x) = \|x\|^{-\beta}$ for all $x\in \R^d\setminus \{0\}$, which implies that 
for any $s\in (0\,,t]$ and
$z\in \R^d \setminus\{ 0\}$,
\[
	\lim_{N\rightarrow \infty} \E\left(\left\|\frac 1N \sqrt{ 2s(t-s)/t} Z - \frac st z
	\right\| ^{-\beta}  \right) = t^\beta s^{-\beta} \| z\|^{-\beta}.
\]
Moveover,  according to  
Lemma 3.1 of \cite{HNVZ2019},
\[
	\sup_{N\ge 1} \E\left( \left\|\frac 1N \sqrt{ \frac{2s(t-s)}{t}} Z - \frac st z  \right\| ^{-\beta}  
	\right) \le C t^{\beta} s^{-\beta} \| z\|^{-\beta}.
\]
Because $\beta<1$, the dominated convergence theorem implies that
\[
\lim_{N\rightarrow \infty}  N^\beta V^{(1)}_N(t)=\int_{[-1,1]^d} \d z\, \psi(z) \int_0^t \d s  \
 t^\beta s^{-\beta} \| z\|^{-\beta} <\infty.
 \]
 Finally to complete the proof of \eqref{2var:beta<1} it suffices to show that
 \begin{equation} \label{ecu1}
 \lim_{N\rightarrow \infty}  N^\beta V^{(2)}_N(t)=0.
 \end{equation}
 Using the same arguments as before and recalling $\chi_s(y)$ in \eqref{I:chi}, we can write
\begin{align*}
	N^\beta V^{(2)}_N(t)  &\le  \int_{[-1,1]^d} \d z\, \psi(z) \int_0^t \d s \
 	\E \left[\chi_s\left( \sqrt{ \frac{2s(t-s)}{t}} Z  - \frac st Nz \right) 
	\left\|\frac 1N \sqrt{ 2s(t-s)/t} Z - \frac st z  \right\| ^{-\beta}  \right].
 \end{align*}
Thus, we can conclude \eqref{ecu1} from the fact  that  
$\chi_s(\sqrt{ 2s(t-s)/t} Z  - (sNz)/t)$ is uniformly bounded [Theorem \ref{th:U}] 
and converges to zero almost surely as $N\to \infty$ [Lemma \ref{chi:infty}].
\end{proof}

Before we move on to proving part 2 and part 3, we express the quantities $V^{(1)}_N(t)$ and $V^{(2)}_N(t)$ 
using $\hat{f}(\d x)= \kappa_{\beta, d} \|x\|^{\beta-d}\d x$. In fact, from \eqref{V^{(1)}_N(t)} and using the identity \eqref{equal},
we see that
\begin{align}
V^{(1)}_N(t)&=   \frac{\kappa_{\beta, d}}{\pi^d}\int_{0}^t\d s \int_{\R^d}\e^{-s(t-s)\|z\|^2/t} \prod_{j=1}^d\frac{1-\cos(Nz_js/t)}{(Nz_js/t)^2}
   \|z\|^{\beta-d}\d z\nonumber\\
   &=  \frac{\kappa_{\beta, d}}{\pi^d N^{\beta}}\int_0^t\d s\,\frac{t^\beta}{s^\beta}
		\int_{\R^d}\d z\ \|z\|^{\beta-d}
		\prod_{j=1}^d\varphi(z_j)
		\exp\left\{ - \frac{t(t-s)}{N^2 s}\|z\|^2\right\} \nonumber\\
    &=  \frac{t\,\kappa_{\beta, d}}{\pi^d N^{\beta}} \int_{\R^d}\d z\ \|z\|^{\beta-d}
		\prod_{j=1}^d\varphi(z_j)
		\int_{0}^{\infty}\d r\, (1+r)^{\beta-2} \exp\left\{- \frac{rt\|z\|^2}{N^2}\right\}
		\label{V4}\\
	&=  \frac{t^{2-\beta}\,\kappa_{\beta, d}}{\pi^d N^{2-\beta}} \int_{\R^d}\d z\ \|z\|^{2-\beta-d}
		\prod_{j=1}^d\varphi(z_j)
		\int_{0}^{\infty}\d r\, \left(\frac{t\|z\|^2}{N^2}+r\right)^{\beta-2} \e^{-r},  \label{V_N1}
\end{align}
where $\varphi$ is defined in \eqref{varphi} and we use change of variables in the last three equalities. 
Similarly, using change of variables and \eqref{V^{(2)}_N(t)}
           \begin{align}
     V^{(2)}_N(t) 
       &\leq  \frac{t\,\kappa_{\beta, d}}{\pi^d N^{\beta}} \int_{\R^d}\d z\ \|z\|^{\beta-d}
		\prod_{j=1}^d\varphi(z_j)
		\int_{0}^{\infty}\d r\, (1+r)^{\beta-2} \e^{-r\cdot \frac{t\|z\|^2}{N^2}}\chi_{t(1+r)^{-1}}(0)
		\label{V_N12}\\
	&=  \frac{t^{2-\beta}\,\kappa_{\beta, d}}{\pi^d N^{2-\beta}} \int_{\R^d}\d z\ \|z\|^{2-\beta-d}
		\prod_{j=1}^d\varphi(z_j)
		\int_{0}^{\infty}\d r\, \left(\frac{t\|z\|^2}{N^2}+r\right)^{\beta-2} \e^{-r}
		\chi_{t(1+rN^2/(t\|z\|^2))^{-1}}(0). \nonumber
	\end{align}

\begin{proof}[Proof of Theorem \ref{Riesz} part 2: $\beta=1$]
Using \eqref{V_N1} with $\beta=1$, we have
     \begin{align}\label{N/log N}
     \frac{N}{\log N}V^{(1)}_N(t) 
	&=  \frac{\kappa_{1, d}}{\pi^d } \int_{\R^d}\d z\ \|z\|^{1-d}
		\prod_{j=1}^d\varphi(z_j)
		\frac{t}{\log N}\int_{0}^{\infty}\d r\, \left(\frac{t\|z\|^2}{N^2}+r\right)^{-1} \e^{-r}.
	\end{align}
According to Lemma A.1 of Chen et al \cite{CKNP_c}, we have
	\begin{align}\label{A.1}
	\frac{t}{\log N}\int_{0}^{\infty}\d r\, \left(\frac{t\|z\|^2}{N^2}+r\right)^{-1} \e^{-r}&=  \frac{t}{\log N}\int_0^t \exp\left( - \frac{(t - s)t}{s}\cdot \frac{\|z\|^2}{N^2}\right) \nonumber
	\,\frac{\d s}{s}\\
	&\leq  7t\log_+(1/t) \log_+(1/\|z\|) \quad \text{for all $N\geq \e$},
	\end{align}
    where $\log_+(a)=\log(\e +a)$ for $a>0$, and 
    \begin{align}\label{A.1:lim}
	\lim_{N\to\infty}\frac{t}{\log N}\int_{0}^{\infty}\d r\, \left(\frac{t\|z\|^2}{N^2}+r\right)^{-1} \e^{-r} =2t, 
	\qquad \text{for all $z\in \R^d\setminus \{0\}$}.
   \end{align}
   Therefore, since  
   $\int_{\R^d} \|z\|^{1-d}\prod_{j=1}^d\varphi(z_j)  \log_+(1/\|z\|) \d z<\infty$,     
   by \eqref{N/log N}--\eqref{A.1:lim} and  the dominated convergence theorem, 
   \begin{align*}
    \lim_{N\to\infty} \frac{N}{\log N}V^{(1)}_N(t) =  \frac{2t\, \kappa_{1, d}}{\pi^d }
     \int_{\R^d} \|z\|^{1-d}\prod_{j=1}^d\varphi(z_j)\d z.
   \end{align*}
      In light of \eqref{var:beta=1},  it suffices to prove 
       \begin{align}
          \lim_{N\to\infty}\frac{N}{\log N}V^{(2)}_N(t)&=0. \label{beta=113}
       \end{align}
  	Similarly, letting $\beta=1$ in \eqref{V_N12}, 
    \begin{align}\label{V2N/log N}
     \frac{N}{\log N}V^{(2)}_N(t) 
	&\leq  \frac{\kappa_{1, d}}{\pi^d } \int_{\R^d}\d z\ \|z\|^{1-d}
		\prod_{j=1}^d\varphi(z_j)
 \frac{t}{\log N}\int_{0}^{\infty}\d r\, \left(\frac{t\|z\|^2}{N^2}+r\right)^{-1} \e^{-r}
		\chi_{t(1+rN^2/(t\|z\|^2))^{-1}}(0).
	\end{align}
        Choose and fix $0<\varepsilon<2$. 
       We see from \eqref{V2N/log N} and \eqref{A.1} that 
     \begin{align}
     \frac{N}{\log N}V^{(2)}_N(t) 
	&\leq  \frac{\kappa_{1, d}}{\pi^d }\sup_{0\leq s\leq t}\chi_s(0)
	 \int_{\R^d}\d z\ \|z\|^{1-d}
		\prod_{j=1}^d\varphi(z_j)
		\frac{t}{\log N}\int_{0}^{N^{-\varepsilon}}\d r\, \left(\frac{t\|z\|^2}{N^2}+r\right)^{-1} \nonumber\\
	& \quad +   \frac{\kappa_{1, d}}{\pi^d } 7t\log_+(1/t) \int_{\R^d}\d z\ \|z\|^{1-d}
		\prod_{j=1}^d\varphi(z_j) \log_+(1/\|z\|) 
		\sup_{0\leq s\leq t^2\|z\|^2/N^{2-\varepsilon}}\chi_s(0)\nonumber
	\end{align}
	Letting $N\to\infty$ and using Lemma \ref{lem:chi->0} and  the dominated convergence theorem, we conclude that  
	for every $0<\varepsilon<2$, 
	 \begin{align*}
    \limsup_{N\to\infty} \frac{N}{\log N}V^{(2)}_N(t)&\leq  \frac{t(2-\varepsilon)\kappa_{1, d}}{\pi^d }\sup_{0\leq s\leq t}\chi_s(0)
	 \int_{\R^d}\d z\ \|z\|^{1-d}
		\prod_{j=1}^d\varphi(z_j).
	\end{align*}
	Since the choice of  $0<\varepsilon<2$ is arbitrary, we  let $\varepsilon\to 2$ to obtain \eqref{beta=113}. 
	This proves \eqref{var:beta=1}.
\end{proof}

\begin{proof}[Proof of Theorem \ref{Riesz} part 3: $1<\beta<2$]
    Recall \eqref{V_N1}. Under the condition $1<\beta<2$, 
      we    have $ \int_{\R^d}\|z\|^{2-\beta-d}\prod_{j=1}^d\varphi(z_j)\d z <\infty$ 
      and $ \int_{0}^{\infty} r^{\beta-2} \e^{-r}\d r <\infty$. Hence by the  dominated convergence theorem,  
      \begin{align}
       \lim_{N\to\infty}N^{2-\beta}V^{(1)}_N(t)=  \frac{t^{2-\beta}\,\kappa_{\beta, d}}{\pi^d}
        \int_{\R^d}\|z\|^{2-\beta-d}\prod_{j=1}^d\varphi(z_j)\d z \int_{0}^{\infty} r^{\beta-2} \e^{-r}\d r. \label{beta>11}
      \end{align} 
      Moreover, from  \eqref{V_N12}, Lemma \ref{lem:chi->0} and the dominated convergence theorem
	  \begin{align*}
      &\limsup_{N\to\infty}N^{2-\beta}V^{(2)}_N(t)  \\
	&\quad \leq  \frac{t^{2-\beta}\,\kappa_{\beta, d}}{\pi^d } \int_{\R^d}\d z\ \|z\|^{2-\beta-d}
		\prod_{j=1}^d\varphi(z_j)
		\int_{0}^{\infty}\d r\, \e^{-r} \lim_{N\to\infty}
		\left(\frac{t\|z\|^2}{N^2}+r\right)^{\beta-2} 
		\chi_{t(1+rN^2/(t\|z\|^2))^{-1}}(0)\\
		&\quad =0,
	\end{align*}
       which together with \eqref{beta>11} proves \eqref{var:beta>1}.
     \end{proof}

\section{Total variation distance}
\label{Sec:TVD}

In this section, we will estimate the total variation distance and prove Theorems  \ref{TVD3}-\ref{TVD2}.

We recall that
\[
\mathcal{S}_{N,t}=\frac {1}{N^d} \int_{[0, N]^d} [ U(t\,, x) -1]\,
\d x\quad  \text{and} \quad V_{N}(t) = {\rm Var}(\mathcal{S}_{N,t}).
\]
 We can estimate the total variation distance between the normalized random variable
\[
\widetilde{ \mathcal{S}}_{N,t} :=\mathcal{S}_{N,t} /\sqrt{V_N(t)}
\]
 and a ${\rm N}(0\,,1)$ random variable $\Z$ using  the inequality (\ref{SM1}).
 According to the inequality (\ref{SM1}), we need to express the random variable  $ \widetilde{ \mathcal{S}}_{N,t} $ as a divergence, or, as an It\^o-Walsh stochastic integral. From equation \eqref{mild:U} we obtain  $\widetilde{ \mathcal{S}}_{N,t}=V_N(t)^{-1/2} \delta(v_N)$, where
 \begin{equation} \label{ec1}
 v_N(s,y)=  \frac {1}{N^d}  U(s\,,y) \int_{[0, N]^d} \bm{p}_{s(t-s)/t}\left(
	y - \frac st x \right)  \d x.
 \end{equation}
   In this way,   inequality (\ref{SM1}) yields
\begin{equation} \label{dTV2}
d_{\rm TV} ( \widetilde {\mathcal{S}}_{N,t} , Z) \le \frac 2 {V_N(t)}
\sqrt{ {\rm Var} \left(\langle D\mathcal{S}_{N,t} , v_N \rangle_{\HH} \right)}.
\end{equation}

The Malliavin derivative of $\mathcal{S}_{N,t}$ can be computed as follows
\begin{align} \notag
	D_{s,y} \mathcal{S}_{N,t} & =  \frac {1}{N^d} \left( \int_{[0, N]^d} \bm{p}_{s(t-s)/t}\left(
	y - \frac st x \right) \d x \right)U(s\,,y) \\
				  &\qquad + \frac {1}{N^d} \int_{(s,t)\times\R^d} 
				  \left( \int_{[0, N]^d} \bm{p}_{r(t-r)/t}\left(
	w - \frac st x \right) \d x \right) D_{s,y}U(r\,,w)\,\eta( \d r, \d w).   \label{ec22}
 \end{align}
 From (\ref{ec1}) and (\ref{ec22}), we obtain
 \begin{align} \notag
 \langle D\mathcal{S}_{N,t} \,, v_N \rangle_{\HH} &=  
\frac 1{N^{2d}}\int_0^t \d s \int_{\R^{2d}}\ f(\d z)\d y\int_{[0, N]^{2d}}\d x\d x' \,  \nonumber \\
 & \qquad \qquad \qquad
\bm{p}_{s(t-s)/t}\left(
	y - \frac st x \right)  \bm{p}_{s(t-s)/t}\left(
	y+z - \frac st x \right) U(s\,, y)U(s\,, y+z) \nonumber\\  
&\qquad + \frac 1{N^{2d}} \int_{0}^{t} \int_{\R^d} \, \eta(\d r\,, \d w)  \int_0^r \d s  \int_{\R^{2d}}\ f(\d z)\d y\int_{[0, N]^{2d}}\d x\d x' \nonumber \\
& \qquad \qquad \quad \bm{p}_{r(t-r)/t}\left(
	w - \frac rt x \right)  \bm{p}_{s(t-s)/t}\left(
	y+z - \frac st x' \right) U(s\,, y+z)D_{s, y}U(r\,,w), \nonumber
   \end{align}
   where we use stochastic Fubini's theorem in the second equality.
     As a consequence,
    \begin{equation}    
    {\rm Var} \left( \langle D\mathcal{S}_{N,t} , v_N \rangle_{\HH} \right)   \leq \frac 2{N^{4d}}( \Phi_N^{(1)} + \Phi_N^{(2)}),  \label{1+2}
    \end{equation}
    where 
    \begin{align*}
    \Phi_N^{(1)}&=  \int_{[0,t]^2}\d s_1\d s_2\int_{\R^{4d}} \, f(\d z_1)f(\d z_2)\d y_1   \d y_2 \int_{[0, N]^{4d}}\d x_1\d x_1' \d x_2\d x_2' \, \bm{p}_{s_1(t-s_1)/t}\left(
	y_1 - \frac{s_1}{t} x_1 \right)  \nonumber\\ 
   & \qquad \quad \times  \bm{p}_{s_1(t-s_1)/t}\left(
	y_1+z_1 - \frac{s_1}{t}x_1' \right) \bm{p}_{s_2(t-s_2)/t}\left(
	y_2 - \frac{s_2}{t} x_2 \right)  \bm{p}_{s_2(t-s_2)/t}\left(
	y_2+z_2 - \frac{s_2}{t}x_2' \right)\nonumber \\
& \qquad \quad\times \Cov\left(U(s_1\,, y_1)U(s_1\,, y_1+z_1)\,, U(s_2\,, y_2)U(s_2\,, y_2+z_2)\right),
    \end{align*}
    and
    \begin{align*}
    \Phi_N^{(2)}&=\int_0^t \d r\int_{[0, r]^2}\d s_1\d s_2\int_{\R^{6d}} f(\d b)\d w f(\d z_1)\d y_1f(\d z_2)\d y_2\int_{[0, N]^{4d}}\d x_1\d x_1' \d x_2\d x_2' \, \nonumber\\
  &\quad \quad \times  \bm{p}_{r(t-r)/t}\left(
	w - \frac{r}{t}x_1 \right) \bm{p}_{s_1(t-s_1)/t}\left(
	y_1+z_1 - \frac{s_1}{t}x_1' \right)\\
	&\quad \quad \times \bm{p}_{r(t-r)/t}\left(
	w+b - \frac{r}{t}x_2 \right)\bm{p}_{s_2(t-s_2)/t}\left(
	y_2+z_2 - \frac{s_2}{t}x_2' \right)\\
	& \quad \quad \times \E\left[U(s_1\,, y_1+z_1)D_{s_1, y_1}U(r\,, w)U(s_2\,, y_2+z_2)D_{s_2, y_2}U(r\,, w+b)\right].
    \end{align*}
We are going to estimate the terms  $\Phi_N^{(1)}$ and $\Phi_N^{(2)}$. 
  Using the Poincar\'e inequality (\ref{Poincare:Cov}), we can write
    \begin{align*}
    \Phi_N^{(1)} & \le  \int_{[0,t]^2}\d s_1\d s_2\int_0^{s_1\wedge s_2}\d r\int_{\R^{6d}} \, f(\d z_1)f(\d z_2)f(\d b)\d a\d y_1   \d y_2 \int_{[0, N]^{4d}}\d x_1\d x_1' \d x_2\d x_2' \,\\ 
   & \quad \quad \times \bm{p}_{s_1(t-s_1)/t}\left(
	y_1 - \frac{s_1}{t} x_1 \right)  \bm{p}_{s_1(t-s_1)/t}\left(
	y_1+z_1 - \frac{s_1}{t}x_1' \right) \\
 & \quad \quad	\times\bm{p}_{s_2(t-s_2)/t}\left(
	y_2 - \frac{s_2}{t} x_2 \right)  \bm{p}_{s_2(t-s_2)/t}\left(
	y_2+z_2 - \frac{s_2}{t}x_2' \right) \\
& \quad \quad\times   \Big(\left\|D_{r, a}U(s_1\,, y_1)\right\|_4\left\|U(s_1\,, y_1+z_1)\right\|_4 + \left\|U(s_1\,, y_1)\right\|_4\left\|D_{r, a}U(s_1\,, y_1+z_1)\right\|_4
\Big)\\
&\quad\quad \times \Big(\left\|D_{r, a+b}U(s_2\,, y_2)\right\|_4\left\|U(s_2\,, y_2+z_2)\right\|_4 + \left\|U(s_2\,, y_2)\right\|_4\left\|D_{r, a+b}U(s_2\,, y_2+z_2)\right\|_4
\Big).
 \end{align*}
 The estimates \eqref{moment:U})and \eqref{C_{t, k}}  and semi-group property yield
       \begin{align}
    \Phi_N^{(1)}
 & \le 4  C_{t,4}^2 c_{t,4}^2  \int_{[0,t]^2}\d s_1\d s_2\int_0^{s_1\wedge s_2}\d r\int_{\R^{5d}} \, f(\d z_1)f(\d z_2)f(\d b)\d y_1   \d y_2 \int_{[0, N]^{4d}}\d x_1\d x_1' \d x_2\d x_2' \,\nonumber\\
 &\quad \times \bm{p}_{s_1(t-s_1)/t}\left(
	y_1 - \frac{s_1}{t} x_1 \right) \nonumber\bm{p}_{s_1(t-s_1)/t}\left(
	y_1+z_1 - \frac{s_1}{t}x_1' \right) \\
&\quad\times	\bm{p}_{s_2(t-s_2)/t}\left(
	y_2 - \frac{s_2}{t} x_2 \right)  \bm{p}_{s_2(t-s_2)/t}\left(
	y_2+z_2 - \frac{s_2}{t}x_2' \right)  \nonumber\\
&  \quad\times   \Bigg[ \bm{p}_{r(s_1- r)/s_1 + r(s_2- r)/s_2}\left(
	b - \frac{r}{s_2}y_2 + \frac{r}{s_1}y_1 \right) 
	 + \bm{p}_{r(s_1- r)/s_1 + r(s_2- r)/s_2}\left(
	b - \frac{r}{s_2}(y_2+z_2) + \frac{r}{s_1}y_1 \right)  \nonumber \\
&\qquad  \qquad  + \bm{p}_{r(s_1- r)/s_1 + r(s_2- r)/s_2}\left(
	b - \frac{r}{s_2}y_2 + \frac{r}{s_1}(y_1 +z_1)\right)  \nonumber \\
& \qquad \qquad  + \bm{p}_{r(s_1- r)/s_1 + r(s_2- r)/s_2}\left(
	b - \frac{r}{s_2}(y_2+z_2) + \frac{r}{s_1}(y_1+z_1) \right)  \Bigg]. \nonumber
\end{align}
By symmetry, we conclude that
\begin{align}	
    \Phi_N^{(1)}
 & \le16  C_{t,4}^2 c_{t,4}^2   \int_{[0,t]^2}\d s_1\d s_2 \int_0^{s_1\wedge s_2}\d r\int_{\R^{5d}} \, f(\d z_1)f(\d z_2)f(\d b) 
 \d y_1   \d y_2 \int_{[0, N]^{4d}}\d x_1\d x_1' \d x_2\d x_2' \, \nonumber\\
&\quad\quad \times \bm{p}_{s_1(t-s_1)/t}\left(
	y_1 - \frac{s_1}{t} x_1 \right)   \bm{p}_{s_1(t-s_1)/t}\left(
	y_1+z_1 - \frac{s_1}{t}x_1' \right)\nonumber\\
&\quad \quad \times \bm{p}_{s_2(t-s_2)/t}\left(
	y_2 - \frac{s_2}{t} x_2 \right)  \bm{p}_{s_2(t-s_2)/t}\left(
	y_2+z_2 - \frac{s_2}{t}x_2' \right)\nonumber \\
& \quad \quad\times \bm{p}_{r(s_1- r)/s_1 + r(s_2- r)/s_2}\left(
	b - \frac{r}{s_2}y_2 + \frac{r}{s_1}y_1 \right). \label{1,1}
  \end{align}

As for $\Phi_N^{(2)}$, similarly, by Cauchy-Schwarz inequality and the  estimates \eqref{moment:U})and \eqref{C_{t, k}}, one sees that
\begin{align}
  \Phi_N^{(2)} &\leq C_{t,4}^2 c_{t,4}^2\int_0^t \d r\int_{[0, r]^2}\d s_1\d s_2\int_{\R^{6d}} f(\d b)\d w f(\d z_1)\d y_1f(\d z_2)\d y_2\int_{[0, N]^{4d}}\d x_1\d x_1' \d x_2\d x_2' \, \nonumber\\
  & \qquad \quad\times   \bm{p}_{r(t-r)/t}\left(
	w - \frac{r}{t}x_1 \right) \bm{p}_{s_1(t-s_1)/t}\left(
	y_1+z_1 - \frac{s_1}{t}x_1' \right) \bm{p}_{r(t-r)/t}\left(
	w+b - \frac{r}{t}x_2 \right)  \nonumber\\
	& \qquad \quad \times \bm{p}_{s_2(t-s_2)/t}\left(
	y_2+z_2 - \frac{s_2}{t}x_2' \right) \bm{p}_{s_1(r-s_1)/r}\left(
	y_1 - \frac{s_1}{r}w \right)\bm{p}_{s_2(r-s_2)/r}\left(
	y_2- \frac{s_2}{r}(w+b) \right) \nonumber\\
	&= C_{t,4}^2 c_{t,4}^2\int_0^t \d r\int_{[0, r]^2}\d s_1\d s_2\int_{\R^{4d}} f(\d b)\d w f(\d z_1)f(\d z_2)\int_{[0, N]^{4d}}\d x_1\d x_1' \d x_2\d x_2' \, \nonumber \\
  & \qquad \quad\times   \bm{p}_{r(t-r)/t}\left(
	w - \frac{r}{t}x_1 \right) \bm{p}_{s_1(t-s_1)/t + s_1(r-s_1)/r}\left(
	z_1- \frac{s_1}{t}x_1' + \frac{s_1}{r}w \right) \nonumber\\
	& \qquad \quad \times  \bm{p}_{r(t-r)/t}\left(
	w+b - \frac{r}{t}x_2 \right)  \bm{p}_{s_2(t-s_2)/t + s_2(r-s_2)/r}\left(
	z_2 - \frac{s_2}{t}x_2' +  \frac{s_2}{r}(w+b) \right), \label{N2}
			  \end{align}
where we use semi-group property in the equality.

In the following, we will prove Theorems  \ref{TVD3}-\ref{TVD2} separately. The identity below will be used several times later on:
 \begin{align}\label{scale}
\bm{p}_t(\sigma x) = \sigma^{-d}\bm{p}_{t/\sigma^2}(x), \quad \text{for all $x \in \R^d$ and $t, \,\sigma>0$}.
\end{align}

\subsection{Proof of Theorem \ref{TVD3}}
\begin{proof}[Proof of Theorem \ref{TVD3}]
With the notation introduced in \eqref{1+2} and according to Theorem \ref{th:AV:d>1}, it suffices to show that
\begin{align}\label{1and2}
 N^{-4d+3}\left(\Phi_N^{(1)}+ \Phi_N^{(1)}\right)  \le C,
\end{align}
 for all $N\ge \e$ and for some constant $C$ depending on $t$.

We will start with the expression for  $\Phi_N^{(1)}$ given in \eqref{1,1}.  Using the elementary relation
\begin{equation} \label{E3}
\pmb{p}_\sigma(x)  \pmb{p}_\sigma(y)= 2^d \pmb{p}_{2\sigma}(x+y)  \pmb{p}_{2\sigma}(x-y) , \qquad \sigma>0, \, x,y \in \R^d,
\end{equation}
we can write
\begin{align*}
&\bm{p}_{s_1(t-s_1)/t}\left(	y_1 - \frac{s_1}{t} x_1 \right)   
   \bm{p}_{s_1(t-s_1)/t}\left(
	y_1+z_1 - \frac{s_1}{t}x_1' \right) \\
	&\qquad=
	 2^d \bm{p}_{2s_1(t-s_1)/t}\left(2y_1 +z_1- \frac{s_1}{t} (x_1+x'_1) \right)   
   \bm{p}_{2s_1(t-s_1)/t}\left(
	 z_1 - \frac{s_1}{t}(x_1'-x_1) \right)\\
	 &\qquad = \bm{p}_{s_1(t-s_1)/(2t)}\left(y_1 +\frac {z_1}2- \frac{s_1}{2t} (x_1+x'_1) \right)   
   \bm{p}_{2s_1(t-s_1)/t}\left(
	 z_1 - \frac{s_1}{t}(x_1'-x_1) \right),
	\end{align*}
	where in the second equality we used the scaling property \eqref{scale}. In the same way, we obtain
	 \begin{align*}
&\bm{p}_{s_2(t-s_2)/t}\left(	y_2 - \frac{s_2}{t} x_2 \right)   
   \bm{p}_{s_2(t-s_2)/t}\left(
	y_2+z_2 - \frac{s_2}{t}x_2' \right) \\
	 &\qquad =\bm{p}_{s_2(t-s_2)/(2t)}\left(y_2 +\frac {z_2}2- \frac{s_2}{2t} (x_2+x'_2) \right)   
   \bm{p}_{2s_2(t-s_2)/t}\left(
	 z_2 - \frac{s_2}{t}(x_2'-x_2) \right).
	\end{align*}
	Therefore,
	\begin{align*}
	 \mathcal{L} &:=\int_{\R^{2d}} \d  y_1 \d  y_2\ \bm{p}_{s_1(t-s_1)/t}\left(	y_1 - \frac{s_1}{t} x_1 \right)   
   \bm{p}_{s_1(t-s_1)/t}\left(
	y_1+z_1 - \frac{s_1}{t}x_1' \right)  \bm{p}_{s_2(t-s_2)/t}\left(	y_2 - \frac{s_2}{t} x_2 \right)   \\
	& \qquad \times   \bm{p}_{s_2(t-s_2)/t}\left(
	y_2+z_2 - \frac{s_2}{t}x_2' \right)
	 \bm{p}_{r(s_1- r)/s_1 + r(s_2- r)/s_2}\left(
	b - \frac{r}{s_2}y_2 + \frac{r}{s_1}y_1 \right) \\
	&= \left( \frac {s_1}r\right)^d\bm{p}_{s_1(t-s_1)/t}\left(
	 z_1 - \frac{s_1}{t}(x_1'-x_1) \right)     \bm{p}_{s_2(t-s_2)/t}\left(
	 z_2 - \frac{s_2}{t}(x_2'-x_2) \right) \\
	 & \qquad \times  \int_{\R^{2d}} \d  y_1 \d  y_2 \
	 \bm{p}_{s_1(t-s_1)/(2t)}\left(y_1 +\frac {z_1}2- \frac{s_1}{2t} (x_1+x'_1) \right)   
	  \bm{p}_{s_2(t-s_2)/(2t)}\left(y_2 +\frac {z_2}2- \frac{s_2}{2t} (x_2+x'_2) \right)\\
	  & \qquad \times   \bm{p}_{(s_1/r)^2[r(s_1- r)/s_1 + r(s_2- r)/s_2]}\left(
	\frac {s_1}rb - \frac{s_1}{s_2}y_2 + y_1 \right).
	\end{align*}
	With the notation
	\[
	\mathcal{M}=\bm{p}_{2s_1(t-s_1)/t}\left(
	 z_1 - \frac{s_1}{t}(x_1'-x_1) \right)     \bm{p}_{2s_2(t-s_2)/t}\left(
	 z_2 - \frac{s_2}{t}(x_2'-x_2) \right)
	 \]
	 integrating in $y_1$ and using the semigroup property, yields
	 \begin{align*}
	 \mathcal{L} & =\left( \frac {s_1}r\right)^d \mathcal{M} 
	 \int_{\R^d}  \d y_2\ \bm{p}_{s_2(t-s_2)/(2t)}\left(y_2 +\frac {z_2}2- \frac{s_2}{2t} (x_2+x'_2) \right)\\
	 & \qquad \times   \bm{p}_{ s_1(t-s_1)/(2t) + (s_1/r)^2[r(s_1- r)/s_1 + r(s_2- r)/s_2]}\left(
	\frac {s_1}rb - \frac{s_1}{s_2}y_2  -\frac {z_1}2 + \frac {s_1}{2t} (x_1+ x'_1) \right) \\
	&=\left( \frac {s_2} r\right)^d \mathcal{M} 
	 \int_{\R^d}  \d y_2\ \bm{p}_{s_2(t-s_2)/(2t)}\left(y_2 +\frac {z_2}2- \frac{s_2}{2t} (x_2+x'_2) \right)\\
	 & \qquad \times   \bm{p}_{ (s_2/s_1)^2 \{s_1(t-s_1)/(2t) + (s_1/r)^2[r(s_1- r)/s_1 + r(s_2- r)/s_2]\}}\left(
	\frac {s_2}rb - y_2  -\frac {s_2}{ 2s_1}z_1 + \frac {s_2}{2t} (x_1+ x'_1) \right),
	 \end{align*}
	 where in the second equality we used the scaling property \eqref{scale}.
	 Integrating in $y_2$ and using the semigroup property we finally get
\[
	 \mathcal{L}= \left(\frac {s_2} r\right)^d  \mathcal{M}\,
	  \bm{p}_{  \alpha_1}\left(
	\frac {s_2}rb   -\frac {s_2}{ 2s_1}z_1 + \frac {s_2}{2t} (x_1+ x'_1-x_2-x_2') + \frac {z_2}2 \right),
	\]
	where
	\[
	\alpha_1=\frac {s_2(t-s_2)}{2t}+ 
	 \left( \frac {s_2}{s_1} \right)^2 \left\{  \frac {s_1(t-s_1)}{2t} + \left(\frac{s_1}{r} \right)^2\left [ \frac {r(s_1- r)}{s_1} + \frac {r(s_2- r)}{s_2}\right ] \right\}.
	\]
	A further application of the scaling property \eqref{scale} yields
	\[
	 \mathcal{L}=    \mathcal{M} \,
	  \bm{p}_{  \alpha_2}\left(
	b   -\frac {r}{ 2s_1}z_1 + \frac r{{\color{black}2}s_2}  z_2+ \frac {r}{2t} (x_1+ x'_1-x_2 -x_2') \right),
	\]
	where
	\[
	\alpha_2= \left( \frac r{s_2} \right)^2 \alpha_1 = \frac{ r^2(t-s_2)}{2ts_2} + \frac{ r^2(t-s_1)}{2ts_1}  + \frac {r(s_1-r)}{s_1} + \frac {r(s_2-r)}{s_2}.
	\]
	Making the change of variables $x_i\to Nx_i$ we obtain
	\begin{align*}
	N^{-4d+3}\Phi_N^{(1)}&\leq N^3
	16  C_{t,4}^2 c_{t,4}^2  \int_{[0, 1]^{4d}}\d x_1\d x_1' \d x_2\d x_2' \int_{\R^{3d}}  f(\d z_1) f(\d z_2) f(\d b) 
	\int_{[0,t]^2} \d s_1 \d s_2 \int_0^{s_1\wedge s_2} \d r  \\
	&  \quad \times  \bm{p}_{2s_1(t-s_1)/t}\left(
	 z_1 - \frac{Ns_1}{t}(x_1'-x_1) \right)     \bm{p}_{2s_2(t-s_2)/t}\left(
	 z_2 - \frac{Ns_2}{t}(x_2'-x_2) \right)\\
	 & \qquad \times
  \bm{p}_{  \alpha_2}\left(
	b   -\frac {r}{ 2s_1}z_1 + \frac r{2s_2}  z_2+ \frac {Nr}{2t} (x_1+ x'_1-x_2 -x_2') \right).
  \end{align*}
	With a further change of variables $s_1= \frac t{N}r_1$, $s_2=\frac t{N}r_2$,  $r=\frac tN \sigma$, we can write
	\begin{align*}
	N^{-4d+3}\Phi_N^{(1)}&= 
	16  C_{t,4}^2 c_{t,4}^2   t^3 \int_{[0, 1]^{4d}}\d x_1\d x_1' \d x_2\d x_2' \int_{\R^{3d}}  f(\d z_1) f(\d z_2) f(\d b) 
	\int_{[0, N]^2} \d r_1 \d r_2  \\
		&  \qquad \times    
  \bm{p}_{\frac {2tr_1}{N} (1-\frac {r_1}N  )}\left(
	 z_1 - r_1(x_1'-x_1)  \right)     \bm{p}_{\frac {2 t r_2}{N} (1-\frac { r_1}N )}\left(
	 z_2 - r_2 (x_2'-x_2)  \right)\\
	 & \qquad \times  \int_0^{ r_1\wedge r_2  } \d \sigma\
  \bm{p}_{  \gamma_{3,N}}\left(
	b   -\frac {\sigma}{ 2r_1}z_1 + \frac {\sigma}{2r_2}  z_2+  \frac {\sigma}{2}(x_1+ x'_1-x_2 -x_2') \right),
  \end{align*}
  where
  \[
  \gamma_{3,N}=   \frac {t \sigma^2}{2N} \left( \frac 1 {r_1}+ \frac 1{r_2} -\frac 2 N \right)+ \frac {t\sigma}N \left( 2- \frac {\sigma}{r_1} -\frac { \sigma}{r_2} \right).
  \]
  We also set
  \[
  \gamma_{1,N}= \frac {2tr_1}{N} (1-\frac {r_1}N  ), \qquad    \gamma_{2,N}= \frac {2tr_2}{N} (1-\frac {r_2}N  ).
  \]
With the notation $y_1= r_1(x_1'-x_1)$,  $y_2=r_2 (x_2'-x_2) $, $y_3=  \frac {\sigma}2(x_1+ x'_1-x_2 -x_2')$, the Fourier transform of the function
	\[
\Psi_1(z_1,z_2,b) := \bm{p}_{ \gamma_{1,N}}\left(
	 z_1 -y_1   \right)     \bm{p}_{\gamma_{2,N}}\left(
	 z_2 -  y_2\right)    
  \bm{p}_{  \gamma_{3,N}}\left(
	b   -\frac {\sigma} {2r_1}z_1 + \frac  {\sigma}{2 r_2} z_2+ y_3 \right)
	\]
	is given by
	\begin{align*}
	\hat{\Psi} _1 (\xi_1, \xi_2, \xi_3) &= \exp \left(-\frac {\gamma_{1,N}}2 \| \xi_1 + \frac  {\sigma} {2r_1}  \xi_3\|^2 - \frac  {\gamma_{2,N}}2  \| \xi_2 - \frac {\sigma} {2 r_2} \xi_3\|^2 -\frac  {\gamma_{3,N}}2 \| \xi_3 \|^2\right) \\
	& \qquad \times \exp\left( i \left(\xi_1+\frac {\sigma} {2r_1} \xi_3 \right)  \cdot y_1 + i  \left(\xi_2 -\frac {\sigma} {2 r_2} \xi_3 \right) \cdot y_2 - i \xi_3 \cdot y_3  \right).
	\end{align*}
	Notice that
	\begin{align*}
	&\left(\xi_1+\frac {\sigma} {2r_1} \xi_3 \right)  \cdot y_1 +   \left(\xi_2 -\frac {\sigma} {2r_2} \xi_3 \right) \cdot y_2 -  \xi_3 \cdot y_3   \\
		&{\color{black}= -  x_1 \cdot   (r_1\xi_1+ \sigma\xi_3) -x_2 \cdot    (r_2 \xi_2 -  \sigma\xi_3 ) + x_1' \cdot \left( r_1\xi_1 \right) 
	+x_2' \cdot \left( r_2\xi_2\right).}
	\end{align*}
Set
\begin{align*}
\Delta_1( \xi_1, \xi_2,\xi_3)&:= \int_{[0, 1]^{4d}}\d x_1\d x_1' \d x_2\d x_2' \\
&  \qquad {\color{black} \exp\left(i \left( -  x_1 \cdot   (r_1\xi_1+ \sigma\xi_3) -x_2 \cdot    (r_2 \xi_2 -  \sigma\xi_3 ) + x_1' \cdot \left( r_1\xi_1 \right) 
	+x_2' \cdot \left( r_2\xi_2\right) \right)\right).}
	\end{align*}
	Then, Parseval's identity  implies that 
	\begin{align*}
	N^{-4d+3}\Phi_N^{(1)}&\leq
	16  C_{t,4}^2 c_{t,4}^2   t^3   \frac 1 { (2\pi)^{3d}} \int_{[0, N]^2} \d r_1 \d r_2   \int_0^{ r_1\wedge r_2  } \d \sigma
	 \int_{\R^{3d}}  \hat{f}(\d \xi_1) \hat{f}(\d \xi_2) \hat{f}(\d \xi_3)   
	 \Delta_1( \xi_1, \xi_2,\xi_3)\\
	 & \qquad \times \exp \left(-\frac {\gamma_{1,N}}2 \| \xi_1 + \frac  {\sigma} {2r_1}  \xi_3\|^2 - \frac  {\gamma_{2,N}}2  \| \xi_2 - \frac {\sigma} {2 r_2} \xi_3\|^2 -\frac  {\gamma_{3,N}}2 \| \xi_3 \|^2\right) \\
	& \le C \int_{[ 0, \infty)^3}  \d \sigma   \d r_1 \d r_2    \int_{\R^{3d}}  \hat{f}(\d \xi_1) \hat{f}(\d \xi_2) \hat{f}(\d \xi_3)
	|\Delta_1( \xi_1, \xi_2,\xi_3)|.
\end{align*}
Taking into account that $\mathcal{R}(f) < \infty$, which is equivalent by Lemma \ref{lem:f:1} to $\int_{\R^d} \| z \|^{-1} \hat{f}(\d z)<\infty$, it suffices to show that
 \begin{equation} \label{BBB1}
 \int_{[ 0, \infty)^3}  \d \sigma   \d r_1 \d r_2     |\Delta_1( \xi_1, \xi_2,\xi_3)| \le  C(\| \xi_1 \| \| \xi_2 \| \| \xi_3 \|)^{ -1},
  \end{equation}
  for some constant $C$  not depending on $t$.
   We have
   \begin{align*}
   |\Delta_1( \xi_1, \xi_2,\xi_3)| & =  \prod_{j=1}^d \frac { |\e^{- i  r_1 \xi^j_1} -1|}{ r_1| \xi^j_1| }\frac { |\e^{- i  r_2\xi^j_2} -1|}{r_2 | \xi^j_2| }
   \frac { |\e^{ i (r_1\xi^j_1 + \sigma \xi^j_3)} -1|}{ | r_1\xi^j_1   + \sigma \xi^j_3 |}  \frac { |\e^{ i (r_2\xi^j_2 -\sigma \xi^j_3)} -1|}{ | r_2 \xi^j_2   -  \sigma \xi^j_3|}\\
   & \le  2^{4d}\prod_{j=1}^d   (1\wedge  (r_1  | \xi^j_1|)^{-1} ) (1\wedge (r_2 | \xi^j_2|)^{-1})   (1\wedge | r_1\xi^j_1  +\sigma \xi^j_3 |^{-1}) (1\wedge | r_2 \xi^j_2 - \sigma \xi^j_3 |^{-1}).	\end{align*}
For any $x\in \R^d$, we have
   \begin{equation}  \label{BB}
   \prod_{i=1}^d   \left(1\wedge | x^j| ^{-1} \right) \le  1\wedge ( \max_{1\le j \le d}  |x_j|)^{-1} \le  1\wedge d^{-1/2} \|x\| ^{-1}\le 1\wedge \|x\|^{-1}.
   \end{equation}
	As a consequence,
	\[
	 |\Delta_1( \xi_1, \xi_2,\xi_3)|  \le     (1\wedge  (r_1  \| \xi_1\|)^{-1} ) (1\wedge ( r_2 \| \xi_2\|)^{-1})   (1\wedge \| r_1 \xi_1 +\sigma \xi_3\|^{-1}) (1\wedge \| r_2\xi_2 - \sigma\xi_3 \|^{-1}),
	\]
which implies
	\begin{align*}
	& \int_{[ 0, \infty)^3}  \d \sigma   \d r_1 \d r_2  \    |\Delta( \xi_1, \xi_2,\xi_3)|  
	 \le 
	C (\| \xi_1 \| \| \xi_2 \| \| \xi_3 \|)^{ -1 } \\
	& \quad \times   \int_{[0,\infty)^3} \d x \d y \d z\  (1 \wedge  x^{-1}) (1 \wedge  y^{-1} ) ( 1\wedge \| x e_1 + ze_3 \|^{-1})
	 ( 1\wedge \| y e_2 - ze_3 \|^{-1}),
	 \end{align*}
	 where the last inequality follows from a change of variable, and $e_i$, $i=1,2,3$ are unit vectors.
	Note that
	\[
	\| x e_1 + ze_3 \|^2= x^2 + z^2 +2xz \langle e_1, e_3 \rangle \ge x^2 + z^2 -2xz = (x-z)^2.
	\]
	Therefore,
	\begin{align*}
	& \int_{[0,\infty)^3} \d x \d y \d z\  (1 \wedge  x^{-1}) (1 \wedge  y^{-1} ) ( 1\wedge \| x e_1 + ze_3 \|^{-1})
	 ( 1\wedge \| y e_2 - ze_3 \|^{-1}) \\
	& \le
	  \int_{\R^3} \d x \d y \d z\  (1 \wedge | x|^{-1}) (1 \wedge  |y|^{-1} ) ( 1\wedge |x-z|^{-1})
	 ( 1\wedge |y-z|^{-1}).
	 \end{align*}
	Finally, applying H\"older and Young's inequality, we obtain
	\begin{align*}
	& \int_{\R^3} \d x \d y \d z \ (1 \wedge | x|^{-1}) (1 \wedge  |y|^{-1} ) ( 1\wedge |x-z|^{-1})
	 ( 1\wedge |y-z|^{-1})   \\
	& \quad \le  \| (1\wedge | \bullet |^{-1})*(  1\wedge | \bullet |^{-1})\|_{L^2(\R)}^2  \le \| 1\wedge | \bullet |^{-1} \|_{L^{4/3} (\R)} ^4 <\infty.
	 \end{align*}

	 \medskip
	 Let us turn now to the analysis of $\Phi^{(2)}_N$ given in \eqref{N2}. Because the variable $w$ appears in four heat kernels and three of them have different variances, we cannot proceed as in the case of $\Phi^{(1)}_N$. Then, we start making the changes of variables without integrating in $w$. The first change of variables is $x_i \to Nx_i$, which yields
	 \begin{align*}
	N^{-4d +3}\Phi^{(2)}_N&= N^3 C_{t,4}^2 c_{t,4}^2\int_0^t \d r\int_{[0, r]^2}\d s_1\d s_2\int_{\R^{4d}} f(\d b)\d w f(\d z_1)f(\d z_2)\int_{[0, 1]^{4d}}\d x_1\d x_1' \d x_2\d x_2' \,   \\
  & \qquad \quad\times   \bm{p}_{r(t-r)/t}\left(
	w - \frac{rN}{t}x_1 \right) \bm{p}_{s_1(t-s_1)/t + s_1(r-s_1)/r}\left(
	z_1- \frac{s_1N}{t}x_1' + \frac{s_1}{r}w \right)  \\
	& \qquad \quad \times  \bm{p}_{r(t-r)/t}\left(
	w+b - \frac{rN}{t}x_2 \right)  \bm{p}_{s_2(t-s_2)/t + s_2(r-s_2)/r}\left(
	z_2 - \frac{s_2N}{t}x_2' +  \frac{s_2}{r}(w+b) \right).
	 \end{align*}
	 Next we make the change of variables $s_1=\frac t {N}r_1$, $s_2=\frac t {N} r_2$ and $r=\frac tN \sigma$,  in order to obtain
	 	 \begin{align*}
	N^{-4d +3}\Phi^{(2)}_N&=  t^3 C_{t,4}^2 c_{t,4}^2\int_{[0, N]}  d\sigma \int_{[0,\sigma]^2}  \d r_1 \d r_2 \int_{\R^{4d}} f(\d b)\d w f(\d z_1)f(\d z_2)
 	\int_{[0, 1]^{4d}}\d x_1\d x_1' \d x_2\d x_2' \,    \\
  & \qquad  \times      \bm{p}_{\frac {t\sigma}  {N} (1- \frac \sigma{N}) }\left(
	w -  \sigma x_1 \right) \bm{p}_{ \frac { tr_1}{N}  (1-\frac {r_1}{N}) + \frac {tr_1} {N\sigma} (\sigma-r_1) }\left(
	z_1-    r_1x_1' + \frac {r_1}{\sigma}w \right)  \\
	& \qquad   \times  \bm{p}_{\frac {t\sigma}  {N} (1- \frac \sigma{N}) }\left(
	w+b - \sigma  x_2\right)  \bm{p}_{\frac {tr_2} {N}  (1-\frac {r_2}{N}) +  \frac {tr_2} {N\sigma} (\sigma-r_2)}\left(
	z_2 - r_2 x_2' +  \frac{ r_2}\sigma(w+b) \right).
	 \end{align*}
	 To simplify the presentation, we set
	 \[
	 \gamma_{0,N}=\frac {t\sigma}  {N} \left(1- \frac \sigma{N}\right), \qquad \gamma_{1,N} =\frac { tr_1}{N}  \left(1-\frac {r_1}{N}\right) + \frac {tr_1} {N\sigma} (\sigma-r_1)
	 \]
	 and
	 \[
	 \gamma_{2,N} =\frac { tr_2}{N}  \left(1-\frac {r_2}{N}\right) + \frac {tr_2} {N\sigma} (\sigma-r_2).
	 \]
	 With the change of variables $z=w-\sigma x_1$ and the notation $y_1= x_1' -x_1$, $y_2=x_2'-x_1 $ and $y_3=x_2-x_1$, we can write
	 \begin{align*}
	 \Psi_2(z_1, z_2, b) &: =  \int_{\R^d} \d w \  \bm{p}_{ \gamma_{0,N} }\left(
	w -  \sigma x_1  \right)  \bm{p}_{  \gamma_{1,N} }\left(
	z_1-    r_1 x_1' + {\color{black}\frac{r_1}{\sigma}}w \right)   \\
	& \qquad \times 
	  \bm{p}_{ \gamma_{0,N} }\left(
	w+b -  \sigma x_2  \right)  \bm{p}_{ \gamma_{2,N}}\left(
	z_2 -  r_2 x_2' +  \frac{r_2}{ \sigma}(w+b) \right)\\
	 &= \int_{\R^d} \d z\   \bm{p}_{ \gamma_{0,N} }\left(
	z \right)  \bm{p}_{  \gamma_{1,N} }\left(
	z_1-    r_1 y_1 +  {\color{black}\frac{r_1}{\sigma}}z \right)   \\
	& \qquad \times 
	  \bm{p}_{ \gamma_{0,N} }\left(
	z+b -  \sigma  y_3  \right)  \bm{p}_{ \gamma_{2,N}}\left(
	z_2 -  r_2  y_2 +  \frac{r_2}{ \sigma}(z+b) \right).
	\end{align*}
	The Fourier transform of the function  $\Psi_2(z_1, z_2, b) $
	is equal to
	\begin{align*}
	\hat{ \Psi}_2(\xi_1, \xi_2, \xi_3) &= \int_{\R^d} \d z  \bm{p}_{ \gamma_{0,N} } (z)
	\exp \left( -\frac 12 \gamma_{1,N} \| \xi_1 \|^2 -\frac 12 \gamma_{0,N} \| \xi_2 -\frac {r_2} \sigma \xi_3  \|^2 -\frac 12 \gamma_{2,N} \| \xi_3  \| ^2 \right) \\
	& \quad \times  \exp \left(i    \xi_1 \cdot \left(\ r_1 y_1  - {\color{black}\frac{r_1}{\sigma}} z\right)  
	+ i  \left(\xi_2  -\frac {r_2} \sigma  \xi_3\right) \cdot \left( \sigma y_3-z \right) + i   \xi_3   \cdot \left(   r_2 y_2 -\frac {r_2} \sigma z  \right) \right)\\
	&= 	\exp \left( -\frac 12 \gamma_{1,N} \| \xi_1 \|^2 -\frac 12 \gamma_{0,N} \| \xi_2 -\frac {r_2} \sigma \xi_3  \|^2 -\frac 12 \gamma_{2,N} \| \xi_3  \| ^2
	-\frac  12 \gamma_{0,N} \| {\color{black}\frac{r_1}{\sigma}} \xi_1 +  \xi_2\|^2 \right) \\
	& \quad \times  \exp \left(i    r_1 \xi_1 \cdot  y_1   
	+ i  \left(\sigma \xi_2  -r_2 \xi_3\right) \cdot  y_3  + i    r_2 \xi_3   \cdot    y_2 \right).
	\end{align*}
	Set 
	\begin{align*}
	\Delta_2( \xi_1, \xi_2, \xi_3)&:= 
	\int_{[0, 1]^{4d}}\d x_1\d x_1' \d x_2\d x_2' \\
&  \qquad   \exp \left(i    r_1 \xi_1 \cdot   (x_1'- x_1)   
	+ i  \left(\sigma \xi_2  -r_2 \xi_3\right) \cdot  (x_2-x_1)  + i    r_2 \xi_3   \cdot    (x_2'-x_1) \right).
	\end{align*}
Then, Parseval's identity implies that
\begin{align*}
	N^{-4d+3}\Phi_N^{(1)}&=
  C_{t,4}^2 c_{t,4}^2   t^3   \frac 1 { (2\pi)^{3d}}
  \int_{[0, N]}  d\sigma \int_{[0,\sigma]^2}  \d r_1 \d r_2 
   \int_{\R^{3d}}  \hat{f}(\d \xi_1) \hat{f}(\d \xi_2) \hat{f}(\d \xi_3)  
	 \Delta_1( \xi_1, \xi_2,\xi_3)\\
 & \qquad \times	 \exp \left( -\frac 12 \gamma_{1,N} \| \xi_1 \|^2 -\frac 12 \gamma_{0,N} \| \xi_2 -\frac {r_2} \sigma \xi_3  \|^2 -\frac 12 \gamma_{2,N} \| \xi_3  \| ^2
	-\frac  12 \gamma_{0,N} \| {\color{black}\frac{r_1}{\sigma}}\xi_1 +  \xi_2\|^2 \right) \\
	 & \le C \int_{[0, \infty)^3}  \d \sigma  \d r_1 \d r_2    \int_{\R^{3d}}  \hat{f}(\d \xi_1) \hat{f}(\d \xi_2) \hat{f}(\d \xi_3)
	|\Delta_2( \xi_1, \xi_2,\xi_3)|.
\end{align*}
Taking into account that $\mathcal{R}(f) < \infty$, which is equivalent by Lemma \ref{lem:f:1} to $\int_{\R^d} \| z \|^{-1} \hat{f}(\d z)<\infty$, it suffices to show that
 \begin{equation} \label{BBB}
\int_{[0, \infty)^3}  \d \sigma  \d r_1 \d r_2 |\Delta_2( \xi_1, \xi_2,\xi_3)| \le  C(\| \xi_1 \| \| \xi_2 \| \| \xi_3 \|)^{ -1},
  \end{equation}
  for some constant $C$ not depending on $t$. Taking into account that
\begin{align*}
|  \Delta_2( \xi_1, \xi_2, \xi_3)|
 & =  \prod_{j=1}^d  \frac  { | \e^{-i(r_1 \xi^j_1 + r_2 \xi^j_3|} -1} { | r_1 \xi^j_1 + r_2 \xi^j_3|}
   \frac {  | \e^{ i r_1 \xi^j_1} -1|} {| r_1 \xi^j_1|}
    \frac {  | \e^{ i( \sigma \xi^j_2-r_2  \xi^j_3)} -1|} {| \sigma \xi^j_2-r_2 \xi^j_3 |}
     \frac {  | \e^{ i r_2 \xi_3^j} -1|} {| r_2 \xi_3^j|} \\
     &\le  (1\wedge \| r_1  \xi_1 + r_2 \xi_3 \| ^{-1} ) (1\wedge \| r_1 \xi_1 \|^{-1})(1\wedge \| \sigma \xi_2 -r_2 \xi_3 \|^{-1})(1\wedge \| r_2 \xi_3 \|^{-1}),
  \end{align*}  the  proof of  \eqref{BBB} can be done by the same arguments as in the proof of \eqref{BBB1}.
  The proof of Theorem  \ref{TVD3} is now complete.
\end{proof}

\subsection{Proof of Theorem \ref{TVD1}}

\begin{proof}[Proof of Theorem \ref{TVD1}]
By Theorem \ref{th:AV:d=1} and Proposition \ref{propTV}, we need show that there exists a constant $C>0$ such that for all $N \geq \e$,
\begin{align} \label{var0}
	{\rm Var} \left(\langle D\mathcal{S}_{N,t} , v_N \rangle_{\HH}\right)
	\leq C  \left(\frac{\log N}{N}\right)^{3}.
\end{align}
We recall the decomposition \eqref{1+2} of  ${\rm Var} \left(\langle D\mathcal{S}_{N,t} , v_N \rangle_{\HH}\right)$.

\medskip
\noindent
{\it Estimation of  $\Phi_N^{(1)}$}.
 According to \eqref{1,1},  we integrate $x_1'$ and $x_2'$ on $\R$ and obtain
 \begin{align*}
 \Phi_N^{(1)}&\le 16  C_{t,4}^2 c_{t,4}^2     f(\R)^2 \int_{[0,t]^2}\d s_1\d s_2\, \frac{t^2}{s_1s_2}\int_0^{s_1\wedge s_2}\d r \int_{\R^3}f(\d b)\d y_1   \d y_2 \int_{[0, N]^2}\d x_1 \d x_2\\
   & \quad \quad  \times   \bm{p}_{s_1(t-s_1)/t}\left(
	y_1 - \frac{s_1}{t} x_1 \right)   \bm{p}_{s_2(t-s_2)/t}\left(
	y_2 - \frac{s_2}{t} x_2 \right)   \bm{p}_{r(s_1- r)/s_1 + r(s_2- r)/s_2}\left(
	b - \frac{r}{s_2}y_2 + \frac{r}{s_1}y_1 \right)\\
	&=16  C_{t,4}^2 c_{t,4}^2  f(\R)^2 \int_{[0,t]^2}\d s_1\d s_2\, \frac{t^2}{s_2}\int_0^{s_1\wedge s_2}\d r\, \frac{1}{r} \int_{\R^3}f(\d b)\d y_1   \d y_2 \int_{[0, N]^2}\d x_1 \d x_2\,  \\ 
   & \quad \quad  \times  \bm{p}_{s_1(t-s_1)/t}\left(
	y_1 - \frac{s_1}{t} x_1 \right) \bm{p}_{s_2(t-s_2)/t}\left(
	y_2 - \frac{s_2}{t} x_2 \right) \\
 & \quad \quad  \times   \bm{p}_{[r(s_1- r)/s_1 + r(s_2- r)/s_2]/(r^2/s_1^2)}\left(
	\frac{s_1}{r}b - \frac{s_1}{s_2}y_2 +y_1 \right),
  \end{align*}
where in the equality we use  property \eqref{scale} with $d=1$.
Hence, by the semigroup property, we see that
\begin{align*}
 \Phi_N^{(1)}&\le  16  C_{t,4}^2 c_{t,4}^2 f(\R)^2 \int_{[0,t]^2}\d s_1\d s_2\, \frac{t^2}{s_2}\int_0^{s_1\wedge s_2}\d r\, \frac{1}{r} \int_{\R^2}f(\d b) \d y_2 \int_{[0, N]^2}\d x_1 \d x_2\,  \\ 
   & \quad \quad   \bm{p}_{s_2(t-s_2)/t}\left(
	y_2 - \frac{s_2}{t} x_2 \right)   \bm{p}_{s_1(t-s_1)/t +[r(s_1- r)/s_1 + r(s_2- r)/s_2]/(r^2/s_1^2)}\left(
	\frac{s_1}{r}b - \frac{s_1}{s_2}y_2 + \frac{s_1}{t} x_1 \right).
\end{align*}
We repeat the use of  \eqref{scale} with $d=1$ and the semigroup property to obtain 
\begin{align*}
 \Phi_N^{(1)}&\le  16  C_{t,4}^2 c_{t,4}^2 f(\R)^2 \int_{[0,t]^2}\d s_1\d s_2\, \frac{t^2}{s_1}\int_0^{s_1\wedge s_2}\d r\, \frac{1}{r} \int_{\R^2}f(\d b) \d y_2 \int_{[0, N]^2}\d x_1 \d x_2\,  \\ 
   & \quad \quad   \times  \bm{p}_{s_2(t-s_2)/t}\left(
	y_2 - \frac{s_2}{t} x_2 \right)   \bm{p}_{[s_1(t-s_1)/t +[r(s_1- r)/s_1 + r(s_2- r)/s_2]/(r^2/s_1^2)]/(s_1^2/s_2^2)}\left(
	\frac{s_2}{r}b - y_2 + \frac{s_2}{t} x_1 \right)\\
	&=  16  C_{t,4}^2 c_{t,4}^2 f(\R)^2 \int_{[0,t]^2}\d s_1\d s_2\, \frac{t^2}{s_1}\int_0^{s_1\wedge s_2}\d r\, \frac{1}{r} \int_{\R}f(\d b)  \int_{[0, N]^2}\d x_1 \d x_2\,  \\ 
   & \quad \quad \times \bm{p}_{s_2(t-s_2)/t +[s_1(t-s_1)/t +[r(s_1- r)/s_1 + r(s_2- r)/s_2]/(r^2/s_1^2)]/(s_1^2/s_2^2)}\left(
	\frac{s_2}{r}b  + \frac{s_2}{t} x_1 - \frac{s_2}{t} x_2\right)\\
	&= 16  C_{t,4}^2 c_{t,4}^2  f(\R)^2 \int_{[0,t]^2}\d s_1\d s_2\, \frac{t^2}{s_1s_2}\int_0^{s_1\wedge s_2}\d r\int_{\R}f(\d b)  \int_{[0, N]^2}\d x_1 \d x_2\,  \bm{p}_{2r(t-r)/t}\left(
	b  + \frac{r}{t} (x_1 - x_2)\right),
  \end{align*}
where in the second equality we use  the relation
$$
2r(t-r)/t= [s_2(t-s_2)/t +[s_1(t-s_1)/t +[r(s_1- r)/s_1 + r(s_2- r)/s_2]/(r^2/s_1^2)]/(s_1^2/s_2^2)]/(s_2^2/r^2).
$$
Now, using the notation $I_N=\frac{1}{N}1_{[0, N]}$ and Plancherel's identity, we conclude that
\begin{align*}
 \Phi_N^{(1)}&\le  16  C_{t,4}^2 c_{t,4}^2N^2 f(\R)^2 \int_{[0,t]^2}\d s_1\d s_2\, \frac{t^2}{s_1s_2}\int_0^{s_1\wedge s_2}\d r\, \left( I_N*\tilde{I}_N*\left(f*\bm{p}_{2r(t-r)/t}\left(
	\frac{r}{t} (\cdot)\right)\right)\right)(0)\\
	&=16  C_{t,4}^2 c_{t,4}^2 \frac{N^2}{\pi} f(\R)^2 \int_{[0,t]^2}\d s_1\d s_2\, \frac{t^3}{rs_1s_2}\int_0^{s_1\wedge s_2}\d r\, \int_{\R}\d z\, \frac{1-\cos (Nz)}{N^2z^2}\hat{f}\left(\frac{tz}{r}\right)\e^{-\frac{t(t-r)}{r}z^2}\\
	& \leq 16  C_{t,4}^2 c_{t,4}^2\frac{N}{\pi} f(\R)^3 \int_{[0,t]^2}\d s_1\d s_2\, \frac{t^3}{rs_1s_2}\int_0^{s_1\wedge s_2}\d r\, \int_{\R}\d z\, \frac{1-\cos (z)}{z^2}\e^{-\frac{t(t-r)}{r}\frac{z^2}{N^2}}\\
& =32  C_{t,4}^2 c_{t,4}^2 \frac{N}{\pi} f(\R)^3 \int_{0 \leq s_1\leq s_2\leq t}\d s_1\d s_2\d r\, \frac{t^3}{rs_1s_2}\int_0^{s_1\wedge s_2}\d r\, \int_{\R}\d z\, \frac{1-\cos (z)}{z^2}\e^{-\frac{t(t-r)}{r}\frac{z^2}{N^2}}.
  \end{align*}
  Integrating in the variables $s_1$ and $ s_2$ yields
    \[
      \Phi_N^{(1)}  \le 32  C_{t,4}^2 c_{t,4}^2 \frac {N}{\pi}f(\R)^3
      \int_0^t   \d r  \frac {t^3}{r }    \left( \log \left(\frac tr\right) \right)^2    \int_{\R}  \e^{-\frac { t(t-r)}r \frac { z^2} {N^2}}  \varphi(z) \d z,
    \]
    where we recall that $\varphi(z)= (1-\cos z)/z^2$.
    Making the change of variables $\frac {t-r} r =\theta$, allows us to write
       \[
      \Phi_N^{(1)}  \le 32  C_{t,4}^2 c_{t,4}^2 \frac {N}{\pi}   t^3f(\R)^3 \int_{\R}  \varphi(z) \d z
      \int_0^\infty   \d \theta \frac 1{\theta +1}   \left( \log (\theta +1) \right)^2     \e^{-\frac { t\theta  z^2} {N^2}}   .
    \]
    Integrating by parts and using the fact that 
    \[
\left(\frac 13 ( \log(\theta+1))^3    \e^{-\frac { t\theta  z^2} {N^2}}  \right)  _{\theta=0} ^{\theta=\infty} =0,
\]
we obtain
           \begin{align*}
      \Phi_N^{(1)}&  \le 32  C_{t,4}^2 c_{t,4}^2 \frac {N}{3\pi}   t^3  \int_{\R}   \varphi(z) \d z
      \int_0^\infty   \d \theta   \left( \log (\theta +1) \right)^3     \e^{-\frac { t\theta  z^2} {N^2}}  \frac {tz^2}{N^2}  \\
      &=32  C_{t,4}^2 c_{t,4}^2  \frac {N}{3\pi}   t^3  \int_{\R} \varphi(z) \d z
      \int_0^\infty   \d \theta   \left( \log \left(\frac {N^2} {tz^2}\theta +1\right) \right)^3     \e^{-\theta}.
    \end{align*}
   Using the inequality
   \begin{align*}
   \log \left(\frac {N^2} {tz^2}\theta +1\right) &\le  
  2 \log N + \log(\theta+1) + \log ( \frac 1t+ 1) + \log ( \frac 1 {z^2} +1) \\
  &\le  \left( 2 \log N + \log ( \frac 1t+ 1)\right) \left(1+ \log(\theta+1)+ \log ( \frac 1 {z^2} +1)  \right),
  \end{align*}
  and taking into account that
  \[
 C:=  \int_{\R} \varphi(z) \d z
      \int_0^\infty   \d \theta   \left(1+ \log(\theta+1)+ \log ( \frac 1 {z^2} +1)  \right)^3 e^{-\theta}<\infty,
      \]
      we finally get
      \[
         \Phi_N^{(1)}   \lesssim      C_{t,4}^2 c_{t,4}^2 t^3  N \left( 2 \log N + \log ( \frac 1t+ 1)\right)^3,
         \]
         which provides the desired estimate.

      \medskip

\noindent
{\it Estimation of  $\Phi_N^{(2)}$}.  Recall the estimate in  \eqref{N2}.     Notice that we should not integrate the variables $x_1'$ and $x_2'$ on the whole real line, because this would produce a factor
  $(s_1s_2)^{-1}$ which is not integrable on $[0,r]^2$. For this reason, we choose to integrate the variables $x_1$ and $x_2$ on $\R$ and we obtain,  using \eqref{scale} with $d=1$,
  \begin{align*}
  \Phi_N^{(2)} &\leq C_{t,4}^2 c_{t,4}^2\int_0^t \d r\,  \frac {t^2} {r^2} \int_{[0, r]^2}\d s_1\d s_2\int_{\R^4} f(\d b)\d w f(\d z_1)f(\d z_2)\int_{[0, N]^2}\d x_1' \d x_2' \, \\
  & \qquad \quad  \times  \bm{p}_{s_1(t-s_1)/t + s_1(r-s_1)/r}\left(
	z_1 - \frac{s_1}{t}x_1' + \frac{s_1}{r}w  \right)  \bm{p}_{s_2(t-s_2)/t +s_2(r-s_2)/r}\left(
	z_2 - \frac{s_2}{t}x_2' + \frac{s_2}{r}(w+b) \right)\\
	&= C_{t,4}^2 c_{t,4}^2\int_0^t \d r\,  \frac {t^2} {s_1s_2} \int_{[0, r]^2}\d s_1\d s_2\int_{\R^3} f(\d b) f(\d z_1)f(\d z_2)\int_{[0, N]^2}\d x_1' \d x_2' \, \\
& \qquad \quad   \times \bm{p}_{\alpha}\left(\frac{r}{s_2}z_2 -\frac{r}{s_1}z_1+b - \frac{r}{t}(x_2'- x_1')  \right) \\
	&= C_{t,4}^2 c_{t,4}^2\int_0^t \d r\,  \frac {t^2} {s_1s_2} \int_{[0, r]^2}\d s_1\d s_2\int_{\R^2} f(\d z_1)f(\d z_2)\int_{[0, N]^2}\d x_1' \d x_2' \, \\
	& \qquad \quad  \times \left(f*\bm{p}_{\alpha}\right)\left(  \frac{r}{t}(x_2'- x_1') +\frac{r}{s_1}z_1 -\frac{r}{s_2}z_2\right),
  \end{align*}
where
\begin{align*}
\alpha = [s_1(t-s_1)/t + s_1(r-s_1)/r]/(s_1^2/r^2) + [s_2(t-s_2)/t +s_2(r-s_2)/r]/(s_2^2/r^2).
\end{align*}

Using $I_N=\frac{1}{N}1_{[0, N]}$, we write
    \begin{align*}
  \Phi_N^{(2)} &\leq N^2C_{t,4}^2 c_{t,4}^2\int_0^t \d r\,  \frac {t^2} {s_1s_2} \int_{[0, r]^2}\d s_1\d s_2\int_{\R^2} f(\d z_1)f(\d z_2) \, \\
  &\qquad \quad \times \left(I_N*\tilde{I}_N*\left(\left(f*\bm{p}_{\alpha}\right)\left(  \frac{r}{t}(\cdot) +\frac{r}{s_1}z_1 -\frac{r}{s_2}z_2\right)\right)\right)(0).
  \end{align*}
 We apply  Plancherel's identity to conclude that 
   \begin{align*}
  \Phi_N^{(2)} &\leq \frac{N^2}{\pi}C_{t,4}^2 c_{t,4}^2\int_0^t \d r\,  \frac {t^3} {rs_1s_2} \int_{[0, r]^2}\d s_1\d s_2\int_{\R^2} f(\d z_1)f(\d z_2)\int_{\R}\d z \,\\
  & \qquad \quad \times \frac{1-\cos(Nz)}{N^2z^2}  \e^{iz \left(\frac{t}{s_2}z_2- \frac{t}{s_1}z_1\right)}\hat{f}\left(\frac{tz}{r}\right)\e^{-\frac{\alpha t^2}{2r^2}z^2}\\
  &= \frac{N}{\pi}C_{t,4}^2 c_{t,4}^2\int_0^t \d r\,  \frac {t^3} {rs_1s_2} \int_{[0, r]^2}\d s_1\d s_2\int_{\R}\d z \,\frac{1-\cos(z)}{z^2} \hat{f}\left(\frac{tz}{s_2}\right)\hat{f}\left(-\frac{tz}{s_1}\right)\hat{f}\left(\frac{tz}{r}\right)\e^{-\frac{\alpha t^2}{2r^2}\frac{z^2}{N^2}}\\
  &\leq \frac{N}{\pi}C_{t,4}^2 c_{t,4}^2 f(\R)^3\int_0^t \d r\,  \frac {t^3} {rs_1s_2} \int_{[0, r]^2}\d s_1\d s_2\int_{\R}\d z \,\frac{1-\cos(z)}{z^2}\e^{-\frac{\alpha t^2}{2r^2}\frac{z^2}{N^2}}.
  \end{align*}
  Denote 
    \[
 \sigma:= \frac{\alpha t^2}{r^2}=t(t-s_1)/s_1+   t^2(r-s_1)/(rs_1)+ t(t-s_2)/s_2+ t^2(r-s_2)/(rs_2).
  \]
 Recalling $\varphi(z)=(1-\cos(z))/z^2$, we can write
      \begin{align*} 
     \Phi_N^{(2)} & \le  \frac N {\pi} 
      C_{t,4}^2 c_{t,4} ^2  f(\R)^3
 \int_{\R} \varphi(z)  \d z     \int_0^t   \int_{[0,r]^2}
    \d s_1 \d s_2 \d r    \frac {t^3} {rs_1s_2}   \e^{ -\frac {\sigma z^2} {2N^2}} \\
    &=  \frac N {\pi}  C_{t,4}^2 c_{t,4} ^2  f(\R)^3
 \int_{\R}  \varphi(z)  \d z     \int_0^t  \d r  \frac {t^3} r  \left(  \int_0^r
    \d s   \frac 1{s}    \e^{ -\frac {[ t(t-s)/s+   t^2(r-s)/(rs) ]  z^2} {2N^2}} \right)^2.
      \end{align*}
Making the change of variables $(r-s)/s =\theta$, yields
  \[
 \int_0^r
    \d s   \frac 1{s}    \e^{ -\frac {[ t(t-s)/s+   t^2(r-s)/(rs) ]  z^2} {2N^2}}
  =\int_0 ^\infty \frac 1 {1+\theta} \e^{-\frac {tz^2}{2N^2} ( 2\theta t +t-r)/r)} \d \theta.
  \]
  As a consequence,
   \begin{align*}
   \Phi_N^{(2)}
   &\le
  \frac N {\pi}    C_{t,4}^2 c_{t,4} ^2  t^3  f(\R)^3
  \int_{\R}   \varphi(z)  \d z 
    \int_0^t  \frac 1r  \e^{- \frac {z^2}{N^2} t(t-r)/r}
  \left(\int_0 ^\infty \frac 1 {1+\theta} \e^{-\frac {t^2z^2\theta}{rN^2} } \d \theta
 \right)^2 \d z \d r.
  \end{align*}
  With the further change of variable $\frac {t-r}r =\xi$, we obtain
    \begin{align*}
   \Phi_N^{(2)}
   &\le
    \frac N {\pi}    C_{t,4}^2 c_{t,4} ^2  t^3 f(\R)^3
  \int_{\R}   \varphi(z)  \d z 
    \int_0^\infty  \frac 1{1+\xi}  e^{- \frac {tz^2 \xi}{N^2} }
  \left(\int_0 ^\infty \frac 1 {1+\theta} \e^{-\frac {t(\xi+1)z^2\theta}{N^2} } \d \theta
 \right)^2 \d z \d \xi \\
 &\le
      \frac N {\pi}    C_{t,4}^2 c_{t,4} ^2  t^3 f(\R)^3
  \int_{\R}   \varphi(z)  \d z 
  \left(\int_0 ^\infty \frac 1 {1+\theta} \e^{-\frac {tz^2\theta}{N^2} } \d \theta
 \right)^3 \d z \\
& =     \frac N {\pi}    C_{t,4}^2 c_{t,4} ^2  t^3 f(\R)^3
  \int_{\R}    \varphi(z) \d z 
  \left(\int_0 ^\infty \frac 1 {\theta+  \frac {t z^2}{N^2}}     \e^{-\theta}   \d \theta
 \right)^3 \d z.
  \end{align*}
  We have
  \begin{align*}
  \int_0 ^\infty \frac 1 {\theta+  \frac {t z^2}{N^2}}     \e^{-\theta}   \d \theta &
  \le 
    \int_1 ^\infty     \e^{-\theta}   \d \theta +
     \int_0 ^1  \frac 1 {\theta+  \frac {t z^2}{N^2}}         \d \theta =  \e^{-1} +
 \log \left( 1+ \frac {N^2} {tz^2} \right)     \\
 & \le \e^{-1} +  2 \log N + \log (1+ 1/t)+ \log (1+  z^{-2}).
\end{align*}
Taking into account that
\[
\int_{\R}    \varphi(z)     (1 + \log (1+  z^{-2}) )^3 \d z <\infty,
\]
we obtain the desired estimate for the term  $   \Phi_N^{(2)}$.
 This completes the proof of the estimate \eqref{var0}.
 \end{proof}

\subsection{Proof of Theorem \ref{TVD2}}

\subsubsection{Estimation of  $\Phi_N^{(1)}$}
Recalling \eqref{1,1} and using change of variables:  $y_1-\frac{s_1}{t}x_1= \alpha_1$, $y_2-\frac{s_2}{t}x_2= \alpha_2$, $y_1+z_1-\frac{s_1}{t}x_1'=\alpha_3$, $y_2+z_2-\frac{s_2}{t}x_2' = \alpha_4$, $b-\frac{r}{s_2}y_2+ \frac{r}{s_1}y_1=\alpha_5$,   yields that
 \begin{align*}
 \Phi_N^{(1)}&\leq 16  C_{t,4}^2 c_{t,4}^2   \int_{[0,t]^2}\d s_1\d s_2 \int_0^{s_1\wedge s_2}\d r\int_{\R^{5d}} \, \d \alpha_1\d \alpha_2\d \alpha_3\d \alpha_4\d \alpha_5 \int_{[0, N]^{4d}}\d x_1\d x_1' \d x_2\d x_2' \,  \nonumber\\ 
   & \quad \quad \times 
    \bm{p}_{s_1(t-s_1)/t}\left(\alpha_1 \right)
   \bm{p}_{s_2(t-s_2)/t}\left(
	\alpha_2 \right)  \bm{p}_{s_1(t-s_1)/t}\left(
\alpha_3\right)  \bm{p}_{s_2(t-s_2)/t}\left(
\alpha_4\right)\bm{p}_{r(s_1- r)/s_1 + r(s_2- r)/s_2}\left(
	\alpha_5 \right) \\
& \quad \quad\times \|\alpha_3-\alpha_1 -\frac{s_1}{t}(x_1-x_1')\|^{-\beta}\|\alpha_4-\alpha_2 -\frac{s_2}{t}(x_2-x_2')\\
& \quad \quad\times \|^{-\beta}\|\alpha_5+\frac{r}{s_2}\alpha_2 -\frac{r}{s_1}\alpha_1+ \frac{r}{t}(x_2-x_1)\|^{-\beta}.
  \end{align*}
 Let $Z_1, Z_2, Z_3, Z_4, Z_5$ be i.i.d. ${\rm N}(0, 1)$. We can write 
 \begin{align}
 \Phi_N^{(1)}&\leq 16  C_{t,4}^2 c_{t,4}^2     \int_{[0,t]^2}\d s_1\d s_2 \int_0^{s_1\wedge s_2}\d r \int_{[0, N]^{4d}}\d x_1\d x_1' \d x_2\d x_2' \,  \nonumber \\ 
 & \quad \quad\times  \E\Bigg[ \|\sqrt{s_1(t-s_1)/t}Z_3-\sqrt{s_1(t-s_1)/t}Z_1 -\frac{s_1}{t}(x_1-x_1')\|^{-\beta}\nonumber \\
  & \quad \quad\times \|\sqrt{s_2(t-s_2)/t}Z_4-\sqrt{s_2(t-s_2)/t}Z_2 -\frac{s_2}{t}(x_2-x_2')\|^{-\beta}\nonumber \\
 & \quad \quad \times \|\sqrt{r(s_1-r)/s_1 + r(s_2-r)/s_2}Z_5+\frac{r}{s_2}\sqrt{s_2(t-s_2)/t}Z_2 \nonumber \\
 & \qquad \qquad \qquad \qquad \qquad \qquad \qquad -\frac{r}{s_1}\sqrt{s_1(t-s_1)/t}Z_1+ \frac{r}{t}(x_2-x_1)\|^{-\beta}\Bigg]\nonumber \\
 &=16  C_{t,4}^2 c_{t,4}^2    N^{4d-3\beta}  \int_{[0,t]^2}\d s_1\d s_2 \int_0^{s_1\wedge s_2}\d r\, \left(\frac{t}{r}\right)^{\beta}\left(\frac{t}{s_1}\right)^{\beta}\left(\frac{t}{s_2}\right)^{\beta} \int_{[0, 1]^{4d}}\d x_1\d x_1' \d x_2\d x_2' \,  \nonumber \\ 
 & \quad \quad\times  \E\Bigg[\|\frac{t}{Ns_1}\sqrt{s_1(t-s_1)/t}Z_3-\frac{t}{Ns_1}\sqrt{s_1(t-s_1)/t}Z_1 -(x_1-x_1')\|^{-\beta}\nonumber \\
  & \quad \quad\times \|\frac{t}{Ns_2}\sqrt{s_2(t-s_2)/t}Z_4-\frac{t}{Ns_2}\sqrt{s_2(t-s_2)/t}Z_2 -(x_2-x_2')\|^{-\beta}\nonumber \\
 & \quad \quad \times \|\frac{t}{Nr}\sqrt{r(s_1-r)/s_1 + r(s_2-r)/s_2}Z_5+\frac{t}{Ns_2}\sqrt{s_2(t-s_2)/t}Z_2\nonumber  \\
 & \qquad \qquad \qquad \qquad \qquad \qquad \qquad -\frac{t}{Ns_1}\sqrt{s_1(t-s_1)/t}Z_1+ (x_2-x_1)\|^{-\beta}\Bigg], \label{1,1=}
  \end{align}
where in the second equality we have made a change of variables. 

\medskip

\noindent \textbf{Case 1}: $0<\beta <1$. Appealing to Lemma 3.1 of \cite{HNVZ2019} to the random variables $Z_5, Z_4, Z_3$ in this order, we see that the spatial integral  in \eqref{1,1=} is bounded above by 
\begin{align*}
&C\, \E\Bigg[\int_{[0, 1]^{4d}}\d x_1\d x_1' \d x_2\d x_2'\, \|\frac{t}{Ns_2}\sqrt{s_2(t-s_2)/t}Z_2-\frac{t}{Ns_1}\sqrt{s_1(t-s_1)/t}Z_1+ (x_2-x_1)\|^{-\beta}
\\
&\quad \quad \times \|\frac{t}{Ns_1}\sqrt{s_1(t-s_1)/t}Z_1 +(x_1-x_1')\|^{-\beta}
\|\frac{t}{Ns_2}\sqrt{s_2(t-s_2)/t}Z_2 +(x_2-x_2')\|^{-\beta} \Bigg]\\
&\quad \leq C\, \E\Bigg[\int_{[-1, 1]^{3d}}\d y_1\d y_2 \d y_3\, \|\frac{t}{Ns_2}\sqrt{s_2(t-s_2)/t}Z_2-\frac{t}{Ns_1}\sqrt{s_1(t-s_1)/t}Z_1+ y_3\|^{-\beta}
\\
&\quad \quad \qquad \times \|\frac{t}{Ns_1}\sqrt{s_1(t-s_1)/t}Z_1 +y_1\|^{-\beta}
\|\frac{t}{Ns_2}\sqrt{s_2(t-s_2)/t}Z_2 +y_2\|^{-\beta} \Bigg]\\
& \quad \leq C\, \left( \sup_{z\in \R^d} \int_{[-1, 1]^d}\|z+y\|^{-\beta}\d y\right)^3=C'<\infty,
\end{align*}
where in the first inequality we use a change of variables and in the second inequality we use the fact (see also \cite[(3.10)]{HNVZ2019})
\begin{align}\label{sup}
\sup_{z\in \R^d} \int_{[-1, 1]^d}\|z+y\|^{-\beta}\d y <\infty. 
\end{align}

Denote 
\begin{align*}
A_t=\int_{[0,t]^2}\d s_1\d s_2 \int_0^{s_1\wedge s_2}\d r\, \left(\frac{t}{r}\right)^{\beta}\left(\frac{t}{s_1}\right)^{\beta}\left(\frac{t}{s_2}\right)^{\beta}.
\end{align*}
Condition  $\beta <1$ implies $A_t<\infty$. Therefore, in the case $0<\beta <1$, we conclude that 
\begin{align}\label{1case1}
 \Phi_N^{(1)}&\leq C'C_{t,4}^2 c_{t,4}^2   A_t\, N^{4d-3\beta}.
 \end{align}

\medskip

\noindent \textbf{Case 2}: $1< \beta <2$. Recall \eqref{1,1=}. Applying Lemma \ref{max} to $Z_5, Z_4, Z_3$, using  the change of variables ($x_1'=x_1-y_1$, $x_2'=x_2-y_2$, $x_1=x_2-y_3$) and the fact that for all $c_1 >0$ and $z\in \R^d$
\begin{align}\label{comparemin}
\int_{[-1, 1]^d} c_1 \wedge \|z + y_1\|^{-\beta}\d y_1 &\leq 2^d \left(c_1 \wedge \int_{[-1, 1]^d}  \|z + y_1\|^{-\beta}\d y_1\right)\nonumber\\
& \lesssim c_1\wedge 1,  \qquad \text{see \eqref{sup}}
\end{align} 
we obtain that 
\begin{align*} 
 \Phi_N^{(1)}
 &\lesssim C_{t,4}^2 c_{t,4}^2N^{4d-3\beta}  \int_{[0,t]^2}\d s_1\d s_2 \int_0^{s_1\wedge s_2}\d r\,  
  \left[\left[N^{\beta} (s_1(t-s_1)/t)^{-\beta/2}\right]\wedge \left(\frac{t}{s_1}\right)^{\beta} \right] \\
& \qquad \times  \left[\left[N^{\beta} (s_2(t-s_2)/t)^{-\beta/2}\right]\wedge \left(\frac{t}{s_2}\right)^{\beta} \right]
\left[\left[N^{\beta} (r(s_1-r)/s_1 + r(s_2-r)/s_2)^{-\beta/2}\right]\wedge \left(\frac{t}{r}\right)^{\beta} \right].
\end{align*}
The change of variables $s_1\to \frac{ts_1}{N^2}$, $s_2\to \frac{ts_2}{N^2}$ and $r\to \frac {tr}{N^2}$ allows us to write
\begin{align*} 
 \Phi_N^{(1)}
 &\lesssim  N^{4d+3\beta-6}  \int_{[0,N]^2}\d s_1\d s_2 \int_0^{s_1\wedge s_2}\d r\,  
  \left[\left[ (s_1(1-s_1/N^2) )^{-\beta/2}\right]\wedge  s_1^{-\beta} \right] \\
& \qquad \times   \left[\left[ (s_2(1- s_2/N^2) )^{-\beta/2}\right]\wedge s_2^{-\beta} \right]
\left[ \left[ r^{-\beta/2} \left( 1- \frac r{2s_1} -\frac r{2s_2}\right )^{-\beta/2}\right]\wedge  r^{-\beta} \right].
\end{align*}
For the  integral in the variable $r$ we make the further change of variables  $r \left( \frac 1{2s_1} + \frac 1{2s_2} \right) =\lambda$ in order to obtain
         \begin{align}
       \Phi_N^{(1)}
     &\lesssim  N^{4d+3\beta-6}    \int_{ [0, N^2]^2}\d s_1\d s_2 \,  
        \left[\left[ (s_1(1-s_1/N^2))^{-\beta/2}\right]\wedge s_1^{-\beta} \right] \left[\left[ (s_2(1-s_2/N^2))^{-\beta/2}\right]\wedge s_2^{-\beta} \right]\nonumber \\
     & \qquad \times 
     \left(\frac{1}{2s_1}+ \frac{1}{2s_2}\right)^{\beta-1}
          \int_0^{1}\d \lambda\,  \lambda^{-\frac{\beta}{2}}       
      \left( \left[\left(\frac{1}{2s_1}+ \frac{1}{2s_2}\right)^{-\beta/2} \left(1-\lambda\right)^{-\frac{\beta}{2}}\right]
          \wedge \lambda^{-\frac{\beta}{2}}  \right) . \label{change}
               \end{align}

          From    \eqref{change},  we apply Lemma \ref{beta>1} to conclude that in the case $1<\beta<2$,                
                       \begin{align}
       \Phi_N^{(1)}
     &\lesssim   N^{4d+3\beta-6} . \label{1case2}
               \end{align}

\medskip 
\noindent 
\textbf{Case 3}: $\beta =1$. Notice that the estimate in \eqref{change} still holds for $\beta=1$. Now we apply Lemma \ref{beta=1} to conclude that in the case $\beta=1$
                       \begin{align}
   \Phi_N^{(1)}
     &\lesssim   N^{4d-3}(\log N)^3 . \label{1case3}
            \end{align}

\subsubsection{Estimation of  $\Phi_N^{(2)}$}

 Recall \eqref{N2}. Using  the change of variables: $w -\frac{r}{t}x_1=\alpha_1$, 
$w+b-\frac{r}{t}x_2 =\alpha_2$,  $z_1 -\frac{s_1}{t}x_1' +\frac{s_1}{r}w = \alpha_3$, $z_2-\frac{s_2}{t}x_2' + \frac{s_2}{r}(w+b)=\alpha_4$,  we obtain  
\begin{align}
  \Phi_N^{(2)} &\leq C_{t,4}^2 c_{t,4}^2\int_0^t \d r\int_{[0, r]^2}\d s_1\d s_2
  \int_{\R^{4d}}\d \alpha_1\d \alpha_2 \d \alpha_3\d \alpha_4
  \int_{[0, N]^{4d}}\d x_1\d x_1' \d x_2\d x_2' \, \nonumber \\
  & \qquad  \times   \bm{p}_{r(t-r)/t}\left(\alpha_1 \right) 
  \bm{p}_{r(t-r)/t}\left(\alpha_2\right)
  \bm{p}_{s_1(t-s_1)/t + s_1(r-s_1)/r}\left(\alpha_3 \right)
   \bm{p}_{s_2(t-s_2)/t + s_2(r-s_2)/r}\left(\alpha_4\right) \nonumber\\
   &\qquad \times \left\| \alpha_2-\alpha_1 +\frac{r}{t}(x_2-x_1) \right\|^{-\beta}
  	\left\| \alpha_3 -\frac{s_1}{r}\alpha_1 + \frac{s_1}{t}(x_1'-x_1) \right\|^{-\beta}
	\left\| \alpha_4-\frac{s_2}{r}\alpha_2 + \frac{s_2}{t}(x_2'-x_2) \right\|^{-\beta}\nonumber\\
    &= C_{t,4}^2 c_{t,4}^2\int_0^t \d r\int_{[0, r]^2}\d s_1\d s_2
   \int_{[0, N]^{4d}}\d x_1\d x_1' \d x_2\d x_2' \,  \nonumber \\
 &\qquad\qquad  \times  \E\Bigg[
	\left\| \sqrt{\frac{r(t-r)}{t}}\,
	Z_2- \sqrt{\frac{r(t-r)}{t}}\,Z_1 + \frac{r}{t}(x_2-x_1) \right\|^{-\beta} \nonumber \\
& \qquad\qquad  \times 
	\left\| \sqrt{\frac{s_1(t-s_1)}{t} + \frac{s_1(r-s_1)}{r}}\,Z_3- 
	\frac{s_1}{t}\sqrt{\frac{r(t-r)}{t}}\,Z_1 
	+ \frac{s_1}{t}(x_1'-x_1) \right\|^{-\beta} \nonumber \\
 & \qquad\qquad  \times 
	\left\| \sqrt{\frac{s_2(t-s_2)}{t} + \frac{s_2(r-s_2)}{r}}\,Z_4- 
	\frac{s_2}{t}\sqrt{\frac{r(t-r)}{t}}\,Z_2 + \frac{s_2}{t}(x_2'-x_2) \right\|^{-\beta}\Bigg]. \nonumber 
			  \end{align}
Now, using a change of variables  yields that 
\begin{align}
  \Phi_N^{(2)} &\leq  C_{t,4}^2 c_{t,4}^2N^{4d-3\beta}\int_0^t \d r\int_{[0, r]^2}\d s_1\d s_2
  \left(\frac{t}{r}\right)^{\beta} \left(\frac{t}{s_1}\right)^{\beta} \left(\frac{t}{s_2}\right)^{\beta}
   \int_{[0, 1]^{4d}}\d x_1\d x_1' \d x_2\d x_2' \, \nonumber \\
 &\qquad  \times   \E\Bigg[ 
	\left\| \frac{t}{Nr}\sqrt{\frac{r(t-r)}{t}}\,Z_2- \frac{t}{Nr}\sqrt{\frac{r(t-r)}{t}}\,Z_1 + (x_2-x_1)
	\right\|^{-\beta} \nonumber \\
& \qquad  \times 
	\left\| \frac{t}{Ns_1}\sqrt{\frac{s_1(t-s_1)}{t} + \frac{s_1(r-s_1)}{r}}\,Z_3
	- \frac{1}{N}\sqrt{\frac{r(t-r)}{t}}\,Z_1 + (x_1'-x_1) \right\|^{-\beta} \nonumber \\
 & \qquad \times 
	\left\| \frac{t}{Ns_2}\sqrt{\frac{s_2(t-s_2)}{t} + \frac{s_2(r-s_2)}{r}}\,Z_4
	- \frac{1}{N}\sqrt{\frac{r(t-r)}{t}}\,Z_2 + (x_2'-x_2) \right\|^{-\beta}\Bigg] . \label{N22}
			  \end{align}

\noindent
\textbf{Case 1}: $0<\beta<1$. We first apply  Lemma 3.1 of \cite{HNVZ2019} for $Z_4, Z_3$, then use a change of variables and \eqref{sup} to conclude 
\begin{align}
  \Phi_N^{(2)} &\leq  C\, \bar{A}_tC_{t,4}^2 c_{t,4}^2N^{4d-3\beta}, \label{2case1}
\end{align}
where $\bar{A}_t= \int_0^t \d r\int_{[0, r]^2}\d s_1\d s_2\
  (t/r)^{\beta} (t/s_1)^{\beta} (t/s_2)^{\beta} < \infty$ since $\beta<1$.

\medskip 

\noindent
\textbf{Case 2}: $1< \beta < 2$. In this case, recalling  \eqref{N22},  we proceed in the following order:  applying Lemma
 \ref{max} for $Z_4, Z_3$, using the change of variables ($x_1'=y_1+ x_1$, $x_2'=y_2+ x_2$) and \eqref{comparemin}, then applying Lemma
 \ref{max} for $Z_2$ and using change of variables $x_2=y_3+ x_1$ and \eqref{comparemin}, to obtain that 
\begin{align}
  \Phi_N^{(2)} &\lesssim C_{t,4}^2 c_{t,4}^2N^{4d-3\beta}\int_0^t \d r\int_{[0, r]^2}\d s_1\d s_2\,
  \left[\left[N^\beta (r(t-r)/t)^{-\beta/2}\right] \wedge (t/r)^{\beta}\right]
  \nonumber \\
  & \qquad\qquad \times  \left[\left[N^\beta (s_1(t-s_1)/t + s_1(r-s_1)/r)^{-\beta/2}\right] \wedge (t/s_1)^{\beta}\right]
  \nonumber \\
  & \qquad \qquad \times  \left[\left[N^\beta (s_2(t-s_2)/t + s_2(r-s_2)/r)^{-\beta/2}\right] \wedge (t/s_2)^{\beta}\right]
  \nonumber \\
 &\lesssim N^{4d-3\beta}\int_0^1 \d r\int_{[0, r]^2}\d s_1\d s_2\,
  \left[\left[N^\beta (r(1-r))^{-\beta/2}\right] \wedge r^{-\beta}\right]
  \nonumber \\
  & \qquad\qquad \qquad  \times  \left[\left[N^\beta (s_1(r-s_1)/r)^{-\beta/2}\right] \wedge s_1^{-\beta}\right]
  \left[\left[N^\beta ( s_2(r-s_2)/r)^{-\beta/2}\right] \wedge s_2^{-\beta}\right]\nonumber \\
  & =  N^{4d-3\beta}\int_0^1 \d r\, 
  \left[\left[N^\beta (r(1-r))^{-\beta/2}\right] \wedge r^{-\beta}\right]
 \left[\int_{0}^r    \left[\left[N^\beta ( s(r-s)/r)^{-\beta/2}\right] \wedge s^{-\beta}\right]\d s\right]^2,
  \label{N223}
			  \end{align}
where the second inequality follows by a change of variables.  Using a change of variables again, we see from \eqref{N223}
 that 
 \begin{align}
  \Phi_N^{(2)} &\lesssim N^{4d-3\beta}\int_0^1 \d r\, 
  \left[\left[N^\beta (r(1-r))^{-\beta/2}\right] \wedge r^{-\beta}\right] r^{2-2\beta}\nonumber\\
 & \qquad \qquad \times \left[ \int_{0}^1    \left[\left[(Nr^{1/2})^\beta ( s(1-s))^{-\beta/2}\right] \wedge s^{-\beta}\right]\d s\right]^2\nonumber\\
 & = 2N^{4d+3\beta-6}\int_0^N\d \alpha\,  \alpha^{5-4\beta}
 \left[\left[ (\alpha^2(1-{\alpha^2}/N^2))^{-\beta/2}\right] \wedge \alpha^{-2\beta}\right] \nonumber\\
  & \qquad \qquad \times \left[ \int_{0}^1    \left[\left[\alpha^\beta ( s(1-s))^{-\beta/2}\right] \wedge s^{-\beta}\right]\d s\right]^2.
  \label{change2}
			  \end{align}

We apply Lemma \ref{2beta>1} to conclude that in the case $1<\beta<2$, 
\begin{align}
  \Phi_N^{(2)} \lesssim N^{4d+3\beta-6}.
  \label{2case2}
			  \end{align}

\medskip
\noindent
\textbf{Case 3}: $\beta=1$. Notice that the estimate in \eqref{change2} still holds for $\beta=1$. We apply Lemma \ref{2beta=1} to conclude that in the case $\beta=1$
\begin{align}
  \Phi_N^{(2)} \lesssim N^{4d-3}(\log N)^3.
  \label{2case3}
			  \end{align}

\subsubsection{Proof of Theorem \ref{TVD2}}

\begin{proof}[Proof of Theorem \ref{TVD2}]
           Recall   \eqref{dTV2} and \eqref{1+2}.  The case $0<\beta<1$ follows from Theorem \ref{Riesz} Item 1,  \eqref{1case1} and \eqref{2case1}; the case $\beta=1$ follows from Theorem \ref{Riesz} Item 2,  \eqref{1case3} and \eqref{2case3}; the case $1<\beta<2$ follows from Theorem \ref{Riesz} Item 3,  \eqref{1case2} and \eqref{2case2}.
\end{proof}

\section{Appendix}

\begin{lemma}\label{identity}
       Let $I_N$ and $\tilde{I}_N$ be defined in \eqref{I_N}. Then for all $s<t$ and $w\in \R^d$,
        \begin{align}
      &   \int_{\R^d} \d x
    \left(I_N*\tilde{I}_N\right)(x)
   \left( f*\bm{p}_{2s(t-s)/t} \right)
     \left(\frac st x +   w\right) \nonumber\\
     & \quad = \frac{1}{\pi^d} \int_{\R^d}\e^{-s(t-s)\|z\|^2/t} \prod_{j=1}^d\frac{1-\cos(Nz_js/t)}{(Nz_js/t)^2}
     \e^{i z\cdot w}\hat{f}(\d z) \label{equal}\\
     & \quad \leq  \frac{1}{\pi^d} \int_{\R^d}\e^{-s(t-s)\|z\|^2/t}
      \prod_{j=1}^d\frac{1-\cos(Nz_js/t)}{(Nz_js/t)^2} \hat{f}(\d z). \label{inequality}
        \end{align}
\end{lemma}
\begin{proof}
       Clearly, it suffices to prove  \eqref{equal}, which is a consequence of the identity \eqref{Fourier} and  the fact that 
       the Fourier transform of $I_N*\tilde{I}_N$ is $2^d\prod_{j=1}^d\frac{1-\cos(Nz_j)}{(Nz_j)^2}$.
\end{proof}
\begin{lemma}\label{max}
Let $\Z\sim {\rm N}(0, 1)$. There exists a constant $C>0$ such that for all $s>0$ and $y\in \R^d$
\begin{align}\label{min}
\int_{\R^d}\bm{p}_s(x+y)\|x\|^{-\beta}\d x& =\E\left[\| \sqrt{s}\Z +y\|^{-\beta}\right] 
\leq C\,  \left(s^{-\beta/2} \wedge \|y\|^{-\beta}\right).
\end{align}
\end{lemma}
\begin{proof}
Since the convolution between $\bm{p}_s$ and $\|\cdot\|^{-\beta}$ is nonnegative definite and maximized at $0$,  using a change of variable we can write
\begin{align*}
\sup_{y\in \R^d}\int_{\R^d}\bm{p}_s(x+y)\|x\|^{-\beta}\d x = \int_{\R^d}\bm{p}_s(x)\|x\|^{-\beta}\d x=s^{-\beta/2}\int_{\R^d}\bm{p}_1(x)\|x\|^{-\beta}\d x.
\end{align*}
 This together with Lemma 3.1 of \cite{HNVZ2019} implies \eqref{min}.
\end{proof}

\begin{lemma}\label{smallbig}
          Fix $1\leq \beta <2$.
          Then we have for all $\alpha>0$,
          \begin{align}\label{intmin}
          \int_0^1 \lambda^{-\beta/2}\left(\left[\alpha^{\beta}(1-\lambda)^{-\beta/2}\right]\wedge \lambda^{-\beta/2}\right)
           \d \lambda \asymp
           \begin{cases}
            \bm{1}_{\{0<\alpha<1\}}\alpha + \bm{1}_{\{\alpha\geq 1\}}\log \alpha, & \beta=1,\\
           \bm{1}_{\{0<\alpha<1\}}\alpha^{\beta} + \bm{1}_{\{\alpha\geq 1\}}\alpha^{2\beta-2}, & 1<\beta<2.
           \end{cases}
          \end{align}
\end{lemma}
\begin{proof}
        We observe that 
        \begin{align*}
        \alpha^{\beta}(1-\lambda)^{-\beta/2} \leq  \lambda^{-\beta/2}
        \Leftrightarrow  \lambda \leq \frac{1}{1+\alpha^2}.
        \end{align*}
        Hence, 
        \begin{align*}
        &  \int_0^1 \lambda^{-\beta/2}\left(\left[\alpha^{\beta}(1-\lambda)^{-\beta/2}\right]\wedge \lambda^{-\beta/2}\right)
           \d \lambda \\
           & \quad =  \alpha^{\beta}  \int_0^{1/(1+\alpha^2)} \lambda^{-\beta/2}(1-\lambda)^{-\beta/2}\d \lambda + 
             \int_{1/(1+\alpha^2)}^1 \lambda^{-\beta} \d \lambda.
       \end{align*}
       \textbf{Case 1}: $\beta=1$. In this case, for $0<\alpha <1$, 
       \begin{align*}
       \alpha \int_0^{1/(1+\alpha^2)} \lambda^{-1/2}(1-\lambda)^{-1/2}\d \lambda \asymp  \alpha
       \end{align*}
       and 
       \begin{align*}
       \int_{1/(1+\alpha^2)}^1 \lambda^{-1} \d \lambda =\log (1+\alpha^2)\asymp  \alpha^{2}.
       \end{align*}
       On the other hand, for $\alpha \geq 1$, 
       \begin{align*}
       \alpha  \int_0^{1/(1+\alpha^2)} \lambda^{-1/2}(1-\lambda)^{-1/2}\d \lambda 
       \asymp \alpha \int_0^{1/(1+\alpha^2)} \lambda^{-1/2}\d \lambda
       \asymp 1,
       \end{align*}
       and 
        \begin{align*}
       \int_{1/(1+\alpha^2)}^1 \lambda^{-1} \d \lambda =\log (1+\alpha^2)\asymp  \log \alpha.
       \end{align*}
       This proves the first part of \eqref{intmin}.
       
       \medskip
       \noindent 
       \textbf{Case 2}: $1<\beta<2$.. In this case, for $0<\alpha <1$, 
       \begin{align*}
       \alpha^{\beta}  \int_0^{1/(1+\alpha^2)} \lambda^{-\beta/2}(1-\lambda)^{-\beta/2}\d \lambda \asymp  \alpha^{\beta}
       \end{align*}
       and 
       \begin{align*}
       \int_{1/(1+\alpha^2)}^1 \lambda^{-\beta} \d \lambda = \frac{1}{\beta-1}((1+\alpha^2)^{\beta-1} -1)\asymp  \alpha^{2}.
       \end{align*}
       On the other hand, for $\alpha \geq 1$, 
       \begin{align*}
       \alpha^{\beta}  \int_0^{1/(1+\alpha^2)} \lambda^{-\beta/2}(1-\lambda)^{-\beta/2}\d \lambda 
       \asymp \alpha^{\beta}  \int_0^{1/(1+\alpha^2)} \lambda^{-\beta/2}\d \lambda
       \asymp  \alpha^{2\beta-2}.
       \end{align*}
       This proves the second part of \eqref{intmin} and hence completes the proof. 
\end{proof}

\begin{lemma}\label{beta>1}
          Fix $1<\beta< 2$. Then   
          \begin{align}
          \label{1<beta}
         & \sup_{N\geq \e}
           \int_{ [0, N^2]^2}\d s_1\d s_2 \,  
        \left[\left[ (s_1(1-s_1/N^2))^{-\beta/2}\right]\wedge s_1^{-\beta} \right] \left[\left[ (s_2(1-s_2/N^2))^{-\beta/2}\right]\wedge s_2^{-\beta} \right]\nonumber \\
     & \qquad \times 
     \left(\frac{1}{2s_1}+ \frac{1}{2s_2}\right)^{\beta-1}
          \int_0^{1}\d \lambda\,  \lambda^{-\frac{\beta}{2}}       
      \left( \left[\left(\frac{1}{2s_1}+ \frac{1}{2s_2}\right)^{-\beta/2} \left(1-\lambda\right)^{-\frac{\beta}{2}}\right]
          \wedge \lambda^{-\frac{\beta}{2}}  \right) < \infty.
      \end{align}
\end{lemma}
\begin{proof}
          Applying Lemma \ref{smallbig} (second part) with $\alpha = \left(\frac{1}{2s_1}+ \frac{1}{2s_2}\right)^{-1/2} $,
           for all $N\geq \e$, the above integral is bounded above by  a constant times
          \begin{align*}
         &\int_{ [0, N^2]^2}\d s_1\d s_2 \,  
        \left[\left[ (s_1(1-s_1/N^2))^{-\beta/2}\right]\wedge s_1^{-\beta} \right] \left[\left[ (s_2(1-s_2/N^2))^{-\beta/2}\right]\wedge s_2^{-\beta} \right]\nonumber \\
     &\quad  \times \left(\frac{1}{2s_1}+ \frac{1}{2s_2}\right)^{\beta-1}
     \left(
     \bm{1}_{\{\frac{1}{2s_1}+ \frac{1}{2s_2} >1\}}\left(\frac{1}{2s_1}+ \frac{1}{2s_2}\right)^{-\beta/2} 
     + 
          \bm{1}_{\{\frac{1}{2s_1}+ \frac{1}{2s_2} \leq 1\}}\left(\frac{1}{2s_1}+ \frac{1}{2s_2}\right)^{1-\beta} 
     \right). 
          \end{align*}
          
          For $(s_1, s_2)\in [0, 1]^2$, the above  integrand   is bounded above  by
          \begin{align*}
          s_1^{-\beta/2}(1-1/N^2)^{-\beta/2}s_2^{-\beta/2}(1-1/N^2)^{-\beta/2}
          \left(\frac{1}{2s_1}+ \frac{1}{2s_2}\right)^{\beta/2-1},
          \end{align*}
         whence for all $N\geq \e$ the integral over  $[0, 1]^2$ is bounded  above  by
          \begin{align*}
          &(1-1/\e^2)^{-\beta}\int_{[0, 1]^2}
          s_1^{-\beta/2}s_2^{-\beta/2}
          \left(\frac{1}{2s_1}+ \frac{1}{2s_2}\right)^{\beta/2-1} \d s_1\d s_2\\
          & \quad \lesssim \int_{0}^{1}s_2^{-\beta/2} \d s_2 \int_0^{1}s_1^{-\beta/2}s_1^{1-\beta/2}\d s_1 <\infty. 
          \end{align*}
          
          Moreover, for $(s_1, s_2)\in [0, 1]\times (1, N^2]$, the integrand is bounded above   by a constant times
          \begin{align*}
          & s_1^{-\beta/2}(1-1/N^2)^{-\beta/2}s_2^{-\beta}
          \left[\left(\frac{1}{2s_1}+ \frac{1}{2s_2}\right)^{\beta/2-1} + 1\right]\\
          &\quad \lesssim (1-1/\e^2)^{-\beta/2}\left(s_1^{1-\beta}s_2^{-\beta} + s_1^{-\beta/2}s_2^{-\beta}\right),
          \end{align*}
          whence for all $N\geq \e$ the integral over  $[0, 1]\times (1, N^2]$ is bounded  above by a constant times
          \begin{align*}
          & \int_0^{1}s_1^{1-\beta}\d s_1\int_{1}^{\infty}s_2^{-\beta} \d s_2 
          + \int_0^{1}s_1^{-\beta/2}\d s_1\int_{1}^{\infty}s_2^{-\beta} \d s_2 <\infty. 
          \end{align*}
          Similarly, the integral over  $(1, N^2] \times [0, 1]$ is also finite uniformly for $N\geq \e$.

          Furthermore, for  $(s_1, s_2)\in (1, N^2]^2$, the integrand is bounded above by 
          $s_1^{-\beta}s_2^{-\beta}$, which implies that the  integral over 
           $(1, N^2]^2$ is also finite uniformly for $N\geq \e$.
           The proof is complete. 
\end{proof}

\begin{lemma}\label{beta=1}
          There exists a constant $C>0$ such that for all $N\geq \e$ 
          \begin{align}
          \label{1=beta}
         &
           \int_{ [0, N^2]^2}\d s_1\d s_2 \,  
        \left[\left[ (s_1(1-s_1/N^2))^{-1/2}\right]\wedge s_1^{-1} \right] \left[\left[ (s_2(1-s_2/N^2))^{-1/2}\right]\wedge
         s_2^{-1} \right]\nonumber \\
     & \hskip2in \times 
              \int_0^{1}\d \lambda\,  \lambda^{-\frac{1}{2}}       
      \left( \left[\left(\frac{1}{2s_1}+ \frac{1}{2s_2}\right)^{-1/2} \left(1-\lambda\right)^{-\frac{1}{2}}\right]
          \wedge \lambda^{-\frac{1}{2}}  \right) \nonumber \\
          & \quad \leq C\, (\log N)^3.
      \end{align}
\end{lemma}
\begin{proof}
          The proof is similar to that of Lemma \ref{beta>1}.

          Applying Lemma \ref{smallbig} (first part) with $\alpha = \left(\frac{1}{2s_1}+ \frac{1}{2s_2}\right)^{-1/2} $,
           for all $N\geq \e$, the above integral is bounded above by  a constant times
          \begin{align*}
         &\int_{ [0, N^2]^2}\d s_1\d s_2 \,  
        \left[\left[ (s_1(1-s_1/N^2))^{-1/2}\right]\wedge s_1^{-1} \right] \left[\left[ (s_2(1-s_2/N^2))^{-1/2}\right]\wedge s_2^{-1} \right]\nonumber \\
     &\quad \qquad \times
     \left(
     \bm{1}_{\{\frac{1}{2s_1}+ \frac{1}{2s_2} >1\}}\left(\frac{1}{2s_1}+ \frac{1}{2s_2}\right)^{-1/2} 
     + 
          \bm{1}_{\{\frac{1}{2s_1}+ \frac{1}{2s_2} \leq 1\}}\log\left[ \left(\frac{1}{2s_1}+ \frac{1}{2s_2}\right)^{-1/2} \right]
     \right). 
          \end{align*}
          
          For $(s_1, s_2)\in [0, 1]^2$, the above integrand  is bounded above  by
          \begin{align*}
          s_1^{-1/2}(1-1/N^2)^{-1/2}s_2^{-1/2}(1-1/N^2)^{-1/2}
          \left(\frac{1}{2s_1}+ \frac{1}{2s_2}\right)^{-1/2},
          \end{align*}
          whence for all $N\geq \e$ the integral over  $[0, 1]^2$ is bounded  above by a constant times
          \begin{align*}
          &(1-1/\e^2)^{-1}\int_{[0, 1]^2}
          s_1^{-1/2}s_2^{-1/2}
          \left(\frac{1}{2s_1}+ \frac{1}{2s_2}\right)^{-1/2} \d s_1\d s_2\\
          & \quad \lesssim \int_{0}^{1}s_2^{-1/2} \d s_2 \int_0^{1}s_1^{-1/2}s_1^{1/2}\d s_1 <\infty. 
          \end{align*}
          
          Moreover, for $(s_1, s_2)\in [0, 1]\times (1, N^2]$, the integrand is bounded  above by
          \begin{align*}
          & s_1^{-1/2}(1-1/N^2)^{-1/2}s_2^{-1}
          \left[\left(\frac{1}{2s_1}+ \frac{1}{2s_2}\right)^{-1/2} + 
          \log\left[ \left(\frac{1}{2s_1}+ \frac{1}{2s_2}\right)^{-1/2} \right]\right]\\
          &\quad \lesssim (1-1/\e^2)^{-1/2}\left(s_2^{-1} + s_1^{-1/2}s_2^{-1}\log (2s_2)\right),
          \end{align*}
          whence for all $N\geq \e$ the integral over  $[0, 1]\times (1, N^2]$ is bounded above  by a constant times 
          \begin{align*}
          & \int_{1}^{N^2}s_2^{-1} \d s_2 
          + \int_0^{1}s_1^{-1/2}\d s_1\int_{1}^{N^2}\frac{\log(2s_2)}{s_2} \d s_2 \asymp (\log N)^2. 
          \end{align*}
          Similarly, the integral over  $(1, N^2] \times [0, 1]$ is also bounded  above  by $C\, (\log N)^2$ for $N\geq \e$.

          Furthermore, the integral over $ (1, N^2]^2$  is bounded  above by 
          \begin{align*}
          & \int_{(1, N^2]^2} \d s_1\d s_2 \, s_1^{-1}s_2^{-1}\log\left[ \left(\frac{1}{2s_1}+ \frac{1}{2s_2}\right)^{-1/2} \right]\\
          &\quad =2 \int_1^{N^2} \d s_1\int_1^{s_1}\d s_2 \, s_1^{-1}s_2^{-1}
          \log\left[ \left(\frac{1}{2s_1}+ \frac{1}{2s_2}\right)^{-1/2} \right]\\
          &\quad \leq  \int_1^{N^2} \d s_1\int_1^{s_1}\d s_2 \, s_1^{-1}s_2^{-1}
          \log s_1 =  \int_1^{N^2} \frac{(\log s_1)^2}{s_1}\d s_1 \asymp (\log N)^3.
          \end{align*}
           The proof is complete. 
\end{proof}

\begin{lemma} \label{2beta>1}
          Fix $1<\beta<2$. Then,
          \begin{align}\label{1<beta2}
         \sup_{N\geq \e} \int_0^N\d \alpha\,  \alpha^{5-4\beta}
         \left[\left[ (\alpha^2(1-{\alpha^2} /N^2))^{-\beta/2}\right] \wedge \alpha^{-2\beta}\right]
          \left( \int_{0}^1    \left[\left[\alpha^\beta ( s(1-s))^{-\beta/2}\right] \wedge s^{-\beta}\right]\d s\right)^2
          <\infty.
          \end{align}
\end{lemma}

\begin{proof}
          By Lemma \ref{smallbig}, for all $N\geq \e$, the above integral is bounded above by  a constant times 
          \begin{align*}
          &  \int_0^1\d \alpha\,  \alpha^{5-2\beta}
         \left[\left[ (\alpha^2(1-1/N^2))^{-\beta/2}\right] \wedge \alpha^{-2\beta}\right] + 
         \int_1^{N}\d \alpha\,  \alpha
         \left[\left[ (\alpha^2(1-{\alpha^2}/N^2))^{-\beta/2}\right] \wedge \alpha^{-2\beta}\right]\\
         & \quad \leq (1-{1}/\e^2)^{-\beta/2}\int_0^1 \alpha^{5-3\beta} \d \alpha +         
          \int_1^{\infty}  \alpha^{1-2\beta}\d \alpha < \infty. \qedhere
          \end{align*}
\end{proof}

\begin{lemma}\label{2beta=1}
           There exists a constant $C>0$ such that for all $N\geq \e$,
          \begin{align}\label{beta=12}
          \int_0^N\d \alpha\,  \alpha
           \left[\left[ (\alpha^2(1-{\alpha^2}/N^2))^{-1/2}\right] \wedge \alpha^{-2}\right]  
           \left( \int_{0}^1    \left[\left[\alpha( s(1-s))^{-1/2}\right] \wedge s^{-1}\right]\d s\right)^2 \leq C\,( \log N)^3. 
          \end{align}
\end{lemma}
\begin{proof}
          Thanks to Lemma \ref{smallbig} (first part),  the integral in \eqref{beta=12} is bounded above by a constant times
          \begin{align*}
         (1- {1} /\e^2)^{-1/2} \int_0^1\alpha^{1-1 +2}\d\alpha + \int_1^{N}\alpha^{1-2}(\log \alpha)^2 \d \alpha 
         \asymp (\log N)^3. 
          \end{align*}
\end{proof}

\noindent
{\bf Acknowledgement.} We would like to thank the referees for  their  valuable  and useful comments.


\begin{thebibliography}{999}
\bibitem{ACQ11} Amir, G., Corwin, I. and Quastel, J. (2011).
 Probability distribution of the free energy of the continuum directed random polymer in 1+1 dimensions. \textit{Comm. Pure Appl. Math.} \textbf{64}, 466-537
 
\bibitem{CarlenKree1991}
	Carlen, E. and Kr\'ee, P. (1991).
	$L^p$ estimates on iterated stochastic integrals.
	{\it Ann.\ Probab.}\ {\bf 19}{\it(1)} 354--368.
\bibitem{CD15} Chen, L. and Dalang, C. R. (2015)
Moments and growth indices for the nonlinear stochastic heat equation with rough initial conditions.
{\it Ann.\ Probab.}\ {\bf 43}{\it(6)} 3006--3051.
\bibitem{CHN} Chen, L., Hu, Y.  and Nualart, D. (2016)
Regularity and strict positivity of densities for the nonlinear stochastic heat equation.
arXiv:1611.03909. To appear in {\it Mem. Amer. Math. Soc.}
\bibitem{CH19Comparison}
	Chen, L. and Huang, J. (2019).
	\newblock Comparison principle for stochastic heat equation on $\R^d$.
	\newblock {\it Ann.\ Probab.}\ {\bf 47}{\it(2)} 989-1035.
\bibitem{CKNP} Chen, L. Khoshnevisan, D., Nualart, D. and Pu, F. (2019).
	Spatial ergodicity for SPDEs via Poincar\'e-type inequalities. Preprint available
	at \url{https://arxiv.org/abs/1907.11553}.
\bibitem{CKNP_b} Chen, L., Khoshnevisan, D.,  Nualart, D. and Pu, F. (2019).
	Poincar\'e inequality, and central limit theorems for parabolic stochastic partial
	differential equations.
    Preprint available
	at \url{https://arxiv.org/abs/1912.01482}. 
 \bibitem{CKNP_c} Chen, L., Khoshnevisan, D., Nualart, D. and Pu, F. (2020).
  Spatial ergodicity and central limit theorems for parabolic Anderson model with delta initial condition.
Preprint available
	at \url{https://arxiv.org/abs/2005.10417}. 
 \bibitem{CKNP_d} Chen, L.,  Khoshnevisan, D., Nualart, D. and Pu, F. (2020).
  Central limit theorems for spatial averages of the stochastic heat
	equation via Malliavin-Stein's method. Preprint available
	at \url{https://arxiv.org/abs/2008.02408}.

\bibitem{Dalang1999}
	Dalang, C.R. (1999).
	Extending the martingale measure stochastic integral with
	applications to spatially homogeneous s.p.d.e.'s.
	{\it Electron.\ J. Probab.}\ {\bf 4}{\it (6)}, 29 pp.
\bibitem{Davis1976}
	Davis, B. (1976).
	On the $L^p$ norms of stochastic integrals and other martingales.
	{\it Duke Math.\ J.} {\bf 43}{\it (4)} 697--704.
\bibitem{DNZ2018}
Delgado-Vences, F., Nualart, D. and Zheng, G. (2020).
A Central Limit Theorem for the stochastic wave equation with fractional noise. 
{\it Ann. Inst. Henri Poincaré Probab. Stat.} {\bf 56} {\it 4} 3020--3042.
\bibitem{Doob}
	Doob, J. L. (1990).
	{\it Stochastic Processes.}
	Reprint of the 1953 original.
	John Wiley \&\ Sons, Inc., New York, viii+654.
\bibitem{FoondunKhoshnevisan2013}
	Foondun, M. and Khoshnevisan, D. (2013).
	On the stochastic heat equation with spatially-colored random forcing.
	{\it Trans.\ Amer.\ Math.\ Soc.}\ {\bf 365}{\it (1)} 409--458.

\bibitem{GNZ20}
Guerrero, R.B., Nualart, D. and Zheng, G. (2020).
Averaging 2d stochastic wave equation. 
Preprint available
	at \url{https://arxiv.org/abs/2003.10346}. 
\bibitem{HNV2018}
		Huang, J., Nualart,D., and Viitasaari, L. (2020).
	A central limit theorem for the stochastic heat equation.
	\textit{Stochastic Process. Appl.} \textbf{131}, 7170--7184. 
\bibitem{HNVZ2019}
	Huang, J., Nualart, D.,  Viitasaari, L. and Zheng, G. (2020).
	Gaussian fluctuations for the stochastic heat equation with colored noise.
	{\it Stoch PDE: Anal Comp}\ {\bf 8} 402--421. 
\bibitem{HuNualart09} 
Hu, Y. and Nulart, D. (2009).
  Stochastic heat equation driven by fractional noise and local time.
   \emph{ Probab. Theory Related Fields}, {\bf 143} {\it 1-2} 285--328.
\bibitem{Lyons}
	Lyons, R. (1995).
	Seventy years of Rajchman measures.
	{\it J. Fourier Analysis and Applications.}
	Proceedings of the Conference in Honor of Jean-Pierre Kahane (Orsay, 1993).
	363--377.
\bibitem{Nualart}
	Nualart, D. (2006).
	{\it  The Malliavin Calculus and Related Topics}. Springer, New York.
\bibitem{NN}
     Nualart, D. and Nualart, E.  (2018).
      {\it An introduction to Malliavin calculus}.
      Cambridge University Press, Cambridge, UK.
 \bibitem{NSZ20}
 Nualart, D., Song, X. and Zheng, G. (2020)
 Spatial averages for the Parabolic Anderson model driven by rough noise. 
 Preprint available
	at \url{https://arxiv.org/abs/2010.05905}.
\bibitem{NZ2020}
Nualart, D. and Zheng, G. (2020).
Central limit theorems for stochastic wave equations in dimensions one and two. 
Preprint available
	at \url{https://arxiv.org/abs/2005.13587}. 

\bibitem{Walsh}
	Walsh, J.B. (1986).
	{\it An Introduction to Stochastic Partial Differential Equations.}
	\`Ecole d'\'et\'e de probabilit\'es de Saint-Flour, XIV-1984, 265--439,
	In: {\it Lecture Notes in Math.}\ {\bf 1180}, Springer, Berlin.
\end{thebibliography}
\end{document}